\documentclass[reqno]{amsart}
\input{packages}
\bibliography{sources}

\DeclareMathAlphabet{\mathpzc}{OT1}{pzc}{m}{it}

\makeatletter
\DeclareFieldFormat[article]{title}{{\it #1}}
\DeclareFieldFormat{journaltitle}{{\rm #1}}
\renewbibmacro{in:}{%
  \ifentrytype{article}{}{\printtext{\bibstring{in}\intitlepunct}}}
\DeclareFieldFormat[incollection]{title}{{\it #1}}
\DeclareFieldFormat{journaltitle}{{\rm #1}}

\begin{document}

\def\subsectionautorefname{Section}
\def\subsubsectionautorefname{Section}
\def\sectionautorefname{Section}
\def\equationautorefname~#1\null{(#1)\null}

\newcommand{\mynewtheorem}[4]{
  \if\relax\detokenize{#3}\relax 
    \if\relax\detokenize{#4}\relax 
      \newtheorem{#1}{#2}
    \else
      \newtheorem{#1}{#2}[#4]
    \fi
  \else
    \newaliascnt{#1}{#3}
    \newtheorem{#1}[#1]{#2}
    \aliascntresetthe{#1}
  \fi
  \expandafter\def\csname #1autorefname\endcsname{#2}
}

\mynewtheorem{theorem}{Theorem}{}{section}
\mynewtheorem{lemma}{Lemma}{theorem}{}
\mynewtheorem{lem}{Lemma}{theorem}{}
\mynewtheorem{rem}{Remark}{lemma}{}
\mynewtheorem{prop}{Proposition}{lemma}{}
\mynewtheorem{cor}{Corollary}{lemma}{}
\mynewtheorem{definition}{Definition}{lemma}{}
\mynewtheorem{question}{Question}{lemma}{}
\mynewtheorem{assumption}{Assumption}{lemma}{}
\mynewtheorem{example}{Example}{lemma}{}
\mynewtheorem{exm}{Example}{lemma}{}
\mynewtheorem{rmk}{Remark}{lemma}{}
\mynewtheorem{pb}{Problem}{lemma}{}
\newtheorem*{rmk*}{Remark}
\newtheorem*{pb*}{Problem}


\def\defbb#1{\expandafter\def\csname b#1\endcsname{\mathbb{#1}}}
\def\defcal#1{\expandafter\def\csname c#1\endcsname{\mathcal{#1}}}
\def\deffrak#1{\expandafter\def\csname frak#1\endcsname{\mathfrak{#1}}}
\def\defop#1{\expandafter\def\csname#1\endcsname{\operatorname{#1}}}
\def\defbf#1{\expandafter\def\csname b#1\endcsname{\mathbf{#1}}}

\makeatletter
\def\defcals#1{\@defcals#1\@nil}
\def\@defcals#1{\ifx#1\@nil\else\defcal{#1}\expandafter\@defcals\fi}
\def\deffraks#1{\@deffraks#1\@nil}
\def\@deffraks#1{\ifx#1\@nil\else\deffrak{#1}\expandafter\@deffraks\fi}
\def\defbbs#1{\@defbbs#1\@nil}
\def\@defbbs#1{\ifx#1\@nil\else\defbb{#1}\expandafter\@defbbs\fi}
\def\defbfs#1{\@defbfs#1\@nil}
\def\@defbfs#1{\ifx#1\@nil\else\defbf{#1}\expandafter\@defbfs\fi}
\def\defops#1{\@defops#1,\@nil}
\def\@defops#1,#2\@nil{\if\relax#1\relax\else\defop{#1}\fi\if\relax#2\relax\else\expandafter\@defops#2\@nil\fi}
\makeatother

\defbbs{ZHQCNPALRVWS}
\defcals{ABCDOGFPQMNXYLTRAEHZKCIV}
\deffraks{apijklgmnopqueRCc}
\defops{IVC, PGL,SL,mod,Spec,Re,Gal,Tr,End,GL,Hom,PSL,H,div,rk,Mod,R,T,Tr,Mat,Vol,MV,Res,Hur, vol,Z,diag,Hyp,hyp,hl,ord,Im,ev,U,dev,c,CH,fin,pr,Pic,lcm,ch,td,LG,id,Sym,Aut,Log,tw,irr,discrep,BN,NF,NC,age,hor,lev,ram,NH,av,app,Quad,Stab,Per,Nil,Ker,EG,CEG,PB,Conf,MCG}
\defbfs{uvzwp} 

\def\ep{\varepsilon}
\def\ve{\varepsilon}
\def\abs#1{\lvert#1\rvert}
\def\dd{\mathrm{d}}
\def\WP{\mathrm{WP}}
\def\Aut{\mathrm{Aut}}
\def\pzAut{\mathpzc{Aut}}
\def\inj{\hookrightarrow}
\def\eq{=}

\def\i{\mathrm{i}}
\def\e{\mathrm{e}}
\def\st{\mathrm{st}}
\def\ct{\mathrm{ct}}
\def\rel{\mathrm{rel}}
\def\odd{\mathrm{odd}}
\def\even{\mathrm{even}}

\def\uC{\underline{\bC}}
  
\def\Vrel{\bV^{\mathrm{rel}}}
\def\Wrel{\bW^{\mathrm{rel}}}
\def\twolev{\mathrm{LG_1(B)}}

\def\be{\begin{equation}}   \def\ee{\end{equation}}     
\def\bes{\begin{equation*}}    \def\ees{\end{equation*}}
\def\ba{\be\begin{aligned}} \def\ea{\end{aligned}\ee}   
\def\bas{\bes\begin{aligned}}  \def\eas{\end{aligned}\ees}
\def\={\;=\;}  \def\+{\,+\,} \def\m{\,-\,}


\newcommand*{\proj}{\mathbb{P}}
\newcommand{\IVCst}[1][\mu]{{\mathcal{IVC}}({#1})}
\newcommand{\barmoduli}[1][g]{{\overline{\mathcal M}}_{#1}}
\newcommand{\moduli}[1][g]{{\mathcal M}_{#1}}
\newcommand{\omoduli}[1][g]{{\Omega\mathcal M}_{#1}}
\DeclareDocumentCommand{\qmoduli}{ O{g} O{}}{{\mathrm{Quad}^{#2}_{#1}}}

\newcommand{\modulin}[1][g,n]{{\mathcal M}_{#1}}
\newcommand{\omodulin}[1][g,n]{{\Omega\mathcal M}_{#1}}
\newcommand{\zomoduli}[1][]{{\mathcal H}_{#1}}
\newcommand{\barzomoduli}[1][]{{\overline{\mathcal H}_{#1}}}
\newcommand{\pomoduli}[1][g]{{\proj\Omega\mathcal M}_{#1}}
\newcommand{\pomodulin}[1][g,n]{{\proj\Omega\mathcal M}_{#1}}
\newcommand{\pobarmoduli}[1][g]{{\proj\Omega\overline{\mathcal M}}_{#1}}
\newcommand{\pobarmodulin}[1][g,n]{{\proj\Omega\overline{\mathcal M}}_{#1}}
\newcommand{\potmoduli}[1][g]{\proj\Omega\tilde{\mathcal{M}}_{#1}}
\newcommand{\obarmoduli}[1][g]{{\Omega\overline{\mathcal M}}_{#1}}
\newcommand{\qbarmoduli}[1][g]{{\Omega^2\overline{\mathcal M}}_{#1}}
\newcommand{\obarmodulio}[1][g]{{\Omega\overline{\mathcal M}}_{#1}^{0}}
\newcommand{\otmoduli}[1][g]{\Omega\tilde{\mathcal{M}}_{#1}}
\newcommand{\pom}[1][g]{\proj\Omega{\mathcal M}_{#1}}
\newcommand{\pobarm}[1][g]{\proj\Omega\overline{\mathcal M}_{#1}}
\newcommand{\pobarmn}[1][g,n]{\proj\Omega\overline{\mathcal M}_{#1}}
\newcommand{\princbound}{\partial\mathcal{H}}
\newcommand{\omoduliinc}[2][g,n]{{\Omega\mathcal M}_{#1}^{{\rm inc}}(#2)}
\newcommand{\obarmoduliinc}[2][g,n]{{\Omega\overline{\mathcal M}}_{#1}^{{\rm inc}}(#2)}
\newcommand{\pobarmoduliinc}[2][g,n]{{\proj\Omega\overline{\mathcal M}}_{#1}^{{\rm inc}}(#2)}
\newcommand{\otildemoduliinc}[2][g,n]{{\Omega\widetilde{\mathcal M}}_{#1}^{{\rm inc}}(#2)}
\newcommand{\potildemoduliinc}[2][g,n]{{\proj\Omega\widetilde{\mathcal M}}_{#1}^{{\rm inc}}(#2)}
\newcommand{\omoduliincp}[2][g,\lbrace n \rbrace]{{\Omega\mathcal M}_{#1}^{{\rm inc}}(#2)}
\newcommand{\obarmoduliincp}[2][g,\lbrace n \rbrace]{{\Omega\overline{\mathcal M}}_{#1}^{{\rm inc}}(#2)}
\newcommand{\obarmodulin}[1][g,n]{{\Omega\overline{\mathcal M}}_{#1}}
\newcommand{\LTH}[1][g,n]{{K \overline{\mathcal M}}_{#1}}
\newcommand{\PLS}[1][g,n]{{\bP\Xi \mathcal M}_{#1}}

\DeclareDocumentCommand{\LMS}{ O{\mu} O{g,n} O{}}{\Xi\overline{\mathcal{M}}^{#3}_{#2}(#1)}
\DeclareDocumentCommand{\Romod}{ O{\mu} O{g,n} O{}}{\Omega\mathcal{M}^{#3}_{#2}(#1)}

\newcommand*{\Tw}[1][\Lambda]{\mathrm{Tw}_{#1}}  
\newcommand*{\sTw}[1][\Lambda]{\mathrm{Tw}_{#1}^s}  

\newcommand{\HH}{{\mathbb H}}
\newcommand{\MM}{{\mathbb M}}
\newcommand{\bbC}{{\mathbb C}}

\newcommand\PP{\mathbb P}
\newcommand\C{\mathbb C}
\renewcommand\R{\mathbb R}
\renewcommand\Z{\mathbb Z}
\newcommand\N{\mathbb N}
\newcommand\Q{\mathbb Q}
\renewcommand{\H}{\mathbb{H}}
\newcommand{\bk}{k}

\newcommand{\bfa}{{\bf a}}
\newcommand{\bfb}{{\bf b}}
\newcommand{\bfd}{{\bf d}}
\newcommand{\bfe}{{\bf e}}
\newcommand{\bff}{{\bf f}}
\newcommand{\bfg}{{\bf g}}
\newcommand{\bfh}{{\bf h}}
\newcommand{\bfm}{{\bf m}}
\newcommand{\bfn}{{\bf n}}
\newcommand{\bfp}{{\bf p}}
\newcommand{\bfq}{{\bf q}}
\newcommand{\bft}{{\bf t}}
\newcommand{\bfP}{{\bf P}}
\newcommand{\bfR}{{\bf R}}
\newcommand{\bfU}{{\bf U}}
\newcommand{\bfu}{{\bf u}}
\newcommand{\bfx}{{\bf x}}
\newcommand{\bfz}{{\bf z}}

\newcommand{\bfl}{{\boldsymbol{\ell}}}
\newcommand{\bfmu}{{\boldsymbol{\mu}}}
\newcommand{\bfeta}{{\boldsymbol{\eta}}}
\newcommand{\bftau}{{\boldsymbol{\tau}}}
\newcommand{\bfomega}{{\boldsymbol{\omega}}}
\newcommand{\bfeps}{{\boldsymbol{\varepsilon}}}
\newcommand{\bfsigma}{{\boldsymbol{\sigma}}}
\newcommand{\bfdelta}{{\boldsymbol{\delta}}}
\newcommand{\bfnu}{{\boldsymbol{\nu}}}
\newcommand{\bfrho}{{\boldsymbol{\rho}}}
\newcommand{\bfone}{{\boldsymbol{1}}}


\newcommand\cl{\mathcal}


\newcommand{\calA}{\mathcal A}
\newcommand{\calK}{\mathcal K}
\newcommand{\calD}{\mathcal D}
\newcommand{\calC}{\mathcal C}
\newcommand{\calT}{\mathcal T}
\newcommand{\calB}{\mathcal B}
\newcommand{\calF}{\mathcal F}
\newcommand{\calV}{\mathcal V}
\newcommand{\calP}{\mathcal P}
\newcommand{\calQ}{\mathcal Q}
\newcommand{\calR}{\mathcal R}
\newcommand{\calS}{\mathcal S}
\newcommand{\calH}{\mathcal{H}}
\newcommand{\h}{\calH}

\newcommand{\Loc}{\operatorname{Loc}}
\renewcommand{\Hom}{\operatorname{Hom}}
\newcommand{\Ext}{\operatorname{Ext}}
\newcommand\ext{\operatorname{ext}}
\newcommand{\coker}{\operatorname{coker}}
\newcommand{\RHom}{\operatorname{RHom}}

\newcommand{\spn}{\operatorname{span}}
\newcommand{\Tilt}{\operatorname{Tilt}}

\newcommand{\heart}{\heartsuit}

\newcommand{\sph}{\operatorname{sph}}
\newcommand{\Br}{\operatorname{Br}}
\renewcommand{\Tw}{\operatorname{Tw}}

\newcommand{\torsion}{\mathcal{T}}
\newcommand{\torsionfree}{\mathcal{F}}
\newcommand{\torsionfreeg}{\mathcal{G}}
\newcommand{\torsionfreey}{\mathcal{Y}}
\newcommand{\torsionfreex}{\mathcal{X}}
\newcommand{\tstr}{\mathcal{L}}

\newcommand{\Sim}{\operatorname{Sim}}

\newcommand\CY{CY}
\newcommand\cy{\CY}

\newcommand{\half}{\frac{1}{2}}


\newcommand\DQI{{\mathcal{D}^3_{Q_I}}}
\newcommand\AQ{{\mathcal{A}_Q}}
\newcommand\AQI{{\mathcal{A}_{Q_I}}}
\DeclareDocumentCommand{\DQ}{ O{3} }{{\mathcal{D}^{#1}_Q}}
\DeclareDocumentCommand{\DQQ}{ O{}O{}}{{\mathcal{D}^{#1}_{#2}}}
\DeclareDocumentCommand{\AQQ}{ O{q} }{{\mathcal{A}_{#1}}}
\newcommand{\Exch}{\operatorname{Exch}}

\newcommand\Rep{\operatorname{Rep}}
\newcommand\rep{\operatorname{rep}}
\newcommand{\modules}{\operatorname{mod}}
\renewcommand{\Mod}{\operatorname{Mod}}
\newcommand\pvd{\operatorname{pvd}}
\newcommand\perf{\operatorname{perf}}
\newcommand\per{\operatorname{per}}


\newcommand{\sgn}{\operatorname{sgn}}
\renewcommand{\GL}{\operatorname{GL}}
\renewcommand{\rk}{\operatorname{rank}}
\newcommand{\opL}{\operatorname{L}}
\renewcommand{\Re}{\operatorname{Re}} 
\renewcommand{\Im}{\operatorname{Im}}


\newcommand\stab{\operatorname{Stab}}
\newcommand\gstab{\operatorname{GStab}}
\newcommand{\GStab}{\gstab}
\newcommand{\PGStab}{\bP\mathrm{GStab}}
\newcommand{\MStab}{\operatorname{MStab}}
\newcommand{\PMStab}{\bP\mathrm{MStab}}
\newcommand{\halfplane}{\mathbb{H}}
\newcommand{\chalfplane}{\overline{\mathbb{H}}}


\newcommand{\bra}{\langle}
\newcommand{\ket}{\rangle}
\newcommand{\del}{\partial}
\newcommand{\wh}{\widehat}
\newcommand{\wt}{\widetilde}
\newcommand{\whmu}{\widehat{\mu}}
\newcommand{\whrho}{\widehat{\rho}}
\newcommand{\whLa}{\widehat{\Lambda}}
\newcommand{\ol}{\overline}



\newcommand{\ps}{\mathrm{ps}}  

\newcommand{\tdpm}[1][{\Gamma}]{\mathfrak{W}_{\operatorname{pm}}(#1)}
\newcommand{\tdps}[1][{\Gamma}]{\mathfrak{W}_{\operatorname{ps}}(#1)}
\newcommand{\cal}[1]{\mathcal{#1}}

\newlength{\halfbls}\setlength{\halfbls}{.8\baselineskip}

\newcommand*{\Teichmuller}{Teich\-m\"uller\xspace}

\DeclareDocumentCommand{\MSfun}{ O{\mu} }{\mathbf{MS}({#1})}
\DeclareDocumentCommand{\MSgrp}{ O{\mu} }{\mathcal{MS}({#1})}
\DeclareDocumentCommand{\MScoarse}{ O{\mu} }{\mathrm{MS}({#1})}
\DeclareDocumentCommand{\tMScoarse}{ O{\mu} }{\widetilde{\mathbb{P}\mathrm{MS}}({#1})}

\newcommand{\kmin}{\kappa_{(2g-2)}}
\newcommand{\ktop}{\kappa_{\mu_\Gamma^{\top}}}
\newcommand{\kbot}{\kappa_{\mu_\Gamma^{\bot}}}

\newcommand{\deld}[1]{\frac{\del}{\del {#1} }}
\newcommand{\frap}{\frac{1}{2\pi\ii}}
\newcommand{\Rp}{\mathbb{R}_{>0}}
\newcommand{\Rm}{\R_{<0}}
\newcommand{\TT}{{\mathbb T}}

\def\w{\mathbf{w}}
\def\wtpt{\w_{+2}}
\def\lift{\mathrm{lift}}

\newcommand{\ST}{\operatorname{ST}}  
\newcommand{\Cone}{\operatorname{Cone}}
\newcommand{\Tri}{\Delta}
\newcommand\surf{\mathbf{S}}  
\newcommand\surfo{{\mathbf{S}}_\Tri}  
\newcommand\surfw{\surf^\w}  
\newcommand\sow{\surf_\w}  
\newcommand\subsur{\Sigma}  
\newcommand\colsur{\overline{\surf}_\w}  
\newcommand\surAn{\mathbf{S}_{A_n}}  

\def\Dsan{\calD_3(\surfo)}
\def\Dsow{\calD(\sow)}
\def\Dsub{\calD_3(\subsur)}
\def\Dcol{\calD(\colsur)}
\newcommand{\FQuad}{\operatorname{FQuad}}
\newcommand{\FQuab}{\FQuad^{\bullet}}
\newcommand\Stap{\Stab^\circ} 
\newcommand\Stas{\Stab^\bullet}
\def\Dsan{\calD_3(\surfo)}
\def\Dsow{\calD(\sow)}
\def\Dsub{\calD_3(\subsur)}
\def\Dcol{\calD(\colsur)}

\newcommand{\EGs}{\EG^{s}} 
\newcommand{\SEG}{\operatorname{SEG}} 
\newcommand{\pSEG}{\operatorname{pSEG}} 
\newcommand{\EGp}{\EG^\circ}       
\newcommand{\SEGp}{\SEG^\circ}       
\newcommand{\EGb}{\EG^\bullet}       
\newcommand{\SEGb}{\SEG^\bullet}       
\newcommand{\SEGV}{\SEG_{{^\perp}\calV}} 
\newcommand{\pSEGV}{\pSEG_{{^\perp}\calV}}
\newcommand{\SEGVb}{\SEGb_{{^\perp}\calV}} 
\newcommand{\pSEGVb}{\pSEG^\bullet_{{^\perp}\calV}}

\newcommand{\EGT}{\EG^\circ}
\newcommand{\uEG}{\underline{\EG}} 
\newcommand{\uCEG}{\underline{\CEG}} 
\newcommand\AS{\mathbb{A}}

\newcommand*{\cgeq}{\succcurlyeq}
\newcommand*{\cleq}{\preccurlyeq}


\title[A smooth compactification: the $A_{n}$-case]
      {A smooth compactification of spaces of stability conditions: the case
        of the $A_{n}$-quiver}

\author{Anna Barbieri}
\address{A.B.:  Dipartimento di Informatica - Settore Matematica, 
Universit\`a di Verona, 
Strada Le Grazie 15, 37134 Verona - Italy }
\email{anna.barbieri@univr.it}
\thanks{Research of A.B. was partially supported by the
project \emph{REDCOM: Reducing complexity in algebra, logic, combinatorics}, financed by the programme \emph{Ricerca Scientifica
di Eccellenza 2018} of the Fondazione Cariverona.
}

\author{Martin M\"oller}
\address{M.M.: Institut f\"ur Mathematik, Goethe-Universit\"at Frankfurt,
Robert-Mayer-Str. 6-8,
60325 Frankfurt am Main, Germany}
\email{moeller@math.uni-frankfurt.de}
\thanks{Research of M.M. and J.S.\ is supported
by the DFG-project MO 1884/2-1, and by the Collaborative Research Centre
TRR 326 ``Geometry and Arithmetic of Uniformized Structures''.}

\author{Jeonghoon So}
\address{J.S.: Institut f\"ur Mathematik, Goethe-Universit\"at Frankfurt,
Robert-Mayer-Str. 6-8,
60325 Frankfurt am Main, Germany}
\email{so@math.uni-frankfurt.de}

\begin{abstract}
We propose a notion of multi-scale stability conditions with the goal
of providing a smooth compactification of the quotient of the space
of projectivized Bridgeland stability conditions by the group of
autoequivalence. For the case of the 3CY category associated with the
$A_n$-quiver this goal is achieved
by defining a topology and complex structure that relies on a plumbing
construction.
\par
We compare this compactification to the multi-scale compactification
of quadratic differentials and briefly  indicate why even for the Kronecker
quiver this notion needs refinement to provide a full compactification. 
\end{abstract}
\maketitle
\tableofcontents

\section{Introduction}\label{sec_intro}

Spaces of Bridgeland stability on a triangulated category~$\cD$ have been
introduced in \cite{BrStab}. By definition these spaces $\Stab(\cD)$
are non-compact, in fact they admit a $\bC$-action that allows to rescale
the central charges. The projectivizations $\bP\Stab(\cD) = \bC\backslash
\Stab(\cD)$ are still non-compact, since the ratio of masses of some objects
may go to zero. Recently several partial compactifications have been proposed
(\cite{BDL_Thurston, barbara, KKO, BPPW}), whose merits we compare at the end
of the introduction. Our goal is to provide a generalized notion of stability
conditions that could provide a \emph{smooth} compactification in the sense
of orbifolds of the quotient $\bC \backslash \Stab(\cD)/\pzAut(\cD)$. In this
paper we achieve this goal for the $\CY_3$-categories~$\DQ$ where $Q$ is a
quiver of~$A_n$-type.
\par
Our approach is motivated by the isomorphism of Bridgeland-Smith \cite{BS15}
of $\Stab(\DQ)$ with spaces of quadratic differentials with simple zeros
and  the generalization  of this isomorphism to differentials
with higher order zeros constructed in our previous paper \cite{BMQS}.
Its main result
states that these are isomorphic to spaces of stability conditions
on quotient categories~$\DQ/\cD^3_{Q_I}$ for some subquivers~$Q_I \subset Q$. In both contexts, simple and higher order zeroes, part of the isomorphism is given by identifying central charges of (simple and stable) objects with the distance between zeroes with respect to the metric induced by a quadratic differential. A first and naive idea would be to interpret collision of zeroes of a quadratic differential as the vanishing of central charges. To get a smooth compactification this idea has to be refined.
\par
Our approach is also motivated by the smooth compactification \cite{LMS} of
strata of differentials by multi-scale differentials.
From there we take the idea that if central charges go to zero we
\lq\lq zoom in\rq\rq, i.e., we rescale and get another non-zero \lq\lq central charge\rq\rq\ on a
subcategory. This \lq\lq central charge\rq\rq\ in turn might vanish on some
simple objects and forces us to rescale again, thus arriving at a filtration
of subcategories. From multi-scale differentials we also borrow
the observation that the result of the rescaling process is only well-defined
up to multiplication by a common scalar factor, resulting in the definition
of equivalence below.
\par
Combining these ideas we can now paraphrase our main
notion, see Definition~\ref{def:mstab} for the precise formulation.
A \emph{non-split multi-scale stability condition $(\cA_\bullet, Z_\bullet)$}
on a triangulated category consists of
\begin{itemize}[nosep]
\item a \emph{multi-scale heart}
$\cA_\bullet = (\cA_i)$, i.e., a collection  $\cA_L \subset \cdots
\cA_1 \subset \cA_0$ of abelian categories, and 
\item a \emph{multi-scale central charge}, i.e., a collection
$Z_\bullet = (Z_i)_{i=0}^L$ of non-zero $\Z$-linear maps on the
  Grothendieck groups $Z_i: K(\cA_i) \to \bC$, where $Z_i$ factors through
  $\Ker(Z_{i-1})$, 
\end{itemize}
with the following properties. First, the categories $\cA_i$ are hearts
of the \lq\lq vanishing\rq\rq\ triangulated subcategories $\cV^Z_i \subset \cD$
generated by objects $E \in \cA_{i-1}$ such that the central charge
of the previous filtration step vanishes, i.e.\ $Z_{i-1}(E) = 0$.
Second, the central charges $Z_i$ map simples in $\cA_i \setminus \cA_{i+1}$
to the semi-closed upper half-plane~$\chalfplane$. (This implies that
$\cV^Z_{i+1} \cap \cA_i$ is a Serre subcategory of $\cA_{i}$.)
Third, the induced quotient heart with quotient central charge $(\ol{\cA}_i,
\ol{Z}_i)$ is a stability condition in the usual sense of \cite{BrStab}
on the quotient category $\cV^Z_i/\cV^Z_{i+1}$.
We say that two non-split multi-scale stability conditions are
\emph{equivalent} if the induced quotients $(\ol{\cA}_i, \ol{Z}_i)$
are projectively equivalent for all~$i \geq 1$. We denote by
$\MStab(\cD)$ the set of equivalence classes of those multi-scale
stability conditions and add a circle (e.g.\ $\MStab^\circ(\cD)$) to
denote a specific connected component or a set of reachable stability
conditions. 
\par
In this paper we only consider multi-scale stability conditions
that are \emph{non-split and thus drop this adjective} from now on.
In Section~\ref{subsec:obstructions} below we will explain why this
notion needs refinements to provide compactifications for more general
categories~$\cD$, even for other $\CY_3$ quiver categories~$\DQ$.
\par
We recall that for $\DQ$ of type $A_n$ the group of autoequivalences
$\pzAut^\circ(\cD^3_{A_n})$ preserving a connected component of $\Stab(\cD^3_{A_n})$
(modulo those acting trivially)
is an extension of~$\bZ/(n+3)\bZ$ by the spherical twist group~$\ST(A_n)$,
which is isomorphic to a braid group, see Section~\ref{sec:STandBraid}.
\par
\begin{theorem} \label{intro:main}
The quotient $\MStab^\circ(\cD^3_{A_n})/\pzAut^\circ(\cD^3_{A_n})$ of the space of
multi-scale stability  conditions has a structure
of a complex orbifold. The projectivization of this orbifold $\bC \backslash
\MStab^\circ(\cD^3_{A_n})/\pzAut^\circ(\cD^3_{A_n})$ is a compactification of the
space of projectivized stability conditions up to auto\-equivalence
$\bC \backslash \Stab^\circ(\cD^3_{A_n})/ \pzAut^\circ(\cD^3_{A_n})$
\end{theorem}
\par
As a complex orbifold, the space $\bC \backslash \Stab^\circ(\cD^3_{A_n})/
\pzAut^\circ(\cD^3_{A_n})$ is simply the moduli space of curves $\cM_{0,n+2}$.
The compactification $\bC \backslash \MStab^\circ(\cD^3_{A_n})
/\pzAut^\circ(\cD^3_{A_n})$
is however not equal to the Deligne-Mumford compactification
$\barmoduli[0,n+2]$. It is rather a blowup of the latter, as we explain
in Section~\ref{sec:BCGGMforAn}.
%

\subsection{Techniques}\label{subsec:techniques}
One important technique is the \emph{plumbing} of a multi-scale stability
condition, depending on complex numbers $\tau_i$ for $i=0,\ldots,L$,  that
builds a usual stability condition. If $\tau_i \in i\bR_{-}$ is purely imaginary
for all~$i$ the result is just the top level heart $\calA_0$ together with
a central charge that is a rescaled linear combination of the~$Z_i$. One should
envision that the size of $Z_i$ is $e^{- \pi i \tau_i}$, thus very small for~$\tau_i$
close to~$-i\infty$ and this is continuously completed by declaring
that $\tau_i = -i\infty$ means no plumbing at all. The process of plumbing
becomes interesting for $\tau_i$ not purely imaginary. This involves
rotating $\calA_i$. The higher level hearts $\calA_{i-1}$ etc.\ then have to
be modified to still contain the rotated heart while still providing
the same quotient heart. This modification of the representative however causes
that the plumbing action of $(\tau_1,\ldots,\tau_L) \in -\bH^L$ is not the
action of a semigroup: the semigroup addition and the action only almost
commute, with an error that goes to zero as $\tau_i \to -i\infty$.
\par
In this way, we give $\MStab^\circ(\cD^3_{A_n})$ a topology by declaring
neighborhoods of a multi-scale stability conditions to be plumbings with
$t_i := e^{-\pi i \tau_i}$ small composed with a small deformation of the stability
condition. However, this space is not locally compact. In fact, for $n=2$
the space is isomorphic to $\HH \cup \bP^1(\bQ)$ with the horoball topology,
as we will explain in Section~\ref{sec:A2revisited}.
\par
The complex orbifold structure on the quotient $\MStab^\circ(\cD^3_{A_n})/
\pzAut^\circ(\cD^3_{A_n})$ is locally given by the functions~$t_i$ together with
the central charges~$Z_i$. This statement requires to control the stabilizer of
a neighborhood of the multi-scale stability condition. We show that this
stabilizer contains a finite index subgroup isomorphic to~$\bZ^L$.
\par
For compactness of $\bC \backslash \Stab^\circ(\cD^3_{A_n})/\pzAut^\circ(\cD^3_{A_n})$
the obvious idea is to normalize in a given sequence of multi-scale stability
conditions the mass of the largest simple to be one, and then define an
order on the set of simples corresponding to the speed in which their central
charges go to zero. The level sets for this order will then correspond
to the index set of the limiting multi-scale stability condition. The
challenge for this idea arises, if the central charge
of a stable but non-simple object tends to zero despite the normalization
while the central charge of its simple factors do not. This forces the central
charge of some simple object to tend to the positive real axis.
\par
Consider for example the sequence $\sigma_n = (\cA, Z_n)$ of stability
conditions on $\cD^3_{A_2}$, all supported on a fixed heart~$\cA$ and
with
\be \label{eq:Z}
Z_n(S_1) = -1 + i/n, \qquad Z_n(S_2) = 1 + i/n\,.
\ee
see Figure~\ref{fig:A2rotate} for the picture of the corresponding quadratic
differential. 
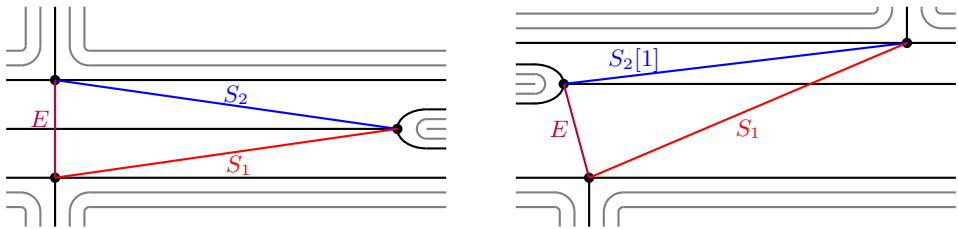
\begin{figure}[h]
  \begin{tikzpicture}[scale=.65]
    \hspace{-0.3cm}
		\begin{scope}[shift={(0,0cm)}]
		
		\coordinate (A) at (1,1);
		\fill[black] (A) circle (3pt);
		
		\coordinate (B) at (8,2);
		\fill[black] (B) circle (3pt);
		
		\coordinate (C) at (1,3);
		\fill[black] (C) circle (3pt);
		
		
		\draw[>=stealth,thick] (1,1) -- (1,0);
		
		\draw[>=stealth,thick] (1,1) -- (0,1);
		
		\draw[>=stealth,thick] (1,1) -- (9,1);
		
		\draw[>=stealth,thick] (1,3) -- (1,4.5);
		
		\draw[>=stealth,thick] (1,3) -- (0,3);
		
		\draw[>=stealth,thick] (1,3) -- (9,3);
		
		\draw[>=stealth,thick] (8,2) -- (0,2);
		
		\draw[thick] (8,2) arc [start angle=180, end angle=270, x radius=0.6cm, y radius =0.4cm];
		
		\draw[thick] (8,2) arc [start angle=180, end angle=90, x radius=0.6cm, y radius =0.4cm];
		
		\draw[>=stealth,thick] (8.6,2.4) -- (9,2.4);
		
		\draw[>=stealth,thick] (8.6,1.6) -- (9,1.6);
		
		
		
		\draw[gray, thick] (0,3.3) -- (0.3,3.3);
		\draw[gray, thick] (0.3,3.3) arc [start angle=270, end angle=360, x radius=0.4cm, y radius =0.4cm];
		\draw[gray, thick] (0.7,3.7) -- (0.7,4.5);
		
		\draw[gray, thick] (0,3.6) -- (0.3,3.6);
		\draw[gray, thick] (0.3,3.6) arc [start angle=270, end angle=360, x radius=0.1cm, y radius =0.1cm];
		\draw[gray, thick] (0.4,3.7) -- (0.4,4.5);

		
		\draw[gray, thick] (0,0.7) -- (0.3,0.7);
		\draw[gray, thick] (0.3,0.7) arc [start angle=90, end angle=0, x radius=0.4cm, y radius =0.4cm];
		\draw[gray, thick] (0.7,0) -- (0.7,0.3);
		
		\draw[gray, thick] (0,0.4) -- (0.3,0.4);
		\draw[gray, thick] (0.3,0.4) arc [start angle=90, end angle=0, x radius=0.1cm, y radius =0.1cm];
		\draw[gray, thick] (0.4,0) -- (0.4,0.3);
		
		
		\draw[gray, thick] (1.7,3.3) -- (9,3.3);
		\draw[gray, thick] (1.3,3.7) arc [start angle=180, end angle=270, x radius=0.4cm, y radius =0.4cm];
		\draw[gray, thick] (1.3,3.7) -- (1.3,4.5);
		
		\draw[gray, thick] (1.7,3.6) -- (9,3.6);
		\draw[gray, thick] (1.7,3.6) arc [start angle=270, end angle=180, x radius=0.1cm, y radius =0.1cm];
		\draw[gray, thick] (1.6,3.7) -- (1.6,4.5);
		
		
		\draw[gray, thick] (1.3,0) -- (1.3,0.3);
		\draw[gray, thick] (1.3,0.3) arc [start angle=180, end angle=90, x radius=0.4cm, y radius =0.4cm];
		\draw[gray, thick] (1.7,0.7) -- (9,0.7);
		
		\draw[gray, thick] (1.6,0) -- (1.6,0.3);
		\draw[gray, thick] (1.6,0.3) arc [start angle=180, end angle=90, x radius=0.1cm, y radius =0.1cm];
		\draw[gray, thick] (1.7,0.4) -- (9,0.4);
		
		\draw[gray, thick] (8.6,2.2) arc [start angle=90, end angle=270, x radius=0.2cm, y radius =0.2cm];
		
		\draw[>=stealth,gray, thick] (8.6,2.2) -- (9,2.2);
		
		\draw[>=stealth,gray, thick] (8.6,1.8) -- (9,1.8);
		
		\draw[>=stealth,gray, thick] (8.6,2) -- (9,2);
		
		
		\draw[blue, thick] (C) -- (B);
		
		\node[right, blue] at (4.25,2.7) {\small $S_2$}; 
		
		\draw[red, thick] (A) -- (B);
		
		\node[right, red] at (4.3,1.25) {\small $S_1$}; 
		
		\draw[purple, thick] (A) -- (C);
		
		\node[right, purple] at (0.3,2.2) {\small $E$}; 
		
		\end{scope}
	         \hspace{0.6cm}
		\begin{scope}[shift={(10cm,0cm)}]

			\coordinate (A) at (1,1);
			\fill[black] (A) circle (3pt);
			
			\coordinate (B) at (7.5,3.76);
			\fill[black] (B) circle (3pt);
			
			\coordinate (C) at (0.48,2.92);
			\fill[black] (C) circle (3pt);

			
			\draw[>=stealth,thick] (1,1) -- (1,0);
			
			\draw[>=stealth,thick] (1,1) -- (-0.5,1);
			
			\draw[>=stealth,thick] (1,1) -- (8.5,1);
			
			\draw[>=stealth,thick] (0.48,2.92) -- (8.5,2.92);
			
			\draw[>=stealth,thick] (-0.5,3.76) -- (8.5,3.76);
			
			\draw[>=stealth,thick] (7.5,3.76) -- (7.5,4.5);
			
			\draw[thick] (0.48,2.92) arc [start angle=360, end angle=270, x radius=0.6cm, y radius =0.4cm];
			
			\draw[thick] (0.48,2.92) arc [start angle=360, end angle=450, x radius=0.6cm, y radius =0.4cm];

			\draw[>=stealth,thick] (-0.1,3.32) -- (-0.5,3.32);
			
			\draw[>=stealth,thick] (-0.1,2.52) -- (-0.5,2.52);
			
			
			
			\draw[gray, thick] (-0.5,4.06) -- (6.8,4.06);
			\draw[gray, thick] (6.8,4.06) arc [start angle=270, end angle=360, x radius=0.4cm, y radius =0.4cm];
			\draw[gray, thick] (7.2,4.46) -- (7.2,4.5);
			
			\draw[gray, thick] (-0.5,4.36) -- (6.8,4.36);
			\draw[gray, thick] (6.8,4.36) arc [start angle=270, end angle=360, x radius=0.1cm, y radius =0.1cm];
			\draw[gray, thick] (6.9,4.46) -- (6.9,4.5);
			

			\draw[gray, thick] (8.2,4.06) -- (8.5,4.06);
			\draw[gray, thick] (8.2,4.06) arc [start angle=270, end angle=180, x radius=0.4cm, y radius =0.4cm];
			\draw[gray, thick] (7.8,4.46) -- (7.8,4.5);
			
			\draw[gray, thick] (8.2,4.36) -- (8.5,4.36);
			\draw[gray, thick] (8.2,4.36) arc [start angle=270, end angle=180, x radius=0.1cm, y radius =0.1cm];
			\draw[gray, thick] (8.1,4.46) -- (8.1,4.5);
			
			
			\draw[gray, thick] (-0.5,0.7) -- (0.3,0.7);
			\draw[gray, thick] (0.3,0.7) arc [start angle=90, end angle=0, x radius=0.4cm, y radius =0.4cm];
			\draw[gray, thick] (0.7,0) -- (0.7,0.3);
			
			\draw[gray, thick] (-0.5,0.4) -- (0.3,0.4);
			\draw[gray, thick] (0.3,0.4) arc [start angle=90, end angle=0, x radius=0.1cm, y radius =0.1cm];
			\draw[gray, thick] (0.4,0) -- (0.4,0.3);
			
			
			\draw[gray, thick] (1.3,0) -- (1.3,0.3);
			\draw[gray, thick] (1.3,0.3) arc [start angle=180, end angle=90, x radius=0.4cm, y radius =0.4cm];
			\draw[gray, thick] (1.7,0.7) -- (8.5,0.7);
			
			\draw[gray, thick] (1.6,0) -- (1.6,0.3);
			\draw[gray, thick] (1.6,0.3) arc [start angle=180, end angle=90, x radius=0.1cm, y radius =0.1cm];
			\draw[gray, thick] (1.7,0.4) -- (8.5,0.4);
			
			\draw[gray, thick] (-0.1,3.12) arc [start angle=90, end angle=-90, x radius=0.2cm, y radius =0.2cm];

			\draw[>=stealth,thick, gray] (-0.1,3.12) -- (-0.5,3.12);
			
			\draw[>=stealth,gray, thick] (-0.1,2.72) -- (-0.5,2.72);
			
			\draw[>=stealth,gray, thick] (-0.1,2.92) -- (-0.5,2.92);
			
			
			\draw[blue, thick] (C) -- (B);
			
			\node[right, blue] at (1.2,3.4) {\small $S_2[1]$}; 
			
			\draw[red, thick] (A) -- (B);
			\node[right, red] at (3.8,2) {\small $S_1$}; 
			
			\draw[purple, thick] (A) -- (C);
			\node[right, purple] at (0,2) {\small $E$}; 
		\end{scope}
	\end{tikzpicture}
  \caption{Quadratic differential illustrating a degenerating sequence
    in $\Stab^\circ(\cD^3_{A_2})$ and a rotated situation}
        \label{fig:A2rotate}
\end{figure}

In the limit $n\to \infty$, the central charge vanishes precisely on
the subcategory generated by the non-trivial extension~$E$ of~$S_1$
by~$S_2$. Since $E$ is not simple, it does not define 
a non-trivial Serre subcategory of~$\calA$, contradicting a consequence
of our definition of multi-scale stability condition.
\par
The solution to find the limiting object is to rotate the sequence
by $\lambda_n$ so that $Z_{\lambda_n \sigma_n}(S_2) \in \HH^-$, see
Figure~\ref{fig:A2rotate} on the right. The heart~$\cA$ is replaced by
$\calA_0:=\mu_{S_2}\cl A$, the tilt one would usually perform also
inside $\Stab(\cD^3_{A_2})$ if the central charge of the simple~$S_2$ approaches
the positive real axis. Now~$E$ is simple and the vanishing category $\calV^Z_1$
generated by~$E$ has the property that $\calV^Z_1 \cap \calA_0$ is Serre in~$\calA_0$.
The limiting multi-scale stability condition consists of the
filtration $\calA_0\supset\calA_1=\bra E\ket$ together with $Z_0(S_2[1]) = -1$
and $Z_0(E)=0$ as well as $Z_1(E)$ arbitrary non-zero in view of the
notion of equivalence.
\par

\subsection{Obstructions to generalization} \label{subsec:obstructions}

Continuing the idea of proof for compactness we consider one of the simplest
cases beyond $A_n$-type quivers, stability conditions on the $\CY_3$-category of
the Kronecker quiver, or, in the language of quadratic differentials (see
\cite[Example~12.5]{BS15}), the stratum $\calQ(-3,-3,1,1)$ with two triple
poles and two simple zeros. We again consider a situation where the central charge
of a stable but non-simple object tends to zero within the normalization that
the mass of the largest simple is approximately one. Now using a small rotation does
not seem to help.  The compactification of strata of differentials (\cite{LMS,kdiff})
that we recall in Section~\ref{sec:BCGGMforAn} hints to the reason for this problem.
\par
The boundary strata of the compactification are encoded by level graphs,
whose vertices correspond to components of stable curves and where
a vertex~$v_1$ is above a vertex~$v_2$ if the differential tends to zero
on~$v_2$ more quickly than on~$v_1$. In terms of multi-scale stability
conditions we find the same level structure (given by index of~$\calA_i$)
and components (given by the components of the ext-quiver on the simples
in~$\calA_i$). However, for differentials we allow for \emph{horizontal
degenerations}, i.e., edges between vertices on the same level. In this
degeneration of the Kronecker quiver alluded to above (with central charge
as in~\eqref{eq:Z}) we should normalize the sequence to keep the length
of the \lq\lq short\rq\rq\ stable object~$E$ (i.e.\ the extension of~$S_1$ by~$S_2$
or geometrically the length of the core curve of the cylinder) constant.
This happens at the expense of letting the mass of both simples go to
infinity. In the geometric picture, the surface splits into two subsurfaces
with quadratic differentials of type $(-3,-2,1)$. It would be interesting
to enlarge the concept of non-split multi-scale differentials so as to
include this \lq\lq splitting\rq\rq\ of the category.
\par
It seems quite plausible that the current definition of (non-split) multi-scale
stability condition provides a partial compactification of $\bC \backslash
\Stab^\circ(\cD^3_Q)/ \pzAut^\circ(\cD^3_Q)$ to a complex orbifold for
general quiver categories (or whenever $\Stab^\circ(\cD)$ is of tame type).
This requires to overcome several technical problems that we highlight
along with the definition of the topology in Section~\ref{sec:multiscale}.
Currently we rely on the fact that hearts in $\calD_{A_n}^3$ have finitely many
indecomposables.

\subsection{Comparison to other compactifications} \label{subsec:other}

We are aware of four other papers aiming to compactify spaces of
stability conditions. Bolognese \cite{barbara} uses a metric completion
to give a partial compactification. The Thurston-type compactifications
of Bapat, Deopukar and Licata \cite{BDL_Thurston} and Kikuta-Koseki-Ouchi
\cite{KKO} use the tuple of all masses to get a map from the space
of projectivized stability conditions to some projective space and
take the closure there. The space of lax stability conditions of
Broomhead, Pauksztello, Ploog, Woolf \cite{BPPW} allows some of the masses
of semistable objects to be zero, but requires a modified support property
and zero being an isolated point of the set of all masses.
\par
Common to all these approaches is that they aim to (partially) compactify
the space $\bP \Stab(\calD)$ of projectivized stability conditions whereas
we compactify its quotient by $\pzAut^\circ(\cD^3_{A_n})$. Moreover in all
these four papers the boundary or boundary strata are real codimension one,
whereas in our approach the boundary has complex codimension one, since
we construct a complex orbifold. We mention that \cite[Section~15]{LMS}
proposes a real-oriented blowup of the complex orbifold, thus a real
manifold with corners, to which the $\GL_2^+(\bR)$-action extends.
This real-oriented blowup construction can certainly also be incorporated
into a modified definition of multi-scale stability conditions. In this
real blowup there are real codimension one boundary strata, which parametrize
stability conditions on a quotient category by a rank one subcategory
together -- this is effect of the real blowup -- with the phase of the
simple with vanishing mass. This seems to agree with the codimension
one boundary strata of \cite{BPPW}. Since both approaches, ours and
\cite{BPPW}, use stability conditions on quotient categories and the
difficulties often stem from lifting problems, it would be interesting
to compare or combine them. However this does not seem to solve the
problem of getting a compact space in a more general setting, as we
see no subsitute for the missing \lq\lq horizontal degenerations\rq\rq .

\subsection*{Acknowledgments} We thank Dawei Chen and Yu Qiu
for inspiring discussions.
\par

\section{Background and notation}\label{sec:notation}

In this section we recall basic material about stability conditions,
quivers and their $\CY_3$-categories $\calD_Q^3$. References for this
include \cite{tiltingbook, gelfandmanin,f-pervers, BrStab,bridgeland_survey,
  neeman, DWZ,keller11}.
\par

\subsection{Notation and fundamental assumptions} \label{sec:fundamental}
We fix some notation that will be used throughout. Let $\bk$ denote an algebraically
closed field, and  any category is additive, $\bk$-linear, and essentially small. 
We deal with finite-dimensional abelian and  triangulated categories of modules
(resp., dg modules) over a finite-dimensional algebra (resp., a dg algebra). Whenever
we define a subcategory, we mean that there is a fully faithful functor 
that we assume to be the embedding. 
\par
Given subcategories $\calA_1,\calA_2$ of an abelian or a triangulated
category~$\calC$, and a set of objects $\cl B$, we define
(usually omitting the subscript~$\calC$)
\begin{align*}
\calA_1\perp_{\cl C}\calA_2 &:=\{M\in\calC\mid\text{$\exists$ s.e.s or triangle
$T\to M\to F$ s.t. $T\in\calA_1, F\in\calA_2$} \},\\
&  \qquad \text{ if }\Hom(H_1,H_2)=0 \text{ for any } H_1\in\calA_1\,,H_2\in\calA_2,\\
\end{align*}
We define $\bra \cl B\ket$ depending on the context to be the
\emph{abelian category generated by}~$\cl B$, the thick
\emph{triangulated category generated by~}$\cl B$, the
\emph{torsion-free class} or the
\emph{torsion class generated by}~$\cl B$.
\par
Let $\calA$ be an abelian category. It is called a (finite) 
length category if any object $E\in \calA$ admits a finite sequence of subobjects
\[0=E_0\subset E_1 \subset \dots \subset E_m=E\]
such that all $E_i/E_{i-1}$ are simple. It is called finite if 
moreover it has finitely many simple objects. If an abelian 
category is finite, its Grothendieck group is generated by the 
isomorphism classes of its simples.
\par
Let $\calD$ be a triangulated category. For simplicity we 
make the strong assumption that its Grothendieck group 
is a finite rank lattice $K(\calD)\simeq \bZ^{\oplus n}$. 
This is not the general situation, though it will hold for 
the most relevant categories considered later. The main 
reason for such an hypothesis is to simplify the definition 
of a (multi-scale) stability condition.
\begin{definition}A \emph{bounded t-structure} on a triangulated 
category $\calD$ is the datum of a full additive subcategory $\cl P\subset \calD$ stable under positive shift such that 
$\cl P\perp\cl P^\perp=\calD$, and moreover 
$\calD$ is generated by $\cup_{m\in\bZ}\big(\cl P[m]\cap\cl P^\perp[-m]\big)$. 
The \emph{heart} of a bounded t-structure is the subcategory 
$\cl P \cap \cl P^\perp[1]$.
\end{definition}
The heart $\calA$ of a bounded t-structure is an 
abelian category. The cohomological functor 
$H^0:\calD\to \calA$ 
realizes an isomorphism at the level of Grothendieck groups
\[H^0_*:K(\calA)\simeq K(\calD).\]
Moreover, a bounded t-structure is uniquely determined by its 
heart as $\cl P= \bra \calH[i],i\geq 0\ket$.
For this reason we will speak about a t-structure or its heart 
interchangeably. 
\par
There is a partial order on hearts $\calA_1\le\calA_2$ defined
by $\calP_1\supset\calP_2$ or equivalently $\calP_1^\perp\supset\calP_2^\perp$.
A heart $\calH$ will be called
\emph{intermediate} with respect to a fixed heart $\calA$ if $\calA \leq
\calH \leq \calA[1]$.

\subsection{Torsion pairs and tilting} \label{subsec:tilt}

A \emph{torsion pair} for an abelian category $\calA$ consists on a
pair $(\torsion,\torsionfree)$ of full additive subcategories of $\calA$
called \emph{torsion class} and  \emph{torsion-free class}, such that
$\calA=\torsion\perp\torsionfree$. In other words, a torsion pair mimics a
bounded t-structure at abelian level. In fact, a torsion pair in the heart of a 
bounded t-structure $\calA$ in a $\calD$ defines new bounded t-structures with hearts 
\[ \mu_{\torsionfree}^\sharp\calA:= \torsion \perp_\calD \torsionfree[1], \quad
            \mu_{\torsion}^\flat\calA:=\torsionfree\perp_\calD \torsion[-1]
    \]
They are called respectively the \emph{forward tilt} at $\torsionfree$ 
(resp.\ \emph{backward tilt} at $\torsion$), \cite{tiltingbook}. They are
related by $\mu^\sharp_{\torsion[-1]} \mu^\flat_\torsion\calA=\calA$ and
$\mu^\flat_{\torsionfree[1]}\mu^\sharp_\torsionfree\calA = \calA$.
The forward tilt of~$\calA$ at a torsion-free class is intermediate with respect 
to~$\calA$; the backward tilt of~$\calA$ at a torsion class 
is intermediate with respect to~$\calA[-1]$.
\par
In a finite abelian category torsion and torsion free classes are closed under 
extensions and are characterized by being closed under 
quotients and subobjects respectively. This implies that any 
Serre subcategory is both torsion and torsion-free class. 
When we tilt at a torsion(-free) class $\bra S\ket$ 
generated by a simple object~$S$, we speak about a simple tilt 
and we simplify the notation to
$\mu^\sharp_S\calA$ and $\mu^\flat_S\calA$.
Suppose $\calA$ is a finite heart with simple objects $\Sim(\calA):=\{S_1,\dots,
S_n\}$, which are rigid, i.e., have no non-trivial self-extensions, and
let $S\in\Sim(\calA)$. Then
\[\begin{aligned}
  \Sim \mu^\sharp\calA &= \{S[1]\}\cup
  \{\Cone\big(S\stackrel{ev}{\to}S[1]\otimes \Ext^1(T,S)^*\big)[-1], 
  \ S\neq T \in\Sim\calA\}\\
    \Sim \mu^\flat\calA &= \{S[-1]\}\cup
    \{\Cone\big(S[-1]\otimes \Ext^1(S,T)\stackrel{ev}{\to} T), 
    \ S\neq T \in\Sim\calA\}\,,
\end{aligned}\]  
see e.g.\ \cite{kq}. Note also that the simple tilting of a finite heart in 
$\calD$ is another finite heart. In some cases tilting at a torsion (or 
torsion-free) class can be decomposed into a finite sequence of 
simple tilts.
\par
\begin{prop}
  \label{prop:tiltviasimpletilt}
  Suppose that $\calA$ is a finite heart. 
\begin{enumerate}
\item Tilting at a torsion-free class in $\calA$ containing only finitely many
indecomposables is equivalent to performing a sequence of simple forward tilts.
\item  Conversely, suppose $a_1, \dots, a_{k}$ is a finite sequence of objects
in $\calA$ such that $a_i \in \calA$ is simple in $\mu^\sharp_{a_{i-1}}\dots
\mu^\sharp_{a_1}\calA$. Then $\mu^\sharp_{a_{k}}\dots\mu^\sharp_{a_1}\calA
= \mu^\sharp_\cal{F}\calA$ where $\cal{F} = \langle a_0,\dots, a_k \rangle$.  
\item More generally, for any two torsion-free classes $\cal{F}_1\subset \cal{F}_2$
with $\calF_2$ having finitely many indecomposables, there is a sequence of
simple tilts at objects $a_i$ such that $\mu_{\cal{F}_2}^\sharp=\mu^\sharp_{a_k}\cdots
\mu^\sharp_{a_1} \mu^\sharp_{\cal{F}_1}$ and $\cal{F}_2=\bra \cal{F}_1,a_1,\cdots, a_k\ket$.
\end{enumerate}
\end{prop}
\begin{proof}
For the proof of the first two items see \cite[Proof of Proposition 2.4]{woolf1}
and use the relation $\mu^\sharp_{\torsion[-1]}(\mu^\flat_\torsion\calA) = \calA$ to
convert the statement about backward tilts in loc.\ cit.\ to the given version.
The last statement follows from \cite[Section 7.1]{ALSV} (see in particular
Proposition~7.5).
\end{proof}

\subsection{Bridgeland stability conditions}\label{subsec:stab}

Recall from \cite{BrStab} that a \emph{stability condition} $\sigma$ on a
triangulated category $\calD$ is a pair $\sigma= (\calA,Z)$, consisting of
the heart of a bounded t-structure $\calA$, together with a central
charge $Z\in\Hom(K(\calA),\bC)$, i.e., a group homomorphism that maps the
class of non-zero elements in $\calA$ to the semi-closed half plane
$\chalfplane:=\{re^{\pi i \theta}\in\R | r\in \R_{>0},\ 0<\theta\leq 1\}$
and that satisfies the support property and Harder-Narasimhan condition
of loc.~cit. We fix a finite rank lattice $K$ and a surjective morphism
$\nu: K(\calA)\to K$ and require that $Z$ factors through $\nu$. In the case
$K(\calA) \simeq \Z^n$ we require that $K = K(\calA)$ and $\nu =
\mathrm{id}$.
\par 
We use that stability conditions can equivalently be specified as a $\sigma =
(\cP, Z)$ using a central charge and a slicing, compatible in the sense that
$E\in \cl P(\phi)$ implies $Z([E])=m\exp(\pi i \phi)$ for some positive $m\in \bR$.
\par
An object $E\in\calD$ is called \emph{$\sigma$-semistable} if 
$E\in\cl P(\phi)$ for some $\phi\in\bR$. It is called 
\emph{$\sigma$-stable} if it is simple in $\cl P(\phi)$. 
This notion makes sense because any subcategory $\cl P(\phi)$ is abelian
if $\cl P=\{\cl P(\phi)\}_{\phi\in\bR}$ is a 
slicing compatible with $Z\in\Hom(K(\calD), \bC)$. 
\par
Let $\lambda \in \bC$ and suppose $0 < \epsilon = \Re(\lambda) \leq 1$.
We observe, and will use later, that $\sigma$-semistable objects $X$ in
the heart~$\calA$ (equivalently $Z$-semistable objects) with $1-\epsilon \leq
\phi(X) <(\leq) 1$ and those with $0<\phi(X)\leq (<)1-\epsilon$, for
$\epsilon\in(0,1)$, form a torsion pair $(\torsion_\lambda,\torsionfree_\lambda)$
in $\calA$ due to the Harder-Narasimhan condition.
\par
The space of stability conditions is a complex manifold 
$\Stab(\cD)$. There are two natural commuting actions: 
\begin{itemize}
  \item a left action by $\bC$, by rescaling the central 
  charge and tilting the heart, if $0 < \Re(\lambda) \leq 1$, 
  \[\lambda \cdot (\calA,Z) \=(\mu^\sharp_{\torsionfree_\lambda} \calA,e^{-\pi i \lambda} Z)\,,\]
  \item and a right action by $\Aut(\calD)$ via pullback,
    \[\Phi.(\calA,Z) \=\big(\Phi \calA,Z\circ [\Phi]^{-1})\big),\]
  where $[\Phi]$ is the map induced by $\Phi$ on $K(\calD)$.
  \end{itemize}
In particular the shift $[1]$ acts as $\lambda=1$. 
Note that the $\bC$-action does not change the notion of 
semistability and stability. 
\par 
We denote by $\Stab^\circ(\cD)$ be a connected component, 
specified by
the context. The stability manifold $\Stab(\cD)$ is tiled into
subsets $\Stab(\calA)$ of stability conditions supported on the heart~$\calA$.
The component $\Stab^\circ(\cD)$ is called \emph{finite type} if is the
union of $\Stab(\calA)$ over \emph{finite hearts}. It is called of \emph{tame type},
if the $\bC$-orbits of $\Stab(\calA)$ for all finite type hearts
cover $\Stab^\circ(\cD)$.
\par
We let $\Aut^\circ(\cD)$ be the subgroup of $\Aut(\cD)$ consisting on
autoequivalences of~$\cD$ that preserve the component $\Stab^\circ(\calD)$
and we define $\Nil^\circ(\calD) \subset \Aut^\circ(\cD)$ the subgroup
of \emph{negligible autoequivalences}, i.e.\ those that act trivially
on~$\Stab^\circ(\calD)$. We use fancy fonts like
\be\label{eq:pzcAut}
\pzAut^\circ(\calD) \= \Aut^\circ(\cD)/\Nil^\circ(\calD)
\ee
to denote the quotient groups by negligible autoequivalences. It's the
quotient spaces $\bC \backslash \Stab(\cD)/\pzAut(\cD)$ by these actions
that we want to compactify.

\par

\subsection{Quivers with potential, module and Ginzburg categories}
\label{subsec:quiver}

In this paper $(Q,W)$ is a quiver $Q=(Q_0,Q_1,s,t)$ with potential~$W$ 
(i.e., a formal sum of cycles) up to right-equivalence, see \cite{DWZ, KY} for
standard results. We assume that $(Q,W)$ has no loops and no 2-cycles, that the
set of vertices $Q_0$ and the set of arrows $Q_1$ are finite, and that 
the potential defines a bilateral ideal 
$\del W=\bra \partial_a W \mid a\in Q_1\ket\subset \bk Q$ 
such that the Jacobian algebra 
\[\cl J(Q,W) := \widehat {\bk Q}/\del W,\]
obtained by quotienting the completed path algebra by the ideal defined by
the potential, is finite dimensional. Note that in our case of interest,
$\widehat{\bk Q}/\del W =\bk Q/\del W$. For a ring $\cl J$, we denote by
$\Mod \cl J$ the abelian category of left modules and by $\modules \cl J$
the abelian category of finitely generated left modules. The
category $\modules \cl J(Q,W)$ is finite with simple
objects $\Sim(\modules \cl J(Q,W))=\{S_1,\dots, S_n\}$, where $n=|Q_0|$.
\smallskip\par
If $I\subset Q_0$ is a collection of vertices of $(Q,W)$, by 
$(Q_I,W_I)$ we mean the \emph{restriction} of $(Q,W)$ to $I$. It is another finite 
quiver with potential, possibly disconnected, defined by 
$(Q_I)_0=I$, $(Q_I)_1=\left\{a:i\to j\in Q_1\mid i,j\in I\right\}$, 
and with source, tail functions, and potential obtained 
by restriction from $(Q,W)$ to $I$. We call it a 
(full) subquiver. The complement of $I$ in $Q_0$ 
will be denoted $I^c$.
\par The \emph{mutation} of a quiver with potential $(Q,W)$ at a vertex $i$ is an 
operation that produces another quiver with the same set of vertices and new set
of arrow and new potential,  defined as follows. From $Q_1$, keep all arrow not
incident to $i$; replace any arrow $a$ with either $s(a)$ or 
$t(a)$ equal to $i$ with its opposite; add an arrow $[ab]:k_1\to k_2$ for any pair 
of consecutive arrow $a:k_1\to i$ and $b:i\to k_2$; finally, remove any two-cycles. 
The new potential is the formal sum of $W$ and $\sum_{a,b\in Q_1}[ab]b^*a^*$.
\smallskip\par
The \emph{Ginzburg algebra} of $(Q,W)$ is a dg algebra denoted 
$\Gamma(Q,W)$ and introduced in \cite{ginzburg, KY}. It does not depend 
on the mutation class of a quiver with potential.
\par
\begin{definition}
  The \emph{perfectly-valued derived category} $\pvd(\Gamma)$ 
  associated with a dg algebra $\Gamma$ is the subcategory of 
  the derived category $\calD(\Gamma)$ consisting on dg 
  modules with finite dimensional total cohomology.
\end{definition}
Once we fixed $(Q,W)$ and 
$I\subset Q_0$, we write $\cl J=\cl J(Q,W)$, $\Gamma=\Gamma(Q,W)$, and 
$\cl J_I=\cl J(Q_I,W_I)$, $\Gamma_I=\Gamma(Q_I,W_I)$. 
We have the following inclusion of 
triangulated categories, \cite{KY}
\[ \pvd(\Gamma) \subset \per(\Gamma) \subset \calD(\Gamma).\]
It is proven in \cite{KY} that the \emph{standard} t-structure with 
heart $\Mod\cl J$ in the derived category 
$\calD(\Gamma)$ restricts to $\per(\Gamma)$ and $\pvd(\Gamma)$, 
on which it defines a bounded t-structure with heart $\mod\cl J$, 
that we call \emph{standard} as well.
\par
The perfectly-valued derived category of the Ginzburg algebra of a quiver
with potential is 3-Calabi-Yau, 
which means that for any objects $E,F\in \pvd(\Gamma)$ there 
is a natural isomorphism  of $\bk$-vector spaces 
$\nu:\Hom(E,F)\stackrel{\sim}{\to}\Hom(F,E[3])^\vee$. 
Moreover, the simple 
objects in the standard heart $\mod\cl J$ are \emph{spherical} 
in $\pvd(\Gamma)$, see \cite[Lemma 4.4]{keller11} and \cite[Corollary 8.5]{kq}.
\par
If two quivers with potential $(Q,W)$ and $(Q',W')$ are 
related by mutations, then $\calD(\Gamma(Q,W))\simeq \calD(\Gamma(Q',W'))$ 
and $\pvd(\Gamma(Q,W))\simeq \pvd(\Gamma(Q',W'))$. 
Therefore $\modules\cl J(Q',W')$ is viewed as another 
heart of bounded t-structure of $\pvd\Gamma(Q,W)$. We 
recall that in general not all bounded t-structures have this shape.
\par
It is clear that any property of $\pvd(\Gamma)$ and $\modules(\cl J)$ also holds for $\pvd(\Gamma_I)$ and $\modules \cl J_I$.
\par
As explained in~\cite{kalckyang2}, the Ginzburg algebra
$\Gamma_I$ is isomorphic to $\Gamma/\Gamma e \Gamma$,
where $e=\sum_{i\in I^c}e_i$ is the idempotent in $\Gamma$ associated to the 
complement $I^c= Q_0\setminus I$. On the other hand the dg algebra $e\Gamma e$
is the endomorphism algebra of the projective module
$\Gamma e = \sum_{i\in I^c}\Gamma e_i$ in $\calD(\Gamma)$ and the Verdier 
quotient $\calD(\Gamma)/\calD(\Gamma_I)$ coincides with $\calD(e\Gamma e)$. Similarly $\cl J_I=\cl J/\cl J e \cl J$ and the 
quotient perfectly valued and abelian categories that will be 
relevant in the rest of the paper are $\pvd e\Gamma e$ and $\modules e\cl J e$:
\[\xymatrix{
 0 \ar[r] & \pvd\Gamma_I \ar[r]\ar[d]^{H^0} & \pvd \Gamma \ar[d]^{H^0} \ar[r] & \pvd e\Gamma e \ar[r] \ar[d]^{H^0} & 0\\
 0 \ar[r] & \modules \cl J_I \ar[r] & \modules \cl J \ar[r] & \modules e \cl J e \ar[r] &0.
}\]
The last line is part of a recollement of abelian categories, described for
instance in \cite{psarou}. 
\par
\medskip
In the rest of the paper we let 
\begin{itemize}
  \item $\calD_Q^3$ be the 3-Calabi-Yau triangulated category $\pvd\Gamma(Q,W)$.
\end{itemize}
The case of primary interest will be quivers \emph{of type $A_n$}, i.e.,
that  can be obtained with by finite sequence of mutations from the quiver
\[\xymatrix{
  A_n \,:=\,\bullet_1\ar[r]&\bullet_2\ar[r] &\cdots \ar[r]&\bullet_n},
\quad n\geq 1.
\]
Any restriction of a quiver of type $A_n$ is a union of quivers of 
type $A_m$'s. 
\par
Given an $A_n$-configuration, and the abelian category 
$\modules \cl J(A_n)$, we denote by $S_i$ the simple module associated 
with the vertex $i$. For $i\leq k$, we denote by $S_{i...k}$ the $\cl J(A_n)$-module 
defined inductively as the indecomposable fitting into the short exact sequence
\begin{equation}\label{S_ij}
		0 \rightarrow S_{k+1} \rightarrow S_{i\dots (k+1)} \rightarrow S_{i\dots k} \rightarrow 0.
\end{equation}
The $S_{i\dots k}$, are the projective resolutions $P_i$ of $S_i$ in the abelian 
subcategory $\bra S_i,\dots, S_{k}\ket$ which is of $A_{(k-i)}$-type by construction.

\section{Stability manifolds for marked surfaces} \label{sec:DMS}

Decorated marked surfaces are one of the natural sources for quiver
categories. They are well-studied thanks also to the Bridgeland-Smith
isomomorphism \cite{BS15} to spaces of quadratic differentials with
simple zeros. We recall this result here, together with the generalization
in our previous paper \cite{BMQS}. This setup contains our main
case study, the $A_n$-quiver, and serves as motivation for the use of
quotient categories. Triangulations of decorated marked surfaces will
serve as reference point to pick out the right connected components
of stability manifolds needed in the later sections.

\subsection{The stability manifold of a decorated marked surface}
\label{sec:wDMS}

A natural way to construct quivers is from triangulations of surfaces
and we will use this formalism to keep track of connected components
of stability spaces and later the multi-scale stabilty conditions.
\par
A \emph{marked surface} $\surf=(\surf,\MM,\PP)$ consists of a connected
bordered differentiable  surface with a fixed orientation, with a finite set
$\MM=\{M_i\}_{i=1}^b$ of marked points on the boundary
$\partial\surf=\bigcup_{i=1}^b \partial_i$, and a finite set $\PP=\{p_j\}_{j=1}^p$
of punctures in its interior $\surf^\circ=\surf-\partial\surf$, such that each
connected component of $\partial\surf$ contains at least one marked point.
\par
A \emph{decorated marked surface $\surfo$} (abbreviated as \emph{DMS})
is obtained from
a marked surface $\surf$ by decorating it with a set $\Delta= \{z_i\}_{i=1}^r$ of
points in the surface interior~$\surf^\circ$. These points are called
\emph{finite critical points} or \emph{finite singularities}.
\par
An \emph{open arc} is an (isotopty class of) curve $\gamma\colon I\to\surfo$
such that its interior is in $\surfo^\circ\setminus\Tri$ and its
endpoints are in the set of marked points~$\MM$.
An \emph{(open) arc system} $\{\gamma_i\}$ is a collection of open arcs
on $\surfo$ such that there is no (self-)intersection between any of them
in~$\surfo^\circ \setminus\Tri$.
A \emph{triangulation~$\TT$} of $\surfo$ is a maximal arc
system of open arcs, which in fact divide $\surfo$ into triangles.
\par
The quiver $Q_\TT$  with potential $W_\TT$ associated to a
triangulation $\TT$ is constructed as follows. The vertices 
correspond to the open arcs in $\TT$, the arrows of $Q_\TT$ correspond to
oriented intersection between open arcs in~$\TT$, so that there is a 3-cycle
in $Q_\TT$ locally in each triangle, and the potential~$W_\TT$ is the sum of
all such $3$-cycles.
\par
For a fixed initial triangulation~$\TT_0$ we denote by $\Gamma_{\TT_0}
= \Gamma(Q_{\TT_0}, W_{\TT_0})$ the Ginzburg algebra associated with the
quiver associated with~$\TT_0$ we let $\calD^3_{Q_{\TT_0}} = \pvd(\Gamma_{\TT_0})$
or simply~$\DQ$ the corresponding $\CY_3$-category. Finally, we define
$\Stab^\circ(\DQ)$ to be the connected component of the space of Bridgeland
stability conditions on $\DQ$ containing stability conditions supported on 
the standard heart $\cl H_0$ of~$Q_{\TT_0}$.
\par
In this paper we fix throughout a DMS~$\surf$ of type~$A_n$.
It is a  disc with $b=1$ boundary component, which has
$n+3$ marked points, $r= n+1$ finite critical points in its interior,
and no punctures.
We use this reference surface and a reference triangulation on it to
define the component $\Stab^\circ(\cD_{A_n})$. Recall from
\cite[Theorem~9.9 and Section~12.1]{BS15} that the subgroup
$\Aut^\circ(\cD_{A_n}) \subset \Aut(\cD_{A_n})$ preserving a connected component
of $\Stab(\cD_{A_n})$ is an extension of~$\bZ/(n+3)\bZ$ by the spherical
twist group~$\ST(A_n)$.
\par
In this language the main theorem of Bridgeland-Smith (for a
general $\CY_3$-quiver category~$\DQ$ associated with a triangulation of
$\surfo$, see \cite{BS15} for the excluded cases) reads:
\par
\begin{theorem}[\cite{BS15,KQ2}] \label{thm:BS15_iso}
There is an isomorphism of complex manifolds
\be
K: \FQuad^\circ(\surf_{\Delta})   \to \Stab^\circ(\DQ)\,.
\ee
This map~$K$ is equivariant with respect to the action of the mapping
class group $\MCG(\surf_{\Delta})$ on the domain and of the automorphism
group $\pzAut^\circ(\cD)$ on the range. These groups act properly discontinuously
on domain, resp.\ range.
\end{theorem}
\par
Here $\FQuad^\circ(\surf_{\Delta})$ is a space of framed quadratic differentials
with simple zeros at $\Delta$, whose definition we recall along with
the examples in Section~\ref{sec:BCGGMforAn}. Its generalization to
non-simple zeros motivates the notion of collapse and the use of quotient
categories, which we recall in Section~\ref{sec:BMQS}.
\par
As technical tool we introduce the \emph{exchange graph} $\EG(\surfo)$,
the directed graph whose vertices are the triangulations of~$\surfo$ and
whose edges are given by (forward) flips of the triangulation. The
\emph{exchange graph} $\EG(\calD)$ of a triangulated category is the directed
graph whose vertices are the finite hearts and whose edges are give by forward
tilts at simples in the heart. We denote by $\EGp(\surfo)$ the connected
component containing the initial triangulation~$\TT_0$ and by $\EGp(\DQ)$ the
connected component corresponding to the standard heart $\mod\cl J(Q_{\TT_0},
W_{\TT_0})$. A key step in the proof of Theorem~\ref{thm:BS15_iso} is the isomorphism 
\be
\EGp(\surfo) \,\cong\, \EGp(\DQ)
\ee
of exchange graphs.

\subsection{Stability manifolds of certain quotient categories}
\label{sec:BMQS}

Higher order zeros are modeled by the collapse of a subsurface
$\Sigma \subset \surfo$ in a DMS. We use this to deduce information on
certain components of the stability manifold of the quotient categories
$\DQ / \DQI$. We decompose $\Sigma$ into connected components $\Sigma_i$,
provide each boundary component of $\Sigma_i$ with an integer \emph{enhancement}
$\kappa_{ij}$. To match hypothesis with \cite{BMQS} we suppose throughout
that $\kappa_{ij} \geq 3$ and consider $\Sigma$ as a marked surface with
$\kappa_{ij}$ points on each boundary component.
\par
To topologically formalize the collapse of~$\Sigma$ we define a \emph{weighted DMS}
(wDMS for short) to be a DMS with a weight function $\w: \Delta \to \bZ_{\geq -1}$
where the total weight is required to be $||\w|| = 4g-4 + |\MM|+2b$.
Contracting~$\Sigma \subset
\surfo$ and replacing each boundary component by a decoration point in~$\Delta$
with weight $\w_{ij} = \kappa_{ij}-2$ defines a wDMS that we usually denote
by~$\colsur$. In the sequel (as in \cite{BMQS}) we \emph{restrict to the case
of no punctures $p=0$, no unmarked boundary components and
$\w: \Delta \to \bZ_{\geq 1}$.}
\par
To categorify the collapse we homotope the initial
triangulation~$\TT_0$ such that the arcs intersect the boundary of $\Sigma$
in the marked points and such that $\TT_0|_\Sigma$ is a triangulation of
this subsurface. In this way $\Sigma$ becomes a DMS with a triangulation
and we may form the $\CY_3$-category $\calD_3(\subsur)$. We define the Verdier
quotient category
\be
\calD(\colsur) \,:=\, \Dsan/\calD_3(\subsur)
\ee
\par
As in Section~\ref{sec:wDMS} there are two exchange graphs associated with this
situation, one based on ``flips'' and topology and the other based on tilts of hearts.
The isomorphism~\eqref{eq:EGcoliso} below between these graphs is one of reasons to
work with the quotient categories.
\par
A \emph{partial triangulation~$\AS$ of $\colsur$} is a collection of open arcs
that triangulates the subsurface of $\surfo$ whose complement is
homeomorphic to $\Sigma$, and such that each boundary component $c_{ij}$
of~$\Sigma$ is homeomorphic in $\colsur \setminus \AS$ to a $(\kappa_{ij}
= w_{ij}+2)$-gon, possibly with ends points identified.
\par
On the set of partial triangulation~$\AS$ there is an operation of
\emph{forward flip of an arc $\gamma \in \AS$}, defined by moving both endpoints
counterclockwise one edge bounding the subsurface of $\surfo \setminus (\AS \setminus
\{\gamma \})$ that contains~$\gamma$. This generalizes the usual notion of flip
of triangulations, see \cite[Figure~2]{BMQS}. We define the exchange graph
$\EG(\colsur)$ to be the (infinite) directed graph whose vertices are the
partial triangulations of the decorated surface~$\colsur$ and whose edges
are given by forward flips.
\par
\begin{definition} \label{def:CompatQuotType}
Let $\calV \subset \calD$ be a thick triangulated subcategory.
We say that a heart~$\calA$ of~$\calD$ is \emph{$\calV$-compatible}, if
$\calA \cap \calV$ is a Serre subcategory of~$\calA$.
\par
We call a heart $\ol{\calA}$ of $\calD/\calV$ \emph{of quotient type} if there
is a $\calV$-compatible heart~$\calA$ of~$\calD$ whose essential image in $\calD/\calV$
is~$\ol{\calA}$.
\end{definition}
\par
We define the
\emph{principal component} $\EGb(\colsur)$ to be the full subgraph of partial
triangulations that admit a refinement to a triangulation in $\EG^\circ(\surfo$), 
i.e.\ the full subgraph given by triangulations reachable by a finite number of flips
from~$\TT_0$. We define the
\emph{principal component} $\EGb(\calD(\colsur))$ to be the full subgraph of
$\EG(\calD(\colsur))$ consisting of hearts of quotient type that admit a
representative in the distiguished component $\EG^\circ(\calD(\surfo))$. It is
a priori not clear that these definitions yield connected components. This is
proven along with \cite[Theorem~5.9]{BMQS}, which moreover states that
\be \label{eq:EGcoliso}
\EGb(\colsur)\,\cong \,\EGb(\Dcol)
\ee
and that both graphs are $(m,m)$-regular.
\par
We now define the \emph{principal component of the stability manifold}
$\Stab(\Dcol)$ to be 
\be \label{eq:defStabbullet}
\Stas(\Dcol) \=\bC\cdot \bigcup_{\calH\in\EGb(\Dcol)} \Stab(\calH).
\ee
The terminology is justified by the following results:
\par
\begin{prop}[\cite{BMQS}]\label{prop:conn_comp}
The space $\Stab^\bullet(\Dcol)$ is union  of connected components of~$\Stab(\Dcol)$.
\end{prop}
\par
Referring to Section~\ref{sec:BCGGMforAn} for the definition of framed quadratic
differentials we recall here the generalization of the Bridgeland-Smith isomorphism
that serves as motivation for definition of multi-scale stability conditions
using the comparison to compactification of strata, see Section~\ref{sec:BCGGMforAn}.
\par
\begin{theorem}[Theorem~1.1 of \cite{BMQS}] \label{thm:KrhoBihol}
There is an isomorphism of complex manifolds
\[
K: \FQuad^\bullet(\colsur)   \to \Stab^\bullet(\calD(\colsur))
\]
between the principal part of the space of Teichm\"uller-framed quadratic
differentials and the principal part of the space of stability conditions
on $\calD(\colsur)$.
\end{theorem}
\par

\subsection{Braid groups and spherical twists} \label{sec:STandBraid}

Motivated by the relation to groups of autoequivalences given
in~\eqref{eq:ses.aut} and~\eqref{eq:MCG1structure} below
we recall a few basic properties of the braid groups $B_{n+1}$ on $n+1$
strands, see e.g.\ \cite[Section~9]{farbmarg}. The 
standard generators are the $\tau_i$ twisting the strands~$i$ and~$i+1$.
They satisfy the defining \emph{standard braid relations}.
\ba \label{eq:braidrel}
\tau_i\tau_j\tau_i\=\tau_j\tau_i\tau_j &\quad\text{if}\quad |i-j|=1,
\quad \text{and} \quad 
\tau_i\tau_j\=\tau_j\tau_i &\quad \text{if}\quad |i-j|\geq 2,
\ea
The \emph{pure braid group} is the kernel of the homomorphism recording
the strand permutation, i.e.\ sits in the exact sequence
\be \label{eq:PBBseq}
0 \to \PB_{n+1} \to B_{n+1} \to S_{n+1} \to 0 \,.
\ee
The center of~$B_{n+1}$ is cyclic, generated by the element
\be \label{eq:centerBn}
\theta_n = (\tau_1 \tau_2 \cdots \tau_n)^{n+1}
\ee
except for the case where of $n=1$, since $B_2$ is cyclic where thus
$\theta_2$ is the square of the generator. In all cases  the element~$\theta_n$
as defined above also belongs to the pure braid group $\PB_{n+1}$.
As a special case of the Birman exact sequence we obtain 
\be
1 \to F_{n+1} \to \PB_{n+1} \to \PB_n \to 1\,,
\ee
where~$F_{n+1}$ is the free group on $n+1$ generators.
This exact sequence is split by adding the extra strand. Iterating this
we obtain for each~$r$ consecutive integers a natural homomorphism
\be \label{eq:phirn}
\varphi_{r,n}: \PB_{r+1} \to \PB_{n+1}
\ee
Via this homomorphism we define for $I = \{1,\ldots,r\}$ the elements
$\theta_{I,n} := \varphi_{r,n} (\theta_r) \in B_{n+1}$. These correspond to
a full rotation of a disc encircling precisely the points in~$I$.
\par
More generally we will define the braid group~$B_Q$ associated with a quiver~$Q$,
see \cite[Definition~10.1 and Proposition~10.4]{qiubraid16}. It is generated
by an element~$\wt\tau_i$ for each vertex of the quiver and defined by
the relations \ba \label{eq:braidrel2}
\wt\tau_i\wt\tau_j\wt\tau_i\=\wt\tau_j\wt\tau_i\wt\tau_j &
\quad\text{if}\quad |i-j|=1\\
\wt\tau_i\wt\tau_j\= \wt\tau_j\wt\tau_i &\quad \text{if}\quad |i-j|\geq 2, \\
R_i = R_j &\quad\text{for each cycle $1 \to 2 \to \cdots \to m \to 1$}\,,
\ea
where $R_i = \wt\tau_i \wt\tau_{i+1} \cdots \wt\tau_m\wt\tau_1 \cdots
\wt\tau_{i-1}$. In the case $Q=A_n$ we retrieve the above definition of
$B_{A_n} = B_{n+1}$. We will be most interested in the case that~$Q$ is
of type~$A_n$ though not necessarily equal to~$A_n$. Suppose the vertices
in the index set~$I$ form a subquiver of type $A_r$. Then the braid group
of the restricted quiver $B_{Q_I} \cong B_{r+1}$ and we let~$\theta_r$ be its
central element. Inclusion of strands again defines a natural homomorphism
$\varphi_{I,n}: B_{Q_I} \to B_{n+1}$ and we define in this more general context
$\theta_{I,n} := \varphi_{I} (\theta_r)$.
\par
Given a quiver~$(Q,W)$ and its Ginzburg algebra~$\Gamma$, we
let $\ST(\Gamma) \leq \Aut(\pvd(\Gamma))$ be the \emph{spherical twist group}
(see Seidel-Thomas \cite{seidelthomas})
of $\pvd(\Gamma)$, that is the subgroup generated by the set of
spherical twists $\Phi_{S}$ for all simples~$S$ of~$\Gamma$,
where the \emph{twist functor} $\Phi_S$ is defined by
\begin{equation}\label{eq:phi+}
    \Phi_S(X)\= \Cone\left(S\otimes\Hom^\bullet(S,X)\to X\right)\,.
\end{equation}
This uses that in the case $\calD = \DQ$, as a consequence of
\cite[Corollary~8.5]{kq}, all the simples in any heart of~$\calD$ are spherical.
\par
\begin{rem} \label{commentonmutation} {\rm
For a heart $\calA$ with simples $S_1,...,S_n$ listed in an order such that
$\dim(\Ext^1(S_j,S_i)) = 0$ for $j<i$ and $\dim(\Ext^1(S_i,S_j)) = 0$ for $j>i$
the mutated heart $\mu_{S_i}^\sharp(\calA)$ has simples
$$S_1,...,S_{i-1}, S_i[1],\Phi_{S_i}^{-1}(S_{i+1}),...,\Phi_{S_i}^{-1}(S_n)\,,$$
compare  \cite[Proof of Proposition~7.1]{BS15}. Moreover, 
\begin{equation}\label{PhiSA}\mu^\sharp_{S[1]}\mu_S^\sharp\calA \=\Phi^{-1}_S\calA\,.
\end{equation}}
\end{rem}
\par
We write $\ST(A_n)$ for spherical twist group of a quiver of type~$A_n$.
\par
\begin{prop}[\cite{seidelthomas, qiubraid16,Q3}]
There is an isomorphism $\ST(\Gamma)\cong B_Q$ between the twist groups,
sending the standard generators $\tau_i \to \Phi_{S_i}$ to the standard
generators. In particular the group $\ST(A_n)$  is isomorphic to the braid
group~$B_{n+1}$.
\end{prop}
\par
We now apply this to understand the groups of autoequivalences for
$\calD = \pvd(\Gamma)$. By \cite[Theorem~9.9]{BS15} there is an exact sequence 
\be\label{eq:ses.aut}
    1\to\mathpzc{ST}(\DQ) \to \mathpzc{Aut}^\circ(\DQ) \to \MCG(\surf) \to 1,
\ee
where $\mathpzc{ST}(\DQ)$ is the quotient of $\ST(\DQ)$ by its subgroup
of negligible automorphisms. In the special case $\DQ = D^3_{A_n}$
the mapping class group is  $\MCG(\ol{\Delta},\MM_{n+3}) \cong \bZ/(n+3)$,
so there is an exact sequence
\be \label{eq:MCG1structure}
1 \to B_{n+1} \to \mathpzc{Aut}^\circ(\calD^3_{A_n})  \to \bZ/(n+3) \to 1
\ee
see e.g.\ \cite[Section~12.1]{BS15},
\par
In Section~\ref{subsec:complexstr} we need the following action on
Grothedieck groups to control the effect of the action of $\theta_{I,n}$.
\par
\begin{lem}[{\cite[Section 4]{IkedaAn}}]On $K(\DQ)$, a spherical twist
$\Phi_{S_i}$ induces a group homomorphism $[\Phi_i]$ defined by 
\be \label{eq:Phiinduced}
[\Phi_i]([E])\=[E]-\chi([S_i],[E])[S_i],
\ee
where $\chi(\cdot,\cdot)$ denotes the Euler pairing. 
\end{lem}
\par


\section{Multi-scale stability conditions} \label{sec:multiscale}

We start with a definition that makes sense for general triangulated
categories with finite rank Grothendieck group. We will then define
a $\bC$-action and a notion of plumbing stability conditions that
will be used to give a topology on the set of multi-scale stability
conditions. In all these steps we have to be much more restrictive,
essentially restricting to $\calD = \calD^3_{A_n}$. We indicate the
technical difficulties needed to overcome in order to generalize
to $\CY_3$-quiver categories or beyond.
\par
\begin{definition} \label{def:mstab}
Let~$\calD$ a triangulated category with $\mathrm{rank}(K(\calD)) < \infty$.
A \emph{multi-scale stability condition} on~$\calD$ consists of an
equivalence class of the following data:
\begin{itemize}[nosep]
\item a \emph{multi-scale heart} $\cA_\bullet = (\cA_i)$, i.e., a
collection  $\cA_L \subset \cdots \cA_1 \subset \cA_0$
of abelian categories, and 
\item a \emph{multi-scale central charge}, i.e., a collection
$Z_\bullet = (Z_i)_{i=0}^L$ of non-zero $\Z$-linear maps on the
Grothendieck groups $Z_i: K(\cA_i) \to \bC$,
\end{itemize}
with the property that
\begin{itemize}
\item the abelian category $\cA_i$ is generated by the non-zero
  objects~$E$ in $\cA_{i-1}$ with $Z_{i-1}(E)=0$ for all $i \geq 1$,
\item the abelian category $\cA_i$ is a heart of~$\calV_i$, which is
defined as $\calV_0 = \calD$ and for all $i \geq 1$ as the thick
triangulated subcategory of~$\calV_{i-1}$ generated by~$\calA_i$,
\item the map $Z_i$ factors through $K_{i-1} := \Ker(Z_{i-1})$,
\item the induced heart $\ol{\cA}_i = \cA_i/\cA_{i+1}$ 
together with the induced central charge $\ol{Z_i}: K(\calV_i/\calV_{i+1}) \to \bC$
form a stability condition in the usual sense on $\calV_i/\calV_{i+1}$
for all $i=0,\ldots,L$.
\end{itemize}
Two multi-scale stability conditions $(\cA_\bullet,Z_\bullet)$ and
$(\cA_\bullet',Z_\bullet')$ are \emph{equivalent}, if
\begin{itemize}[nosep]
\item[i)] there is equality of triangulated categories $\cV_i = \cV_i'$
for $i=0,\ldots L$,
\item[ii)] the induced stability conditions $(\ol{A_i},\ol{Z}_i)$
and $(\ol{A_i}',\ol{Z}_i')$ are projectively equivalent
for $i=1,\ldots,L$, and are equal for $i=0$.
\end{itemize}
\par
Two multi-scale stability conditions $(\cA_\bullet,Z_\bullet)$ and
$(\cA_\bullet',Z_\bullet')$ are \emph{projectively equivalent} if the
projective equivalence in~ii) above holds for $i=0,\ldots,L$.
\end{definition}
\par
We write $[\calA_\bullet,Z_\bullet]$ for an equivalence class, and 
$(\calA_\bullet,Z_\bullet)$ for a representative of a multi-scale 
stability condition. Moreover we denote by $\cl V_\bullet$ the collection 
of nested triangulated subcategories $(\cl V_i)$ defined by 
$(\calA_\bullet,Z_\bullet)$. Sometimes write $\calV_i^Z$ for the categories~$\calV_i$
defined above to indicate the dependence on~$Z_\bullet$.
\par
The definition relies on the following lemma for the quotient hearts to
be meaningful.
\par
\begin{lem}\label{lem:serrenessV} The subcategory $\calA_{i+1}$ is Serre in $\calA_i$
and $\calA_{i+1} = \calV_{i+1} \cap \cA_i$ for all~$i$, i.e., $\calA_{i+1}$ is 
$\calV_{i+1}$-compatible in the sense of Definition~\ref{def:CompatQuotType}.
In particular, the inclusion $\iota: \calV_{i+1}\to \calV_i$ is t-exact with respect
to~$\calA_{i+1}$ and~$\calA_i$, and ~$\calA_i$ induces a quotient heart
in $\calV_i /\calV_{i+1}$. Moreover, $K(\calV_i/\calV_{i+1})=K_i/K_{i+1}$ so that~$Z_i$
descends to~$\ol{Z_i}$, as required.
\end{lem}
\par
\begin{proof}
Serreness of $\calA_{i+1}\subset\calA_i$  follows from the additivity of $Z_i$ on
short exact sequences and the fact that~$Z_i$ takes values in a strictly convex
sector in $\C$. The second statement follows from the observation that
$Z_i(X) = 0$ for any $X \in \calV_{i+1}$. Serreness of $\calA_{i+1}\subset \calA_i$ 
guarantees that $\calV_{i+1}$ consists on objects of $\calV_i$ whose cohomology 
with respect to $\calA_i$ is concentrated in $\calA_{i+1}$, and that the t-structure 
restricts, so the next 
claim follows from
\cite[Proposition~2.20]{antieau} or \cite[Lemma~3.3]{chuang_rouquier}. For the
last observe that $K_{i+1}$ is the image of $\iota_*:K(\calA_{i+1})\to K_i$ and
that the right exact sequence $K(\calA_{i+1})\to K(\calA_i)\to K(\calA_i/\calA_{i+1})
\to 0$ is used to compute the Grothendieck group of the quotient category.
\end{proof}
\par
Let $\MStab(\calD)$ be the set of all multi-scale stability conditions on~$\calD$. 
The integer~$L$ in Definition~\ref{def:mstab} will be referred to as the
number of \emph{levels below zero} of the stability condition. A usual stability
condition has~$L=0$.
\par
\medskip
\paragraph{\textbf{Reachability}}
We now fix a component $\Stab^\circ(\calD)$ of the stability manifold of~$\calD$.
If $\calD = \DQ_3$ is a quiver category we use an initial triangulation~$\TT_0$
of a DMS as in Section~\ref{sec:wDMS} to single out this components. 
\par
A multi-scale stability condition  $(\cA_\bullet,Z_\bullet)$ is called
\emph{reachable} if the top level heart~$\calA_0$ supports stability
conditions in $\Stab^\circ(\calD)$. We denote the set of all reachable
multi-scale stability conditions on~$\calD$ by $\MStab^\circ(\calD)$, and
the set of reachable multi-scale stability conditions with the same $\cl V_\bullet$
by $\MStab^\circ(\calD, \cl V_\bullet)$.
\par
In the next section, we will informally call a multi-scaled stability 
conditions with at least one level below zero a \emph{boundary 
point}, and finally prove this is actually the case.
\par
\medskip
\paragraph{\textbf{Groups of autoequivalences}} For general~$\calD$ we
define the group $\Aut(\calD,\calV)$ to be the autoequivalences that
stabilize~$\calV$.
For $\calD = \DQ$ we use bullets to denote those autoequivalences that moreover
stabilize the principal components defined in Section~\ref{sec:BMQS}. 
\par
\begin{lemma} \label{le:propdisc}
The factor group
\be
\pzAut_{\lift}(\calD/\calV) \,:=\,  \pzAut^\bullet(\calD,\calV)/\pzAut^\bullet(\calV)
\ee
acts properly discontinuously on $\Stab^\bullet(\calD/\calV)$.
\par
The stabilizer $H_{[\sigma]} \subset \pzAut_{\lift}(\calD/\calV)$ of a projectivized
stability condition $[\sigma]$ is a finite extension of a subgroup~$Z$ in the center
of~$\pzAut_{\lift}(\calD/\calV)$. In case $\calD = \calD^3_{A_n}$ the group~$Z$
is the center, generated by $\theta_n$ defined in~\eqref{eq:centerBn}.
\end{lemma}
\par
\begin{proof} 
The first statement is shown in the proof of \cite[Theorems~8.1, 8.2]{BMQS}, 
summarized here in Theorem~\ref{thm:KrhoBihol}, 
(see in particular the part about the orbifold structure, together with Equation~(5.7)). 
\par
For the second statement we may pass, thanks to the first statement, to the finite
index subgroup of the stabilizer $H_{[\sigma]}$ of a projective stability condition
that acts trivially on a neighborhood of the unprojectivized~$\sigma$. 
As in the proof of \cite[Theorems~8.2]{BMQS}, since the action of~$\bC$
and $\pzAut_{\lift}(\calD/\calV)$ commute, the second statement follows. The last
claim is a restatement of braid group properties from Section~\ref{sec:STandBraid}.
\end{proof}
\par

\subsection{Numerical data of multi-scale stability
conditions of type~$A_n$}\label{sec:numdata}

From here on we restrict to the case $\calD = \calD^3_{A_n}$. In this case we
can completely describe the subcategories and hearts that appear in a multi-scale
stability condition.
\par
\begin{lem}\label{lem:possibleV} If $(\cA_\bullet,Z_\bullet)$ is a multi-scale stability
condition and $\calA_0=\modules \cl J_Q$, then  $\calA_{1}=\modules \cl J_{Q_{I}}$
where~$I$ is a subset of the vertices of~$\calQ$, and~$\calV_1 = \pvd \Gamma_I$.
\par
Moreover, there is a bijection of the subcategories~$\calV_1^Z$
with homotopy classes of decorated marked subsurfaces~$\Sigma \subset \surfo$,
such that each component~$\Sigma_j$ for $j \in J$ is of type~$A_{n_j}$, i.e.\ a disc with
$n_j+1$ decoration points in its interior and $n_j+3$ marked points at its boundary.
Here $n_j \geq 1$ and the decomposition is constrained precisely by
$\sum_{j \in J} (n_j+1) \leq  n + 1$, where equality is allowed if and only if
$|J| \geq 2$.
\end{lem}
\par
Using this notation we say that $\calV_1^Z$ is a subcategory \emph{of type
$\rho := (n_1,\ldots,n_{|J|})$}. Iterating this over all~$\calV_i$
appearing in a multi-scale stability condition, we say that $[\calA_\bullet,Z_\bullet]$
\emph{is of type $\bfrho = (\rho_i)_{i=1}^L$}, where $\rho_i$ is the type of $\calV_i$. 
\par
\begin{proof}
Consider $\calA_1 \subset \calV_1^{Z}$. Since $\calA_1 \subset \calA_0$ is
Serre, it is generated by a subset $\calS_1 \subset \Sim(\calA_0)$ of the
simples of~$\calA_0$, those whose $Z_0$-image is zero. By the correspondence
summarized in Section~\ref{subsec:quiver} this defines the subquiver~$Q_I$
and shows  $\calA_{1}=\modules \cl J_{I}$.
\par
For the second claim, we show the chain of equalities
$$ \calV_1 \= \pvd_{J_I}(\Gamma)
\= \pvd_{\Gamma / \Gamma e \Gamma} (\Gamma) \= \pvd(\Gamma_I)\,,$$
where $\pvd_R(\calD) \subset \pvd (\calD)$ is the full subcategory with cohomologies
in~$\modules R$. The first follows using the characterization of~$\calA_0$ as a heart
in terms of a decomposition of objects in~$\pvd_{J_I}(\Gamma)$ into triangles and
Serreness of $\modules \cl J_I$ in $\calA_0$, as in Lemma~\ref{lem:serrenessV}.
For the second we just intersect $\Gamma_I=\Gamma/\Gamma e \Gamma$ with
$H^0\Gamma=\cl J$. The last equality follows from \cite[Corollary 6.4~(b)]{kalckyang2},
where we can take  $B=\Gamma/\Gamma e \Gamma=\Gamma_I$ thanks to Theorem~7.1
in loc.\ cit.
\par
Conversely, choosing
$Z_0$ to be zero for any subset of the simples and $\ol{\bH}$-valued for the
complementary set of simples defines a subcategory~$\calV_1^Z$ that can be
be completed to a multi-scale stability condition.
\par
We now translate into the language of Section~\ref{sec:DMS}. Let $\TT$ be the 
triangulation of the DMS $\surfo$ corresponding to~$\calA_0$. Dual to the open
arcs forming the triangulation there are closed arcs connecting the decorating
points. These are in canonical bijection to the simples in~$\calA_0$,  see e.g.\
the summary in \cite[Theorem~7.2]{BMQS}. Let $\AS_1^\vee = \{\eta_S, S \in \calS_1\}$
be the closed arcs corresponding to~$\calA_1$ and $\AS_1$ the set of dual closed
arcs. Let~$\Sigma = \Sigma_1$ be the subsurface  consisting of a
tubular neighborhood of~$\AS_1^\vee$
and let~$\Sigma^{(j)}$ denote its connected components. They are all homotopic to
a disc containing a certain number, say $n_j+1$, of simple zeros with one boundary
component. Homotoping the
open arcs in~$\AS_0$ so that they intersect~$\Sigma$ minimally, we deduce from
duality that precisely those in~$\AS_1$ have non-trivial intersection with~$\Sigma$.
We may thus mark $\kappa^{(j)} = n_j+3$ points on the boundary of~$\Sigma^{(j)}$
and restrict the arcs in~$\AS_1$ to arcs in $\Sigma$ connecting these boundary
points so that $\AS_1|_\Sigma$ is a triangulation of~$\Sigma$. In fact the quiver
associated with each subsurface~$\Sigma^{(j)}$ is of type $A_{n_j}$.
\par
The constraints for $n_j$ reflect that the total number of decoration points
in~$\surfo$ has to be at least two and the fact that there is at least one
simple outside~$\calA_1$, i.e.\ a closed arc not contained in~$\Sigma$. This
arc has to connect~$\Sigma$ with a decoration point
outside~$\Sigma$ (the case of strict inequality) or connects two components
(the case $|J| \geq 2$).
\end{proof}
\par
As a consequence we may associate with each subsurface~$\Sigma_j^{(i)}$ a subcategory
$\calV_i^{(j)}$ of~$\calV_i$, that jointly give an orthogonal decomposition of
$\calV_i$. We refer to~$\calV_i^{(j)}$ as the \emph{components} of~$\calV_i$.
\par
Recall the notion of principal components from Section~\ref{sec:BMQS}.
\par
\begin{cor} \label{cor:quotprinc}
Suppose that $(\calA_\bullet, Z_\bullet) \in \MStab^\circ(\calD^3_{A_n})$. Then
the quotient hearts $\ol{\calA}_i \subset\calV_{i}/\calV_{i+1}$ support a
stability condition in the principal component $\Stab^\bullet(\calV_{i}/
\calV_{i+1})$ of the stability
manifold of $\calV_{i}/\calV_{i+1}$ for all~$i=0,\ldots,L$.
\end{cor}
\par
\begin{proof}
In fact they belong to the one corresponding to the triangulation $\AS_i|_{\Sigma_i}$,
extending the notation of the previous proof in the obvious way.
\end{proof}
\par

\subsection{The $\bC$-action}
Extending the $\bC$-action from usual
stability conditions to multi-scale stability conditions is a crucial ingredient
for the subsequent plumbing construction.
\par
\begin{prop} \label{prop:Caction}
Recall that we suppose  $\calD = \calD^3_{A_n}$. Then there is an action of~$\bC$
denoted by $(\lambda,[\calA_\bullet,Z_\bullet]) \mapsto \lambda\cdot
[\calA_\bullet,Z_\bullet]
=: [\calA_\bullet',Z_\bullet']$ on $\MStab^\circ(\cD)$, such that
\begin{itemize}
\item [(i)]the collection of subcategories $\calV_i^{Z} = \calV_i^{Z'} =: \calV_i$
is preserved, 
\item[(ii)] $Z_i' \= e^{-\sqrt{-1}\pi \lambda} Z_i$, and
\item [(iii)] $\lambda \cdot (\ol{\calA}_i,\ol{Z}_i) = (\ol{\calA}'_i,\ol{Z}'_i)$
is the usual $\bC$-action on $\Stab(\calV_i/\calV_{i+1})$.
\end{itemize}
\end{prop}
Thanks to Proposition \ref{prop:Caction}, proven below, we define the
projectivized space of multi-scale stability conditions $\PMStab^\circ(\calD)
= \MStab^\circ(\calD)/\bC$. In this language note that retaining just the
filtration steps from a level~$j$ onward, i.e.\ the datum of tuples $(\calA_{\geq j},
Z_{\geq j})$ together with the ambient triangulated category $\calV^Z_{j-1}$
gives by definition an element in $\PMStab^\circ(\calV^Z_{j-1})$ with $L-j+1$ levels
below zero.
\par
The restriction  $\calD = \calD^3_{A_n}$ stems from two requirements in the proof.
First, we need finite type. Second, we need a way to lift tilts from quotient
categories to~$\calD$ itself. We have shown this for quiver categories in
\cite{BMQS} and isolate this step in the following notion.
\par
\begin{definition}
  Given a thick triangulated subcategory $\calV \subset \calD$, a heart~$\calA$
  of~$\calD$
and a simple object~$S$ in $\calA \setminus (\calV\cap\calA)$,
we call a heart $\calA'$ with $\ol{\calA} = \ol{\calA'} :=  \calA'/(\calV
\cap \calA')$ a \emph{convenient representative} with respect to (the forward
tilt at)~$S$ if $\calA'$ is $\calV$-compatible and if
\be \ext^1(T,S) := \dim(\Ext^1(T,S)) = 0\quad \text{for all simples} \quad T \in \calV \cap \calA'\,.
\ee
\end{definition}
It means that a simple tilt of $\calA$ at $S$ induces a simple tilt of $\ol{\calA}$ at $\ol{S}$.
\begin{lemma} \label{le:exconvrep}
Suppose $\calD = \DQ$. For every $\calV$-compatible finite heart~$\calA$ and
every simple~$S$ there exists a convenient representative~$\calA'$, which can be
obtained from~$\calA$ by a finite sequence of simple tilts at simples in $\calA
\cap \calV$.
\par
If $\calA'$ is a convenient representative for~$S$, then $\mu_S^\sharp \calA'$
is $\calV$-compatible and
\be \label{eq:quotfwdtilt}
\ol{\mu_S^\sharp \calA'} \= \mu_S^\sharp \ol{\calA'}\,.
\ee
\end{lemma}
\par
\begin{proof} See \cite[Proposition~5.8]{BMQS}. 
\end{proof}
In the following we will give an explicit procedure for finding a convenient
representative if $\calD = \calD^3_{A_n}$, that will be useful later.
\par
\begin{lemma}\label{lemm:constructconvenience}
Recall that we suppose  $\calD = \calD^3_{A_n}$. Let $\calA$ be a  heart compatible
with~$\calV$ and let $S_0 \in \calA \backslash \calA\cap\calV$ be simple. Then there
exist a (possibly empty) set of indecomposables $\{S_1,\dots, S_{1\dots m}, S'_1,\dots,
S'_{1\dots m'}\}\subset \calA\cap \calV$ explicitly defined in the proof below, such
that $(\mu^\sharp_{S'_{1\dots m'}}\cdots\mu^\sharp_{S'_1})(\mu^\sharp_{S_{1\dots m}}\cdots
\mu^\sharp_{S_1})(\calA)$ is a convenient representative of $\ol{\calA}$ with respect
to $S_0$.
\end{lemma}
\par
\begin{proof}
Observe that for any simple $S \in \calA$ the number of 
(isomorphism classes of) simples $T$ in $\calA$ satisfying $\ext^1(T, S) = 1$ 
is at most $2$, since $\calA=\rep(Q,W)$ with $(Q,W)$ a quiver of $A_n$-type or,
more generally, since it comes from a triangulation of a surface. If there are
no simples in $\calA\cap\calV$  with this property, then~$\calA$ is convenient
with respect to~$S_0$ and we are done. Otherwise, we fix a simple $S_1\in
\calA\cap\calV$ with 
\be \label{firstcond}\ext^1(S_1,S_0)  \= 1,
\ee
and we define $S_1,\dots, S_m$ as the maximal collection of simples 
in $\calA\cap\calV$,
with 
\be\label{seccond} \ext^1(S_{i}, S_{i-1})  \= 1 \quad \text{and}
\quad \ext^1(S_{i+1},S_{i-1}) \= 0
\ee 
for $i\geq 1$. Note that the second condition in~\eqref{seccond} singles out
exactly one among the two possible objects with non-trivial extension 
with $S_i$. Consequently the collection is uniquely specified for a given~$\calA$, and the
definition is well-posed. See Figure~\ref{Lemma:constructconvenience11} for an
example of the corresponding ext-quiver.
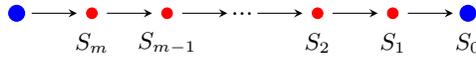
\begin{figure}[h]
	\begin{tikzpicture}
		\filldraw [blue] 
		(0,0) circle [radius=3pt]
		(6,0) circle [radius=3pt];

		\node at (1,-0.4) {\small $S_m$};
		\node at (2,-0.4) {\small $S_{m-1}$};
		\node at (3,-0) {$...$};
		\node at (4,-0.4) {\small $S_2$};
		\node at (5,-0.4) {\small $S_1$};
		\node at (6,0.4) {\small };
		\node at (6,-0.4) {\small $S_0$};
		
		\filldraw [red]
		(1,0) circle [radius=2pt]
		(2,0) circle [radius=2pt]
		(4,0) circle [radius=2pt]
		(5,0) circle [radius=2pt];
		
		\draw [-stealth](0.2,0) -- (0.8,0);
		\draw [-stealth](1.2,0) -- (1.8,0);
		\draw [-stealth](2.2,0) -- (2.8,0);
		\draw [-stealth](3.2,0) -- (3.8,0);
		\draw [-stealth](4.2,0) -- (4.8,0);
		\draw [-stealth](5.2,0) -- (5.8,0);
		
	\end{tikzpicture}
	\caption{(Partial) $\ext$-quiver containing the $A_{m+1}$-configuration of $S_0,S_1,\dots, S_m$ defined in the text. The small red dots correspond to simples in~$\calA \cap \calV$, while
	the big blue dots correspond to simples of $\calA$ not in $\calV$.}
	\label{Lemma:constructconvenience11}
\end{figure}

Tilting $\calA$ at 
$S_1$ produces a configuration of simples in $\mu^\sharp_{S_1}\calA\ni S_0$ such that the sequence 
of objects defined by \eqref{firstcond},\eqref{seccond} in $\mu^\sharp_{S_1}\calA$ has length $m-1$   
and consists on $S_{12},S_3,\dots, S_m$, as displayed in Figure \ref{Lemma:constructconvenience22} using the correspondence between simple tilts and mutations. 

\begin{figure}[h]
	\begin{tikzpicture}
		\filldraw [blue] 
		(0,0) circle [radius=3pt]
		(4.5,0.5) circle [radius=3pt];
		
		\node at (1,-0.4) {\small $S_m$};
		\node at (2,-0.4) {\small $S_{m-1}$};
		\node at (3,-0) {$...$};
		\node at (4,-0.4) {\small $S_{12}$};
		\node at (5,-0.4) {\small $S_1[1]$};
		\node at (4.5,0.9) {\small $S_0$};
		
		\filldraw [red]
		(1,0) circle [radius=2pt]
		(2,0) circle [radius=2pt]
		(5,0) circle [radius=2pt]
		(4,0) circle [radius=2pt];
		
		\draw [-stealth](0.2,0) -- (0.8,0);
		\draw [-stealth](1.2,0) -- (1.8,0);
		\draw [-stealth](2.2,0) -- (2.8,0);
		\draw [-stealth](3.2,0) -- (3.8,0);
		\draw [stealth-](4.2,0) -- (4.8,0);
		\draw [-stealth](4.1,0.1) -- (4.4,0.4);
		\draw [-stealth](4.6,0.4) -- (4.9,0.1);
	
	\end{tikzpicture}
	
	\caption{Mutation at $S_1$ of the $\ext$-quiver of Figure~\ref{Lemma:constructconvenience11}.}\label{Lemma:constructconvenience22}
\end{figure}
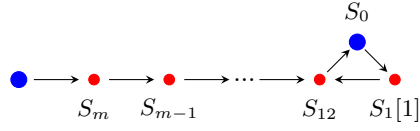

Tilting inductively at $S_{1\dots i}$ (recall the notation from~\eqref{S_ij}), 
the procedure leads to a configuration where such a sequence has length $0$,
as desired, see Figure~\ref{Lemma:constructconvenience3bis}. 
We define $\cl X = \bra S_{1\dots m},\ldots, S_1\ket$ so that
$\mu^\sharp_\torsionfreex:=\mu^\sharp_{S_{1\dots m}}\cdots\mu^\sharp_{S_1}$
by Proposition~\ref{prop:tiltviasimpletilt}.  By construction, $\torsionfreex
\subset\calA\cap\calV$ and $\ol{\mu^\sharp_\torsionfreex\calA}=
\ol{\calA}\subset\calD/\calV$.

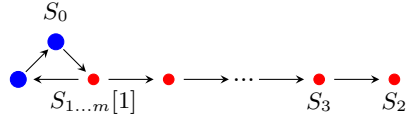
\begin{figure}[h]
	\begin{tikzpicture}
		(6,1) circle [radius=2pt];
		
		\filldraw [blue] 
		(0,0) circle [radius=3pt]
		(0.5,0.5) circle [radius=3pt];
		
		\node at (1,-0.3) {\small $S_{1\dots m}[1]$};
		\node at (3,0) {$...$};
		\node at (4,-0.3) {\small $S_3$};
		\node at (5,-0.3) {\small $S_2$};
		\node at (0.5,0.9) {\small $S_0$};

		\filldraw [red]
		(1,0) circle [radius=2pt]
		(2,0) circle [radius=2pt]
		(5,0) circle [radius=2pt]
		(4,0) circle [radius=2pt];
		
		\draw [stealth-](0.2,0) -- (0.8,0);
		\draw [-stealth](1.2,0) -- (1.8,0);
		\draw [-stealth](2.2,0) -- (2.8,0);
		\draw [-stealth](3.2,0) -- (3.8,0);
		\draw [-stealth](4.2,0) -- (4.8,0);
		\draw [-stealth](0.1,0.1) -- (0.4,0.4);
		\draw [-stealth](0.6,0.4) -- (0.9,0.1);
	\end{tikzpicture}
	\caption{The result of mutating at $S_1, S_{12},\dots, S_{1\dots m}$ the $\ext$-quiver of Figure~\ref{Lemma:constructconvenience11}.}\label{Lemma:constructconvenience3bis}
\end{figure}

Suppose now that there exists another simple $S_1'\neq S_1$
satisfying~\eqref{firstcond}. Proceeding in  the same way we define~$S_i'$
using~\eqref{firstcond} and~\eqref{seccond} and define inductively $S'_{1\dots i}$ for
$i=1,\dots,m'$ as above. They are not in 
$\bra S_1,\dots, S_{m+1}\ket$, due to the second condition in \eqref{seccond}. 
Then, with $\cl X' = \bra S'_{1\dots m},\ldots, S'_1\ket$,
$$\mu^\sharp_{\torsionfreex'}\mu^\sharp_\torsionfreex\left(\calA\right)=(\mu^\sharp_{S'_{1\dots m'}}\cdots\mu^\sharp_{S'_{12}}\mu^\sharp_{S'_1})(\mu^\sharp_{S_{1\dots m}}\cdots \mu^\sharp_{S_{12}}\mu^\sharp_{S_1})(\calA)$$ 
is a convenient representative of $\ol{\calA}$ with respect to $S_0$. 
\end{proof}
\par
\begin{proof}[Proof of Proposition~\ref{prop:Caction}]
Here and in many cases in the sequel, all aspects of the proof are
visible in the situation with just two levels, i.e., $L=1$, and for expository
simplicity we restrict to this case and if needed we mention briefly in the end
how to proceed by induction. Suppose $\Re\lambda \geq 0$, the other case is analogous.
\par
For a rescaling by $e^{-\pi i\lambda}$ with $\lambda \in i\bR$ and for a rotation
($\lambda \in \bR$) so that the phase of no
simple in $\calA_0$ with non-zero central charge exceeds $(0,1]$, we just apply~(ii)
to the multi-scale central charge and all the other conditions still hold. It thus
suffices to consider general rotations, i.e.\ $\lambda \in (0,1]$, repeating
the process $\lfloor \lambda \rfloor$ many times.
\par
We denote by $\torsionfree^1_\lambda\subset\calA_1$ the torsion-free class induced
by the action of~$\lambda$ on~$(\calA_1,Z_1)$. Similarly we
let~$\ol{\torsionfree}_\lambda\subset\ol{\calA_0}$ be the analogous torsion-free
class for the~$\lambda$-action on~$(\ol{\calA_0}, \ol{Z_0})$. We can decompose the
tilt at~$\ol{\torsionfree}_\lambda$ as a composition of tilts at finitely many
simple torsion-free classes $\ol{\torsionfree}_i=\bra\ol{X_i}\ket\subset\ol{\calA_0}$
according to Proposition~\ref{prop:tiltviasimpletilt}.
\par
The subcategory $\torsionfree^1_\lambda$ is a torsion-free class in $\calA_0$ and
we first forward-tilt $\calA_0$ and $\calA_1$ at $\torsionfree^1_\lambda$. Then we
inductively \lq\lq lift\rq\rq\ the simple tilt at $\ol{X_i}$ at
$\mu^\sharp_{\ol{\torsionfree}_{i-1}}\ol{\calA_0}^{(i-1)}$ (with $\ol{\calA_0}^{(0)}
=\ol{\calA_0}$) to the upper level in the following way. If $S=X_1$, we apply
Lemma~\ref{le:exconvrep} and forward tilt $\calA_0$ to arrive at a convenient
representative~$\calA'_0=\mu^\sharp_{\cl X}\calA_0$ for~$S$. Second, we tilt forward
at~$S$, and third, we perform the backward tilt at $\cl X[1]$. At the end
we arrive at a heart~$\calA_0''$ on which $Z_0$ is still well-defined, and
with the following properties:
\begin{itemize}[nosep]
\item[(1)] $\calV_0'' = \calV_0$, since after each of the three steps
the simples annihilated by~$Z_0$ generate the same category;
\item[(2)] the quotient heart $\ol{\calA}_0''$ coincides with $\mu_S^\sharp \ol{\calA_0}$, thanks to~\eqref{eq:quotfwdtilt};
\item[(3)] the intersection $\calA''_0 \cap \cal \calV_0 = \calA_1$, since the forward
and backward tilts cancel on~$\calA_1$.
\end{itemize}
Repeating this process for all~$i$,  in  
the end we change the central charge as required by~(ii). Using~(1) at
each step ensures~(i), and~(2) together with~(3) at each step ensure~(iii).
This procedure indeed defines an action of~$\bC$, since the equivalence class
of the multi-scale stability condition is uniquely determined by the
conditions (i)--(iii). It obviously agrees with the $\bC$-action
on $\stab(\calV_0/\calV_1)$.
\end{proof}
\par
For later use we record that the ``lifts'' of the tilts at ${\ol{\torsionfree}_{i}}$
used in the previous proof are actually tilts at explicit
torsion-free classes~$\calF_i$.
\par
\begin{lemma}\label{le:glueing1}
Recall that  $\calD = \calD_{A_n}$, let $\calA_\bullet=(\calA_0,\calA_1)$ and
$Z_\bullet = (Z_0,Z_1)$. For $\lambda\in\bR_{\geq 0}$ such that \begin{itemize}
		\item $\lambda\cdot (\ol{\calA_0},\ol{Z_0})=(\mu^\sharp_{S_0}\ol{\calA_0}, e^{-\pi  \lambda}\ol{Z_0})$ in $\stab^\circ(\calD/\calV)$, and
		\item there are no indecomposables in $\calA_1$ with phase $\phi_{\ol{Z_1}}$ less than or equal to $\lambda$, 
	\end{itemize}
the action by $\lambda$ on $[\calA_\bullet,Z_\bullet]$ gives $[\calA'_\bullet,Z'_\bullet]$ with nested hearts
\[\calA'_0 \=\mu^\sharp_\torsionfree\calA_0, \quad \calA'_1 \=\calA_1, \]
for $\torsionfree=\bra S'_{01\dots m'}, \dots S'_{01},S_{01\dots m}, \dots S_{01}, S_0\ket$,
using the same notation as in proof of Lemma~\ref{lemm:constructconvenience}.
Moreover $\torsionfree\subset\bra S_0,\calA_1\ket\setminus\calA_1$.
\end{lemma}
\par
\begin{proof}Suppose for simplicity $m>0$, $m'=0$ in the notation of
Lemma~\ref{lemm:constructconvenience}. We know that $\calA'_0=\mu^\flat_{\cl X[1]}
\mu^\sharp_{S_0}\mu^\sharp_{\cl X}\calA_0$ by the procedure described in the proof of
Proposition~\ref{prop:Caction}. Since 
\[\bra S_0,S_{01}, \dots, S_{01\dots m}, S_1,\dots, S_{1\dots m}\ket \= \bra S_1, S_{12},\dots,
S_{1\dots m}, S_0\ket\]
we deduce that 
\[ (\mu^\sharp_{S_{1\dots m}}\cdots\mu^\sharp_{S_{1}})  (\mu^\sharp_{S_{01\dots m}}\cdots
\mu^\sharp_{S_{01}}) \mu^\sharp_{S_0}(\calA) \= \mu^\sharp_{S_0}
(\mu^\sharp_{S_{1\dots m}} \cdots \mu^\sharp_{S_1})(\calA)\]
and hence $\mu^\sharp_\torsionfree (\calA)= \mu^\flat_{\cl X[1]} \mu^\sharp_{S_0}
\mu^\sharp_{\cl X}(\calA).$
\end{proof}
\par
We use the subsequent lemma as a preparation for
Proposition~\ref{prop:plumbassociative} below.
\par
\begin{lemma}\label{le:glueing_interm}
Let $\torsionfree_i$ for $i=1,\ldots,r$ be the torsion-free classes lifting the
classes $\ol{\torsionfree}_i \subset \ol{\calA_0}^{(i-1)}$ appearing in the proof of
Proposition \ref{prop:Caction} and explicitly described
by Lemma \ref{le:glueing1}. Then $\torsionfree_j \cap\torsionfree_{i}[1]
=\{0\}$ for any $j>i$. Moreover, for any $r' \leq r$, the result of the sequence
of forward-tilts at $\torsionfree_i$ of $\calA_0$ for $i=1,\ldots, r'$
is intermediate with respect to $\calA_0$.
\end{lemma}
\par
\begin{proof} By Lemma~\ref{le:glueing1} the class $\torsionfree_i$ is generated by
a simple object $X$ in $\calA^{(i-1)}$ together possibly with extensions of~$X$
with $\calA^{(i-1)}\cap \calV_1$. Similarly is $\torsionfree_{i-1}$ for an object $Y$,
with $\pi(X)\neq \pi(Y)$ in $\calD/\calV_1$, hence $X\neq Y$, and also $X\neq Y[1]$
(since the central charge defines a stability condition on the quotient). None of
the extensions of $X$ with $\calA_1$ can be in $\calV_1$, nor they can just be
extensions~of $Y[1]$ with $\calV_1$. Hence $\torsionfree_i \cap \torsionfree_{i-1}[1]
=\{0\}$ and $\torsionfree_{i-1}[1]\subset{^\perp{\torsionfree_i}}$. This is the
start for an inductive argument. In fact, we deduce that $\mu^\sharp_{\torsionfree_2}
\mu^\sharp_{\torsionfree_1}\calA_0\supset\torsionfree_2[1],\torsionfree_1[1]$. Now the
previous argument shows that $\torsionfree_3$ does not intersect
$\torsionfree_2[1]$ and $\torsionfree_1[1]$ non-trivially and we may proceed with the induction.
\par
The last part of the statement then follows from standard facts in tilting theory.
\end{proof}
\par

\section{The topology on the space of multi-scale stability conditions}

The goal of this section is to provide the space of multi-scale stability conditions
$\MStab^\circ(\calD^3_{A_n})$ with a natural topology so that the quotient by
autoequivalences acquires a complex structure (Section~\ref{subsec:complexstr})
and so that the further taking the quotient by the $\bC$-action gives a compact
space (Section~\ref{subsec:compact}) . 
The definition of neighborhoods is based on the \emph{plumbing} construction of 
Section~\ref{subsec:plumb} and explicitly given in Section~\ref{subsec:neigh}.
The name of the construction is derived from \cite{LMS}, see also
Section~\ref{sec:BCGGMforAn} where a plumbing construction is performed on Riemann
surfaces and where we will see that the constructions are analogous.
\par

\subsection{Plumbing of stability conditions}\label{subsec:plumb}
The plumbing construction takes as input a multi-scale stability condition
$(\cA_\bullet, Z_\bullet)$ and a collection $\bftau=(\tau_1,\ldots,\tau_L) \in -\bH^L$
and outputs an honest stability condition. More generally, we will
allow $\bftau$ to take the formal value $\tau_j = -i\infty$ for~$j$ in any fixed
subset $J \subset \{1,\ldots L\}$ and let $-\bH_\infty = -\bH \cup \{-i\infty\}$. The
result of the plumbing construction will then be a multi-scale stability condition
with~$|J|$ levels below
zero, the extreme case $\bftau = (-i\infty)^L$ being the identity, no plumbing at all.
Using all tuples  $\bftau$ with each entry of large (or infinite) imaginary part
and then allowing small deformations of the resulting generalized stability conditions
at each level will provide a neighborhood of $(\cA_\bullet,Z_\bullet)$.
\par
For simplicity we start with a multi-scale stability conditions
$(\cA_\bullet,Z_\bullet)$ with $L=1$ and let $\calV = \calV_1^Z$.
\par
\begin{prop} \label{prop:stabplumbing} Suppose that the representative
$(\cA_\bullet,Z_\bullet)$ of a multi-scale stability condition 
has precisely one level below zero. For $\tau \in -\bH$ there is a (honest)
stability condition $(\calA',Z') = \tau \ast (\cA_\bullet,Z_\bullet)$,
the \emph{plumbing of $(\cA_\bullet,Z_\bullet)$ with~$\tau$},
uniquely determined by the conditions
\begin{itemize}[nosep]
\item $(\calA' \cap \calV,Z'|_{K(\calV)}) = \tau \cdot(\calA_1,Z_1)$
for the usual $\bC$-action,
\item the quotient central charges agree $Z' = Z \in \Hom(K(\calD/\calV), \bC)$, \and
\item the hearts $\ol{\calA'} = \ol{\calA}_0$ coincide in $\calD/\calV$.
\end{itemize}
\end{prop}
\par
Note that the plumbing procedure depends on a chosen representative. For the
definition of the topology we will use that  the set
$\{ \tau \ast (\cA_\bullet,Z_\bullet), \, -\Im(\tau) > C\}$ for any fixed~$C$
does not depend on this representative, since the change of representative
results in translation of the corresponding~$\tau$ by a real number.
We use that 
\be \label{eq:decomposition}
K(\calA_0) = K(\ol{\calA_0}) \oplus K(\calA_1) \quad \text{given by} \quad
\Sim(\calA_0) = \Sim(\ol{\calA_0}) \,\coprod\, \Sim(\calA_1) ,
\ee
(see e.g.\ the survey \cite[Proposition~2.9]{psarou}) to define two projections
\be \pi_0:K(\calD)\simeq K(\calA_0)\to K(\ol{\calA_0}), \quad
\pi_1: K(\calD)\simeq K(\calA_0)\to K(\calA_1)\,.
\ee
Using these projections we combine central charges as
\bes
Z_0 \oplus Z_1 \,:= \,Z_0 \circ \pi_0  + Z_1 \circ \pi_1\,.
\ees
\par
\begin{proof} The heart of $\tau \cdot(\calA_1, Z_1)$ equals $\mu^\sharp_{\calF} \calA_1$
for some torsion-free class $\calF \subset \calA_1\subset \calA_0$. We take $\calA' =
\mu^\sharp_{\calF} \calA_0$ and by \cite[Lemma~5.6]{BMQS} the last condition holds. Since
$\mu^\sharp_{\calF} \calA_1 \subset \calA'$ and since~$\calA'$ is a finite heart thanks to
$\calD = \calD^3_{A_n}$ we can use the observation~\eqref{eq:decomposition}
to get the decomposition $K(\calA'/ \mu^\sharp_{\calF} \calA_1) \oplus K(\mu^\sharp_{\calF} \calA_1)$
and define $Z = Z_0 \oplus e^{-\pi i \tau} Z_1$. This is indeed a central
charge, since $Z(S) \in \ol{\bH}$ for all simples $S \in \calA_1$ by the hypothesis
on~$Z_0$ and~$Z_1$.
\end{proof}
\par
Next we generalize to the action of $\bftau = (\tau, -i\infty, \ldots, -i\infty)$
on a multi-scale stability condition $(\cA_\bullet,Z_\bullet)$. In this case we apply
Proposition~\ref{prop:stabplumbing} to the first two levels $( (\calA_0,
\calA_1),(Z_0,Z_1))$ and record as $\bftau$-image the tuple
\be
(\tau, -i\infty, \ldots, -i\infty) \cdot (\cA_\bullet,Z_\bullet) \= 
(\calA', \calA_2,\calA_3 \ldots,Z',Z_2,Z_3,\ldots)\,,
\ee
i.e. the top two levels have been merged to obtain a multi-scale stability
condition with $L-1$ levels below zero.
\par\medskip
This construction also gives a recipe for the \emph{plumbing
$\bftau \ast (\cA_\bullet,Z_\bullet)$
of a multi-scale stability condition $(\cA_\bullet,Z_\bullet)$
by a general~$\bftau \in -\bH_\infty^L$}. We take~$j$ to be  the highest index
with $\tau_j \neq -i\infty$. Then we apply the preceding construction to the
multi-scale stability condition $( \calA_{\geq j},Z_{\geq j})$ on $\calV_j^Z$ and iterate
with the action of the remaining coordinates $\bftau' = (\tau_1, \cdots,
\widehat{\tau_j}, \cdots)$. It will turn out that the plumbing procedure
is not quite independent of the order of the levels at which we perform the
plumbing step, only nearly so. The reason is that already one-level plumbing
and rotation are only nearly compatible. We need a quantitative version of
this fact.
\par
\begin{prop} \label{prop:plumbassociative}
Let $(\cA_\bullet,Z_\bullet)$ be a fixed representative of a multi-scale stability
condition with $L=1$. Let $\tau \in -\HH$, $\lambda\in\bC$, with 
\begin{equation}\label{realpart}
0\leq \Re\lambda,\quad 0 \leq \Re\tau, \quad \text{and}\quad
0\leq \Re (\lambda+\tau)<1.
\end{equation}
Then the hearts of the two stability conditions
\be \label{eq:bothhearts}
\widetilde{\sigma}:=(\widetilde{\calA},\widetilde{Z}):=\lambda \cdot
(\tau \ast (\cA_\bullet,Z_\bullet)) \quad \text{and} \quad
\widehat{\sigma}:=(\widehat{\calA},\widehat{Z})
:=\tau \ast (\lambda \cdot (\cA_\bullet,Z_\bullet))
\ee
are intermediate hearts with respect to~$\calA_0$, and the difference of the
central charges may be coarsely estimated by
\be \label{eq:Zdiffestimate1}
\left|(\widehat{Z}-\widetilde{Z})(S_j)\right| \,\leq \, \ell \cdot
|e^{-\pi i (\lambda+\tau)}|  \sum_{\substack{S_i\in\calA_1\\\text{simple}}} |Z_1(S_i)|
\ee
for any simple $S_j \in \Sim(\calA_0)$, where~$\ell$ is the number of classes
of indecomposables in~$K(\ol{\calA_0})$.
\end{prop}
\par
We recall from \cite[Proposition 7.4]{scattering} that the local homeomorphism
given by the forgetful map $\Stab(\calD) \to \Hom(K(\calD),\bC)$ is actually
injective when restricted to all hearts that are intermediate with respect
to a given heart~$\calA_0$, i.e., in $[\calA_0,\calA_0[1]]$. Consequently, to show
that $\widetilde{\sigma}$ and $\widehat{\sigma}$ are nearby it suffices to
estimate the differences of the central charges if $\widetilde{\calA},
\widehat{\calA} \in [\calA_0,\calA_0[1]]$.
\par
\begin{proof} The result of plumbing is given by definition as
$(\calA, Z):=\tau \ast (\cA_\bullet,Z_\bullet)$ with 
	\[\label{Z_plumbed}\begin{aligned}
		&\calA=\mu_{\torsionfree_\tau}^\sharp\calA_0\supset\mu_{\torsionfree_\tau}^\sharp\calA_1,\quad \torsionfree_\tau=\bra E\in \calA_1,\, Z_1\text{-semistable s.t. }\phi_{Z_1}(E)\leq \Re\tau\ket\\
		&Z := \ol{Z_0} \oplus e^{-i\pi\tau} \cdot Z_1\,.
	\end{aligned}\]	
We first consider the case where
\begin{itemize}
\item[($\star$)] there is exactly one isomorphism class $[S_0]\in K(\calA)$ of
a $Z$-stable object in $\calA$ with phase $0<\phi_Z(S_0)\leq \Re \lambda$, and
moreover $[S_0] \not\in K(\calV)$\,. 
\end{itemize}
(Note that $S_0$ must be simple in $\calA$.) In this case the stability condition
$\widetilde{\sigma}$ is given by $\widetilde{Z} =e^{-i\pi\lambda}Z$ and
$\widetilde{\calA} =\mu^\sharp_{S_0}\calA = \mu^\sharp_{\widetilde{\calF}}\calA_0$
with $\widetilde{\calF} = \langle \calF_\tau, S_0 \rangle$. In particular 
$\widetilde{\calA} \in [\calA_0,\calA_0[1]]$.
\par
On the other hand the heart $\widehat{\calA}$ is obtained by $\tau$-plumbing 
the multi-scale heart $(\calA_1 \subset \mu^\sharp_\torsionfree\calA_0)$,
where $\torsionfree=\bra S_0,S_{01},\dots\ket\subset \calA_0$ is explicitly
described in Lemma~\ref{le:glueing1}, since assumption~$(\star)$ implies that the
torsion-free class in~$\ol{\calA_0}$ of objects with phase $0<\phi_{\ol{Z_0}}\leq
\Re \lambda$ is generated by $\ol{S_0}$. Consequently,
\bes \widehat{\calA}\=\mu^\sharp_\cG \mu^\sharp_\torsionfree\calA_0, \quad
\text{where} \quad \cG \= \bra E\in \calA_1,\, Z_1\text{-semistable s.t. }\phi_{Z_1}(E)\leq \Re\tau\ket
\ees
Since $\cG \subset \mu^\sharp_\torsionfree\calA_0\cap\calV_1$ and
$\mu^\sharp_\torsionfree\calA_1 = \calA_1$ we deduce  $\cG \cap\torsionfree[1]
=\{0\}$ and consequently (by the same arguments as in Lemma~\ref{le:glueing_interm})
we deduce $\calA_0\leq \mu^\sharp_\torsionfree\calA\leq
\widehat{\calA}  = \mu^\sharp_\cG \mu^\sharp_\torsionfree\calA_0 \leq \calA_0[1]$
in the partial order from Section~\ref{sec:fundamental}.
\par
To compare central charges note that the plumbing procedure happens over two
different decompositions of $K(\calD)$: one induced by 	$K(\calA_0/\calA_1)
\oplus K(\calA_1)$, the other induced by $K(\mu_\torsionfree^\sharp\calA_0/\calA_1)
\oplus K(\calA_1)$. The change of basis
$K(\mu^\sharp_\torsionfree\calA_0) \simeq K(\calA_0)=K(\calA_0/\calA_1)\oplus
K(\calA_1)$ has the  form of a block lower-triangular matrix
\[ \left[\mu_{\torsionfree}\right]^{-1}=\begin{pmatrix} C_{n-k} & 0 \\
B_{1\ol{0}} & \mathbf{1}_k\end{pmatrix},	\]
where  the entries of $B_{1\ol{0}}=(b_{ij})_{ij}$ have  absolute value at most~$1$,
where we hardly control the block~$C_{n-k}$, and where the block~$\mathbf{1}_k$
expresses that $\calV$ is preserved. Using the projections $\pi_0$ and $\pi_1$
onto the summands $K(\calA_0/\calA_1)\oplus K(\calA_1)$ we find
\[\widetilde{Z} \=e^{-\pi i \lambda}\ol{Z_0}\circ\pi_0 + e^{-\pi i (\lambda+\tau)}
Z_1\circ\pi_1 \]
and
\ba \widehat{Z}&\= e^{-\pi i \lambda}Z_0\circ\widehat{\pi}_0
+ e^{-\pi i (\tau+\lambda)}Z_1\circ\widehat{\pi}_1\\
&\ =e^{-\pi i \lambda}Z_0\circ\pi_0 + e^{-\pi i (\tau+\lambda)}Z_1\circ(\pi_1+B_{1\ol{0}}\pi_0),
\ea
where we see $B_{1\ol{0}}$ as a map $K(\calA_0/\calA_1)\to K(\calA_1)$. 
Now consider the simples in $\calA_0$. For $S_j\in\calA_1$, the two expressions
agree.  For $S_j$ a simple of $\calA_0$ not in $\calA_1$, we find
\bes\label{almost_comm_1}
\left|(\widehat{Z}-\widetilde{Z})(S_j)\right|
\= \Bigl|\sum_{\substack{S_i\in\calA_1\\\text{simple}}}(e^{-\pi i (\lambda+\tau)}b_{ij})Z_1(S_i)\Bigr|
\,\leq\,  |e^{-\pi i (\lambda+\tau)}|
\cdot \sum_{\substack{S_i\in\calA_1\\\text{simple}}} |Z_1(S_i)|.
\ees
\par
We now drop the assumption ($\star$) and allow for multiple indecomposables
$X\in \calA$ with $0<\phi_Z(X)\leq \Re\lambda$ . 
\par 
The assumption~\eqref{realpart}
still guarantees that $\widetilde{\calA}\in\left[\calA_0,\calA_0[1]\right]$, and
$\widetilde{Z}=e^{-\pi i \lambda}\ol{Z_0}\circ\pi_0 + e^{-\pi i (\lambda+\tau)}Z_1
\circ\pi_1$, as before.
\par
The result of $\lambda \cdot (\calA_\bullet, Z_\bullet)$ is a multi-scale stability 
condition $\sigma'$ with nested hearts $\calA_1' \subset \calA_0'$ that can be 
explicitly obtained as in the proof of Proposition~\ref{prop:Caction} by performing
a forward-tilt at $\torsionfree_\lambda^1$ and  a sequence of forward-tilts at
torsion-free classes $\torsionfree_i$ described in  Lemma~\ref{le:glueing1} and
Lemma~\ref{le:glueing_interm}. At each step, at the 0 level, the matrix 
of the change of basis has the form of a block lower triangular matrix, with a 
block $B_{1\bar{0}}= \prod_i B_{1\bar{0}}^{(i)}$, whose entries have absolute 
value $\leq \ell'_i$, bounded by the number $\ell'_i$ of classes of indecomposables 
in $\torsionfree^1_\lambda$ or $\torsionfree_i$. Lemma~\ref{le:glueing_interm}
guarantees that the heart~$\calA_0'$ is intermediate with respect to $\calA_0$,
and so is the heart of $\tau*\sigma'$ thanks to assumption \ref{realpart}.
Similarly to \eqref{almost_comm_1}, we obtain
\[\label{almost_comm_2}
\left|(\widehat{Z}-\widetilde{Z})(S_j)\right| \leq  |e^{-\pi i (\lambda+\tau)}|
	\cdot \ell\sum_{\substack{S_i\in\calA_1\\\text{simple}}} |Z_1(S_i)|,
\]
where $\ell$ is the number of classes of indecomposables in $K(\calA_0)$. 
Last, the case $\Re\lambda=0$ is just easier, and the argument above shows
that in such a  case $\lambda\cdot(\tau\ast(\calA_\bullet,Z_\bullet))
=\tau\ast(\lambda\cdot(\calA_\bullet,Z_\bullet))$.	
\end{proof}
\par
\begin{rem} \label{rem:plumbpath}
The same observation shows that the plumbing (with fixed parameter
$\tau \in -\bH$) of a path~$\gamma \in \MStab^\circ(\calD^3_{A_n}) \setminus
\Stab^\circ(\calD^3_{A_n})$
is not continuous. Discontinuities occur when some simple (not at bottom level)
is tilted. However the size of the jumps decreases with~$|e^{-\pi i \tau}|$. More
precisely, suppose that $L=1$ and that~$\gamma$ is a path for which at precisely one
value~$t_0 \in [0,1]$ such a tilt occurs. Then the hearts of the two stability conditions
\bes
{\sigma}^+:=({\calA}^+,{Z}^+):=\lim_{t \to t_0^+}
(\tau \ast \gamma(t)) \quad \text{and} \quad
{\sigma}^-:=({\calA}^-,{Z}^-):=\lim_{t \to t_0^-}
(\tau \ast \gamma(t))
\ees
are intermediate hearts with respect to the top level heart~$\calA_0$ of
$\lim_{t \to t_0^-}\gamma(t_0)$
and the difference of the central charges may be coarsely estimated by
\be \label{eq:Zdiffestimate2}
\left|({Z}^+-{Z}^-)(S_j)\right| \,\leq \, \ell \cdot
|e^{-\pi i \tau}|  \sum_{\substack{S_i\in\calA_1\\\text{simple}}} |Z_1(S_i)|
\ee
for any simple $S_j \in \Sim(\calA_0)$, where~$\ell$ is the number of classes
of indecomposables in~$K(\ol{\calA_0})$. The proof is exactly the same as for
the previous proposition.
\end{rem}
\par
Let again $0 \leq \Re(\tau) < 1$ and decompose $\tau = \tau_R + i \tau_I$ into
its real and
imaginary part. We observe that the plumbing in Proposition~\ref{prop:Caction}
can be viewed a composition of three steps: First we apply the action of $\tau_R$,
resulting in a tilt at a torsion-free class $\calF \subset \calA_1$ and turning~$Z_1$
by $e^{-\pi i \tau_R}$ resulting in some other representative
$(\calA_\bullet^{\Re(\tau)}, Z_\bullet^{\Re(\tau)})$ of the multi-scale stability condition.
Second we rescale~$e^{-\pi i \tau_R}Z_1$ by $e^{\pi \tau_I}$, and finally we form
the direct sum~$Z$ and drop the lower levels to get an honest stability condition.
The observation that $[\mu^\sharp_\torsionfree]^{-1}$ in the previous proof
preserves $\calV_1$ and the first step just described does preserve~$\calV_1$ as
well, together imply the following corollary where we relax the bound for $\Re(\tau)$
appearing in Proposition~\ref{prop:plumbassociative}.
\par
\begin{cor} \label{cor:plumbassociative}
Let $(\cA_\bullet,Z_\bullet)$ be a fixed representative of a multi-scale stability
condition with $L=1$. Let $\tau \in -\HH$, $\lambda\in\bC$, with
$0 \leq \Re(\lambda) < 1$.
Then the hearts of the two stability conditions
\bes
\widetilde{\sigma}:=(\widetilde{\calA},\widetilde{Z}):=\lambda \cdot
(\tau \ast (\cA_\bullet,Z_\bullet)) \quad \text{and} \quad
\widehat{\sigma}:=(\widehat{\calA},\widehat{Z})
:=\tau \ast (\lambda \cdot (\cA_\bullet,Z_\bullet))
\ees
are intermediate hearts with respect to~$\calA_0^{\Re(\tau)}$ and the difference
of the central charges may be coarsely estimated as in~\eqref{eq:Zdiffestimate1}.
\end{cor}
\par
\begin{rem} For $\lambda \in 2\bZ$ and any $\tau \in -\bH$ plumbing and
the $\lambda$-action commute, i.e.\ the hearts~$\widetilde{\sigma}$ and
$\widehat{\sigma}$ from~\eqref{eq:bothhearts} agree.
\end{rem}
\par
\subsection{Neighborhoods in the space $\MStab^\circ(\calD^3_{A_n})$.}
\label{subsec:neigh}
We work here with a representative $(\cA_\bullet,Z_\bullet)$ of a multi-scale
stability condition and let 
$$\bfdelta = (\delta_1,\ldots,\delta_L)$$ be
a tuple of (small) positive real numbers where~$\delta_j$ will control the size of
plumbing of level~$j$. Moreover we fix a collection of positive real numbers
$$\bfeps \= (\ve_J) \quad  \text{for any} \quad J\subset \{0,\ldots,L\}.$$
The following
Definition ~\ref{def:bdneighborhood} captures the idea that neighborhoods of
boundary points consist of the multi-scale stability conditions, described
heuristically as follows. Suppose $L=1$ for simplicity. We write $\epsilon_0,\epsilon_1,\epsilon_{01}$ for $J=\{0\},\{1\},\{0,1\}$ respectively. Then
\begin{itemize}
\item we may plumb by~$\tau$ with large negative imaginary part (so
that the lower level stability condition $(\calA_1,e^{-\pi i \tau} Z_1)$ stays small
in size) and wiggle the result in $\Stab^\circ(\calD^3_{A_n})$ by a small amount (by
a size controlled by~$\ve_{01}$);
\item alternatively we may not plumb (i.e.\ $\tau = -i\infty$) and wiggle
in~$\Stab^\circ(\calV_1)$ and $\Stab^\circ(\calD/\calV_1)$ a bit (by sizes
controlled by~$\ve_i$) on level $i$, for $i=0,1$.
\end{itemize} 
\par
We say that a stability condition $(\cA_\bullet',Z'_\bullet)$ with $L' \leq L$ levels
below zero arises \emph{by plumbing of size at most~$\bfdelta$ from $(\cA_\bullet,
Z_\bullet)$} if there is $\tau \in (-\bH_\infty)^L$) with
$|e^{-\pi i \tau_j}| < \delta_j $ for $j=1,\ldots,L$ and $(\calA'_\bullet,Z'_\bullet) =
\tau \ast (\calA_\bullet,Z_\bullet)$. For a such a stability condition we denote
the new vanishing categories by~$\calV'_i$. Say levels in the interval~$J_i =
\{j_1,\ldots\} \subset \{0,\ldots,L\}$ have been plumbed to form the new level~$i$.
(This implies by definition that $\tau_{j_1} = -i\infty$.)
\par
We define the \emph{natural inner product} on $K(\calV'_i/\calV'_{i+1})^\vee =
\Hom(K(\calV'_i/\calV'_{i+1}),\bC)$ by using as an orthonormal basis
the basis $Z_{S_i}$ dual to the simples of $\calA_0$. (Note that this norm
depends on the heart~$\calA_0$ and its simples, but the norm around any other
multi-scale stability condition~$\sigma^\dagger = (\cA_\bullet^\dagger,Z^\dagger_\bullet)$
is comparable, scaling by a factor~$C = C(\sigma,\sigma^\dagger)$ given
by the operator norm of the identity map with respect to the to norms.)
\par
\begin{definition} \label{def:bdneighborhood}
We define the set $V_{\bfeps,\bfdelta}(\cA_\bullet,Z_\bullet)$ to be the set of all
multi-scale stability conditions $(\cA_\bullet'',Z''_\bullet)$ with $L'$ levels
below zero such that 
\begin{itemize}[nosep]
\item[(1)] there is a multi-scale stability condition $(\cA_\bullet',Z'_\bullet)$
with $L' \leq L$ levels that arises by plumbing of size at most~$\bfdelta$
from $(\cA_\bullet,Z_\bullet)$, and
\item[(2)] the multi-scale stability condition $(\cA_\bullet'',Z''_\bullet)$
is in a neighborhood of $(\cA_\bullet',Z'_\bullet)$ in
$\prod\stab^\circ(\cl V'_i/\cl V'_{i+1})$ which maps to the
product of $\ve_J$-balls on $K(\calV'_i/\calV'_{i+1})^\vee$ under the forgetful
map retaining just the quotients of the multi-scale central charges.
Here $J$ is the interval that is plumbed to produce level~$i$.
\end{itemize}
A \emph{neighborhood} of $(\cA_\bullet,Z_\bullet,)$ is a set in $\MStab^\circ
(\calD_{A_n})$ that contains
$V_{\bfeps,\bfdelta}(\cA_\bullet,Z_\bullet,)$ for some~$\bfeps$ and~$\bfdelta$.
\end{definition}
\par
This definition includes the case that $L=0$ and that $(\cA, Z)$ is an
honest stability condition, in which case the neighborhoods have to
contain the $\ep$-balls  in the norm with orthonormal bases
given by the simples of~$\cA$, since the deformation of 
stability conditions is locally controlled by the deformation 
of the central charge. This gives the second part of the 
following lemma.
\par
\begin{lemma} The system of neighborhoods given in Definition~\ref{def:bdneighborhood}
defines a topology on $\MStab(\calD^3_{A_n})$ whose restriction to
$\Stab(\calD^3_{A_n})$ is the usual topology where the forgetful map retaining
the central charge is a local homeomorphism.
\end{lemma}
\par
\begin{proof}
The only axiom whose verification is non-trivial is the following. 
Let $U$ be a neighborhood of $\sigma=[A_\bullet,Z_\bullet]$, in the sense of 
Definition \ref{def:bdneighborhood}. Then there is a smaller neighborhood $V$ 
of this point, such that $U$ is a neighborhood of each 
$\sigma^\dagger=[\calA^\dagger_\bullet,Z^\dagger_\bullet]$ in $V$. We 
continue with the case $L = 1$, the general case works with the same argument. 
By definition $U$ contains some $V_{\bfeps,\delta}(A_\bullet,Z_\bullet)$. 
The rough idea is to take $V = V_{\bfeps^*,\delta^*}(A_\bullet,Z_\bullet)$ 
for some $(\bfeps^*,\delta^*)$ smaller  than $(\bfeps, \delta)$ in each entry,
so that $U$ contains  $V_{(\bfeps-\bfeps^*)/C, \delta-\delta^*}(\sigma^\dagger)$, just as
if we'd be working plainly in vector spaces, where~$C = C(\sigma,\sigma^\dagger)\geq 1$
accounts for the change of basis in the definition of the norms.
We will prove that  $V_{(\bfeps-\bfeps^*)/C, \delta-\delta^*}(\sigma^\dagger)$ is indeed
contained in $V_{\bfeps,\delta}(\sigma)$ for~$\bfeps^*$ carefully chosen. We have to
avoid that $\ve_0^*$, $\ve_1^*$  are large compared to $\ve_{01}^*$ so that any
plumbing after deforming $(\calA^\dagger_\bullet,Z^\dagger_\bullet)$ of the order
of $\ve_0$, $\ve_1$ doesn't fail of being near $\sigma$.
\par	
We may thus first take the pair $(\ve_0^*, \ve_1^*)$ so small that any point 
in the $(\ve_0^*, \ve_1^*)$-ball in $\Stab^\circ(\calD/\cl V_1) \times
\Stab^\circ(\cl V_1)$  can be reached from $\sigma$ by a path $\gamma(t)$ involving
at most one tilt, at $t=0$. Suppose the chosen $\sigma^\dagger=(A_\bullet^\dagger,
Z_\bullet^\dagger)\in V_{\bfeps^*,\delta^*}(\sigma)$ has also $L = 1$ levels below zero
and let $\gamma^\dagger(t)$ be the straight path connecting $\sigma$ and
$\sigma^\dagger$. Let $\tau$ in $-\H$ of magnitude at most $\delta^*$, and 
$\hat \sigma=\tau*\sigma^\dagger$, $\sigma'=\tau*\sigma$. By the argument 
in Proposition \ref{prop:plumbassociative} and the one-tilt hypothesis, 
the hearts $\hat \calA$ and $\calA'$ are intermediate with respect to $\calA_0$ 
so the distance between $\hat \sigma$ and $\sigma'$ is controlled by their 
central charges. If the plumbing of the path $\gamma^\dagger$ is continuous, 
then 
\be
||\wh{Z} - Z'|| \leq \ve_0^* + \delta_1^*\ve_1^*\,.
\ee
In the general case of a single tilt use Remark~\ref{rem:plumbpath} and compare
$\sigma' = \sigma^-$ with $\sigma^+ := (\calA^+,Z^+) = \lim_{t \to 0^+}
\tau *\gamma^\dagger(t)$. Now~\eqref{eq:Zdiffestimate2} and the previous estimate
in the new notation give the rough estimate
\be
||\wh{Z} - Z^+|| \leq \ve_0^* + \delta_1^*\ve_1^*\,\quad \text{and}
\quad\,||Z^+ - Z'|| \leq \ell n \delta_1^*.
\ee
The triangle inequality shows that requiring moreover $\ve_{01}^* \leq
\ve_0^* + \delta_1^*(\ve_1^* + \ell n)$ does the job.
\end{proof}	
\par
The next lemma will be used in the proofs at the end of this section. We consider
a sequence $\{\sigma_j\}_j$ of multi-scale stability conditions in
$\MStab^\circ(\calD_{A_n}^3)$.
\par
\begin{lemma} \label{le:convafterturn}
Fix $\lambda \in \bC$ with $0<\Re(\lambda)<1$ and consider the $\bC$-action given in
Proposition~\ref{prop:Caction}. Then a sequence~$\{\sigma_j\}_j$ converges to $\sigma$
in $\MStab^\circ(\calD^3_{A_n})$
if and only if $\{\lambda \cdot \sigma_j\}_j$ converges to $\lambda \cdot \sigma$. 
\end{lemma}
\par
\begin{proof} Again we give the details in the case that $\sigma$ has $L=1$
levels below zero. Extracting subsequences we may assume that all $\sigma_j$
have the same number of levels below zero. If they are strict multi-scale
stability conditions, the claim follows from the corresponding statement in
$\Stab^\circ(\calV_1) \times \Stab^\circ(\calD^3_{A_n}/\calV_1)$. The interesting case
is that $\sigma_j \in \Stab^\circ(D^3_{A_n})$ for all~$j$.
\par
Convergence and the definition of neighborhoods implies that $\sigma_j$ is in
a open set $V_{\ep_j,\delta_j}(\sigma'_j)$, where $\sigma'_j = \tau_j * \sigma$
for some fixed representative of~$\sigma$ and for both $\ep_j \to 0$ and
$\delta_j = |e^{-\pi i \tau_j}| \to 0$ as $j \to \infty$.
We apply the action of~$\lambda$ to this sequence
and use Corollary~\ref{cor:plumbassociative} to see that
$\lambda \cdot (\tau_j * \sigma)$ is close to $\tau_j *  (\lambda \cdot \sigma)$
in a way controlled by~\eqref{eq:Zdiffestimate1}. We conclude that 
$\lambda \cdot \sigma_j$ is in an $\ep'_j$-ball of $\lambda\cdot(\tau_j
* \sigma)$ for $\ep'_j = \ep_j + \ell \delta_j$, which certifies convergence.
\end{proof}
\par
\medskip
\paragraph{\textbf{The Hausdorff property}} We now start proving that
$\MStab^\circ(D^3_{A_n})$ is a nice topological space. 
\par
\begin{lemma}\label{second_count} The space $\MStab^\circ(D^3_{A_n})$ is second
  countable.
\end{lemma}
\par
\begin{proof} The basis of neighborhoods consisting of $V_{\bfeps,\bfdelta}(
\calA_\bullet,Z_\bullet)$ with $(\calA_\bullet,Z_\bullet)$
such that the~$Z_i$ map the collection of simples to a (projectivized) tuple of
rational numbers, and with all entries of $(\bfeps,\bfdelta)$ being rational
obviously generates the same topology as the one using real numbers.
\end{proof}
\par
\begin{theorem} \label{thm:Hausdorff}
For any subgroup $G \subset \Aut^\circ(\calD^3_{A_n})$ the quotient 
$\MStab^\circ(\calD^3_{A_n})/G$ and the projectivized version $\PMStab^\circ(
\calD^3_{A_n})/G$ are Hausdorff topological spaces.
\end{theorem}
\par
\begin{proof}
Since the relevant (quotient) spaces are second countable, being Hausdorff is
equivalent to uniqueness of limits, which we now show. Suppose that the sequence
$\sigma_j = [\cA_{\bullet,j},Z_{\bullet,j}]$ of multi-scale stability condition converges
to $\sigma = [\cA_\bullet,Z_\bullet]$ and that the sequence $\sigma_j' = \Phi_j
[\cA_{\bullet,j},Z_{\bullet,j}]$, with $\Phi_j \in G$, converges to $\sigma' =
[\cA_\bullet',Z'_\bullet]$ in $\MStab^\circ(\calD^3_{A_n})$.  We need to show that
$[\cA'_\bullet,Z'_\bullet] = \Phi [\cA_\bullet,Z_\bullet]$ for some $\Phi \in G$.
We restrict our argument to the cases that all the $\sigma_j$ are honest stability
conditions, that $[ \cA_\bullet,Z_\bullet]$ has $L=1$ level below zero and that
$[\cA'_\bullet,Z'_\bullet]$ has $L' \in \{0,1\}$, leaving the inductive arguments to
treat the general case to the reader. Note that the mass $M_{\max}(\sigma_j)$ of
the longest and the mass $M_{\min}(\sigma_j)$ of the shortest stable object in
$\sigma_j$ is a notion that is invariant under the action of $\Aut^\circ
(\calD^3_{A_n})$.
\par
The case $L'=0$ is absurd, since this implies that $M_{\min}(\sigma_j)/
M_{\max}(\sigma_j)$ is bounded below, while the convergence to $\sigma$ implies
that this ratio tends to zero.
\par
In general, for a sequence $\sigma_j$ converging to $\sigma$ with $L=1$
and for some cut-off parameter~$M>1$, we say that a simple~$S$ is \lq\lq short\rq\rq\
if its mass is less that $1/M$ times the largest mass of a simple, and
\lq\lq long\rq\rq\ otherwise.
\par
In the case $L'=1$ consider the set of short stable objects in the
sequences~$\sigma_j$ and~$\sigma'_j$ respectively. By definition of the topology,
these short stable objects eventually (as $C \to \infty$) generate the vanishing
subcategories~$\calV$ and~$\calV'$. Consequently, $\Phi_j \calV = \calV'$
for $j \geq N$ for some~$N$ large enough. Replacing $\Phi_j$ by
$\Phi_N^{-1} \circ\Phi_j$ we may suppose from now on that $\calV = \calV'$ and
$\Phi_j \in G \cap \Aut^\circ(\calD^3_{A_n},\calV)$ for all $j\geq N$. Using the
triangle
inequality and the definition of the metric, it is easy to show that also the
sequence $\Phi_j (\sigma)$ converges to $\sigma'$ in $\MStab^\circ(\calD^3_{A_n})$.
Since all these objects are now in fact in $\MStab^\circ(\calD^3_{A_n},\calV)$, this
implies that $\Phi_j (\calA_0,Z_0) \to (\calA_0',Z_0')$ in $\Stab^\circ(\calD_{A_n}^3/
\calV)$ and $\Phi_j [\calA_1,Z_1] \to [\calA_1',Z_1']$ as projectivized stability
conditions in $\PP\Stab^\circ(\calV)$. Since $\Aut^\circ(\calD^3_{A_n},\calV)$ acts on
$\Stab^\circ(\calD^3_{A_n}/\calV)$
via $\pzAut_{\lift}(\calD/\calV)$, by Lemma~\ref{le:propdisc} the image
$\mathrm{Im}(G) \subset \Aut^\circ(\calD^3_{A_n}/\calV)$ acts  properly discontinuously
on $\Stab^\circ(\calD^3_{A_n}/\calV)$.  By definition the stabilizer in a neighborhood
of $(\ol{\calA_0},\ol{Z_0})$ is finite, hence after passing to a sub-sequence, we may
assume (again for $j\geq N$, which we assume now throughout) that $\ol{\Phi}_j \equiv
\ol{\Phi}^{(N)} \in \mathrm{Im}(G) \subset \Aut^\circ(\calD^3_{A_n}/\calV)$ with
$\ol{\Phi}^{(N)} (\ol{\calA_0},\ol{Z_0}) = (\ol{\calA_0'},\ol{Z_0'})$. 
\par
On lower level the stabilizer of $[\calA_1,Z_1]$ is not finite. Let~$H$ be the
stabilizer of the  un-projectivized $(\calA_1,Z_1)$. We recall that, by
Lemma~\ref{le:propdisc} the group~$H$ is a finite extension of $Z(B_{\rk(K(\calV))})$, 
which in turn acts trivially on the projectivized $[\calA_1,Z_1]$. This implies that,
at the cost of passing to a sub-sequence, there are  $z_j\in Z(B_{\rk(K(\calV))})$ such
that $\ol\Phi_j\big(z_j[\calA_1,Z_1]\big)$ is in fact the same converging sequence, 
but we can now work with their representatives in $\Stab^\circ(\calD^3_{A_n}/\calV)$
that have finite stabilizer. We can extract a convergent subsequence with
${\Phi_j}_{|\calV}(\calA_1,Z_1) \equiv (\calA_1',Z_1')$. Taken together, this means
that there is $N$ large enough so that $\Phi_{N}\sigma_\bullet=\sigma'_\bullet$. 
\par
The case $\PMStab^\circ(\calD^3_{A_n})$ is similar: one then has to correct also the
top level by a central element to ensure convergence to some~$\ol{\Phi^{(0)}}$.
\end{proof}
\par

\subsection{The complex structure on quotients of boundary neighborhoods in
$\MStab^\circ(\calD^3_{A_n})$} \label{subsec:complexstr}

Next we upgrade from a topology to a structure of complex orbifold. This will
not be possible on $\MStab^\circ(\calD^3_{A_n})$, but only the quotient by
the group of autoequivalences, see Section~\ref{sec:A2revisited}
for an illustration
in the case of the $A_2$-quiver. The first step in this direction is to
exhibit a subgroup $\Tw^s(\calV_\bullet)$ in the stabilizer of the boundary such
that the quotient $\MStab^\circ(\calD,\cl V_\bullet)/\Tw^s(\calV_\bullet)$
is a complex manifold. We then determine the full stabilizer 
of these boundary neighborhoods and exhibit the orbifold structure.
\par
We fix once for all a multi-scale stability condition $\sigma_\bullet=[\calA_\bullet,Z_\bullet]$ with $L$ levels below zero, and let $\calV_\bullet$ be the associated
sequence of nested vanishing subcategories of $\calD=\calD^3_{A_n}$.
\par
\medskip
\paragraph{\textbf{The simple twist group  $\Tw^s(\calV_\bullet)$}} 
Recall the numerical data associated with a multi-scaled stability conditions of
type $A_n$ from Section~\ref{sec:numdata} and suppose $\sigma_\bullet$ is of
type $\bfrho$. We focus at a level~$i$ and let $\calV_i^{(j)}$ be the components
of $\calV_i$. Suppose that $\calV_i^{(j)}$ has type~$n_{i}^{(j)}$, i.e.\ the heart
$\calV_i^{(j)}\cap\calA_i$ has $n_{i}^{(j)}$ simples $S_1,\ldots,S_{n_{i}^{(j)}}$
in the subset~$\calS = \calS(i,j)$ of $\Sim(\calA_0)$.
We recall the definition in Section~\ref{sec:STandBraid} of the group
elements~$\theta_{I,n}$. 
\par 
For $I$ denoting the closed arcs in the subsurface $\Sigma_i^{(j)}$ we let 
\be
\frakc_{i,j} = \begin{cases}
	\theta_{I,n} & \text{if $n_i^{(j)}$ is odd}\\
	\theta_{I,n}^2  & \text{if $n_i^{(j)}$ is even.}
\end{cases}
\ee
\par
From $\bfrho$ we derive another collection of integers $(\ell_i)_{i=1}^L$.
Recall that we define $\kappa_i^{(j)} = n_i^{(j)} + 3$ to be the number of marked
points on the boundary of $\Sigma_i^{(j)}$. We let
\be
\wh\kappa_i^{(j)} \= \begin{cases}
(n_i^{(j)} + 3)/2 & \text{if $n_i^{(j)}$ is odd} \\
(n_i^{(j)} + 3) & \text{if $n_i^{(j)}$ is even} \\
\end{cases}
\ee
(This notation  is consistent with the enhancements in
Section~\ref{sec:BCGGMforAn}.)
We define \be \ell_i \= \lcm\{ \wh\kappa_i^{(j)}\,,\, j=1,\ldots, s_i \}\,. \ee
For each level~$i$ and each~$j$ we now define the elements
	\be \label{eq:defstdtwistgens}
	\frakc_i := \prod_j  \frakc_{i,j}^{\ell_i/\wh\kappa_{i}^{(j)}}
	:= \prod_j \theta_{I(i,j),n}^{\ell_i/\kappa_i^{(j)}}. 
\ee
and we define the \emph{simple twist group} to be
$\Tw^s(\calV_\bullet) :=\langle \frakc_1,\dots, \frakc_L\rangle$.
\par
\begin{prop} \label{prop:nbdhstab}
For each level~$i$ the element $\frakc_i\in  \PB_{n}$
preserves the neighborhoods $V_{\bfeps,\delta}(\calA_\bullet,Z_\bullet)$ for
all $(\bfeps,\delta)$ small enough.
\end{prop}
\par
Implicit in the notation is that the elements $\frakc_{i,j}$ for fixed~$j$ commute.
In fact:
\begin{lemma}
The elements $\frakc_{i,j}$ for all~$(i,j)$ commute. In particular the \emph{simple
twist group} $\Tw^s(\calV_\bullet) = \langle \frakc_1,\dots, \frakc_L\rangle$ is
a free abelian group of rank~$L$.
\end{lemma}
\par
\begin{proof} For fixed~$i$ the elements $\frakc_{i,j}$ and $\frakc_{i,j'}$ commute
by~\eqref{eq:braidrel2} since they correspond to disjoint subsurfaces and hence
any two vertices, one in $I(i,j)$ and one in $I(i,j')$, cannot be connected
by an edge in the corresponding quiver. For different level indices $i < i'$, the
elements $\frakc_{i,j}$ and $\frakc_{i',j'}$ commute for the same reason if
$Q_{I(i',j')}$ is not a subquiver of  $Q_{I(i,j)}$. If it is a subquiver, the
elements commute since $\theta_{I,n}$ is (the image of) the central element in
the braid group corresponding to  $Q_{I(i,j)}$.
\end{proof}
\par
\begin{proof}[Proof of Proposition~\ref{prop:nbdhstab}]
Focusing on the subcategories above and below~$i$ we may reduce to the case
that $(\calA_\bullet,Z_\bullet)$ has $L=1$ and we consider $i=1$, thus writing
$\calV = \calV_1$. Suppose the connected components $\calV^{(j)}$ have type~$n_j$
and correspond to the subsurfaces $\Sigma^{(j)}$. We claim that 
\ba \label{eq:cjaction}
\frakc_{1,j} (\calA_1 \cap \calV^{(j)}Z_1|_{K(\calV^{(j)})})
&\= \wh\kappa_{1}^{(j)} \cdot (\calA_1 \cap \calV^{(j)},Z_1|_{K(\calV^{(j)})}), \\
\frakc_{1,j} Z_0 &\= Z_0 \in \Hom(K(\calD^3_{A_n}/\calV),\bC), \\
\frakc_{1,j} \ol{\calA}_0 &\= \ol{\calA}_0 \in \calD^3_{A_n}/\calV\,.
\ea
Granting the claim, we conclude that $\frakc_{1}$ acts like the shift by~$\ell_1$
on $\calA_1 \cap \calV^{(j)}$ for every~$j$, thus on the whole $(\calA_1,Z_1)$.
Since adding $\wh\kappa_{1}^{(j)}$
to~$\tau$ does not change the norm used in~(1) of the topology definition,
the elements stabilize the neighborhoods as claimed. Since $\tau+\wh\kappa_{1}^{(j)}$
realizes a plumbing of size at most~$\delta$, if $\tau$ does so, the elements
stabilize the neighborhoods (as defined by (1)-(2) in
Definition~\ref{def:bdneighborhood}) as claimed.
\par
The first equality of~\eqref{eq:cjaction} is \cite[Lemma~4.14]{seidelthomas} applied
to the subquiver $Q_{I(1,j)}$. The second equality holds by definition of
the induced actions on Grothen\-dieck groups, see~\eqref{eq:Phiinduced}.
For the third equality we write each spherical twist appearing in the
definition of $\frakc_{1,j}$ as a composition of tilts at simples in~$\calV$
using~\eqref{PhiSA}. The claim then follows since such tilts do not change
the quotient heart, see \cite[Lemma~5.6]{BMQS}.
\end{proof}
\par
\medskip
\paragraph{\textbf{The complex structure on boundary neighborhoods}}
Consider a neighborhood $V_{\bfeps,\delta}(\calA_\bullet,Z_\bullet)$ as given in
Definition~\ref{def:bdneighborhood}. As we stated after~\eqref{eq:cjaction}
the element $\frakc_i$ acts as the shift by~$\ell_i$ on the $i$-th level
of $(\calA_\bullet,Z_\bullet)$. Consequently, plumbing by $\bftau$
and plumbing by $\bftau + \ell_i e_i$ (where $e_i$ is the $i$-th unit vector)
give the same stability condition in the quotient $V_{\bfeps,\delta}(
\calA_\bullet,Z_\bullet)/
\Tw^s(\calV_\bullet)$. Therefore, on this quotient space the parameters
\begin{itemize}
\item $t_i = \exp\{2 \pi \sqrt{-1} \frac{\tau_i}{\ell_i}\}$ for $i=1,\ldots, L$,
\item the ratios of central charges of simples on $\PP\Stab^\circ(\calV_j/{\calV_{j+1}})$
for $j >0$,
\item the central charges of the simples on $\Stab^\circ(\calD^3_{A_n}/{\calV_{1}})$,
\end{itemize}
all together give a complex chart.
\par
Instead of using ratios of central charges we may equivalently fix
a representative $(\calA_\bullet,Z_\bullet)$ of the multi-scale stability condition
such that a (\lq\lq pivot\rq\rq) simple~$S_i$ in each $\calA_i$ for $i>0$ has $Z(S_i) = 1$
and use central charges of the remaining (non-pivot) simples in
$\Sim(\calA_i) \setminus \Sim(\calA_{i+1})$ together with the $t_i$ as coordinates.
\par\medskip
Next we check compatibility, i.e., we compare two charts defined around points
in $V_{\bfeps,\delta}(\calA_\bullet,Z_\bullet)/\Tw^s(\calV_\bullet)$ and $V'_{\bfeps',\delta'}
(\calA'_\bullet,Z'_\bullet)/\Tw^s(\calV'_\bullet)$ with $L\neq L'$, in the case of
non-trivial intersection of the neighborhoods. In particular we check compatibility
with the existing complex structure on $\Stab^\circ(\calD^3_{A_n})$. Fix
$\sigma=(\calA_\bullet,Z_\bullet)$ with $L=1$ and consider a point $\sigma' =
(\calA'_\bullet,Z'_\bullet) \in  V_{\bfeps,\delta}(\calA_\bullet,Z_\bullet)$ and suppose
first that $L=1$ and $\sigma' = (\calA,Z)$  is actually an honest stability condition.
We now use the definition of the plumbed central charge according to
Proposition~\ref{prop:stabplumbing} and the \lq\lq pivot\rq\rq\ viewpoint for coordinates.
We find that the central charge $Z_0'(S)$ of a simple $S \in \Sim(\calA)$ is equal to $Z_0(S)$
if $S \not\in \Sim(\calA_1)$ or equal to $t_1 Z_1(S)$ if $S \in \Sim(\calA_i)$. 
Since $t_1 \neq 0$ near~$\sigma'$, this coordinate change is a biholomorphism.
Similarly, the compatibility holds if $L>1$ and if the neighboring point~$\sigma'$
has $L'>0$ below zero, the coordinate change map being given by a product of
$t_j$'s times $Z_i(S)$ for  $S \in \Sim(\calA_i) \setminus \Sim(\calA_{i+1})$.
We summarize this discussion:
\par
\begin{prop} \label{prop:complexcharts}
The quotients of the boundary neighborhoods $V_{\bfeps,\delta}(\calA_\bullet,Z_\bullet)$
by the simple twist group $\Tw^s(\calV_\bullet)$ admit a complex structure compatible
with the complex structure around any point $(\calA'_\bullet,Z'_\bullet) \in 
V_{\bfeps,\delta}(\calA_\bullet,Z_\bullet)$ with less that~$L$ levels below zero.
\end{prop}
\par
\medskip
\paragraph{\textbf{The orbifold structure}} The following proposition
ensures that the complex structure on $\Tw^s(\calV_\bullet)$-quotients of neighborhoods
actually gives the structure of an orbifold on the quotient space
$\MStab^\circ(\calD^3_{A_n})/\pzAut(\calD^3_{A_n})$.
\par
\begin{prop} \label{prop:finTws}
The stabilizer in $\pzAut(\calD^3_{A_n})$ of a multi-scale stability
condition $\sigma_\bullet=( \calA_\bullet,Z_\bullet)$ contains $\Tw^s(\calV_\bullet)$ as
a finite index subgroup.
\end{prop}
\par
\begin{proof}
For the intended finiteness assertion we may restrict attention from
$\pzAut(\calD^3_{A_n})$ to the group of spherical twists $\ST(A_n) \cong B_{n+1}$ thanks
to~\eqref{eq:MCG1structure}. 
\par
Suppose that $L=1$ and suppose moreover that the lower level is connected.
Consider the intersection~$G$ of the stabilizer~$H_{\sigma_\bullet}$ with $\ST(A_n)$. We know that
any $\rho \in G$ stabilizes~$\calV$ and fixes the (lower level) stability condition
on~$\calV$ projectively. This implies as in the proof of Lemma~\ref{le:propdisc} 
that after passing to a finite index subgroup of $G=H_{\sigma_\bullet}\cap\ST(A_n)$
we may assume that $\rho|_\calV$ as an
element of $\Aut^\circ(\calV)$ is central, i.e.,  a power of $\theta_{I,n}$, since
the~$\bC$-action and the action of autoequivalences commute. Consequently
$\langle \frakc_1 \rangle = \Tw^s(\calV_\bullet) \subset G$ of finite index.
\par
Suppose still $L=1$ but now that the lower level has say~$k$ connected
components~$\calV^{(j)}$. Now the preceding argument implies that after passing to
a finite index subgroup~$G$ of the stabilizer (thereby getting rid of potential
non-trivial pointwise stabilizers of $(\calV_1, Z_1)$) any  $\rho|_{\calV^{(j)}}$ for
$\rho \in G$ is central in the autoequivalence group of each component
of~$\calV^{(j)}$, i.e.\ a power of the $\frakc_{1,j}$. We now use that moreover the
elements in~$G$ act by simultaneously rescaling the restrictions of $(\calV_1, Z_1)$
to the components~$\calV^{(j)}$ by definition of the equivalence relation of
multi-scale stability condition. We deduce that~$G$ is a cyclic group. 
Since the exponents in the definition of~$\frakc_1$ were chosen to 
projectivize simultaneously (raise the first equation of~\eqref{eq:cjaction} to
the right power), the claim follows in this case.
\par
For $L>1$ a new phenomenon occurs: Suppose $\calA_i \cap \calV_i^{(j)}
= \calA_{i+1} \cap \calV_{i+1}^{(j')}$ for certain components $\calV_i^{(j)}$
and $\calV_{i+1}^{(j')}$ of the vanishing subcategories at level~$i$ and~$i+1$.
This is possible if $Z_i|_{\calA_i \cap \calV_i^{(j)}} = 0$ and the required non-vanishing
of $Z_i$ is ensured on some other component~$\calV_i^{({k})}$ of~$\calV$.
Then the condition \lq\lq projectively equivalent\rq\rq\ in the definition of a multi-scale
stability condition imposes no constraint relating the action on 
$(\calA_i \cap \calV_i^{(j)},Z_i|_{\calV_i^{(j)}})$ and $(\calA_i \cap \calV_i^{({k})},
Z_i|_{\calV_i^{(k)}})$. We capture this problem as follows:
\par 
We write $G$ for the finite index
subgroup of the stabilizer that acts trivially on each $(\calV_i,Z_i)$.
Let $E$ be the set of (homotopy classes of) seams of the subsurfaces corresponding
to the components $\calV_i^{(j)}$ and identify $\bZ^E$ with the group generated
by the $\theta_{I(i,j),n}$, i.e., the group generated by the twist around these
seams. Then there is a natural embedding $G \to \bZ^E$. The image is contained
for each level~$i$ by $C_i-1$ constraints due to simultaneous projectivization,
where $C_i$ is the number of components where $\calA_i \cap \calV_i^{(j)}
\neq \calA_{i+1} \cap \calV_{i+1}^{(j')}$, i.e., where the seam of the subsurface
at level~$i$ does not agree with the seam of the subsurface at level~$i+1$.
Since $\sum_{i=1}^L C_i = E$ and since these constraints are obviously independent
we conclude that~$G$ is a free group of rank~$L$. Since
$\Tw^s(\calV_\bullet) \subset G$ is also free group of rank~$L$, this must be
an inclusion of finite index.
\end{proof}
\par
The group~$G$ appearing in the last paragraph of the 
proof should be called
the full twist group $\Tw(\calV_\bullet)$ in analogy with the full twist group
of level graphs appearing in \cite[Section~6]{LMS}, see also
Section~\ref{sec:BCGGMforAn}. The factor group $\Tw(\calV_\bullet)/\Tw^s(\calV_\bullet)$
is thus responsible for the orbifold structure of
$\MStab^\circ(\calD^3_{A_n})/\Aut(\calD^3_{A_n})$.
\par

\subsection{Compactness}\label{subsec:compact}

Our goal is:
\par
\begin{theorem} \label{thm:compact}
For any finite index subgroup $G \subset \Aut(\calD^3_{A_n})$ the quotient space
 $\PMStab^\circ(\calD^3_{A_n})/G$ is compact.
\end{theorem}
\par
\begin{proof}  Since $\PMStab^\circ(\calD^3_{A_n})/G$ is second countable, because $G$
is countable, compactness is
equivalent to being sequentially compact. Given a sequence $\sigma_m$ of stability
conditions we want to extract a convergent sub-sequence after rescaling the family
appropriately. Since the quotient space $\Stab^\circ(\calD^3_{A_n})/G$ has finitely many
chambers (given by undecorated triangulations
of the disc) and since the stability spaces of quotient categories
involved in $\PMStab^\circ(\calD^3_{A_n})/G$ have the same property (corresponding
to partial triangulations, see \cite{BMQS}) we may modify~$\sigma_m$ by suitable
elements of~$G$ and pass to a subsequence and assume that all elements of $\sigma_m$
belongs to a single chamber. We will moreover assume that $\sigma_m = [\calA,Z^{(m)}]
\in \bP\Stab(\calA) \subset \bP\Stab^\circ(\calD^3_{A_n})$
are honest stability conditions. At the end of the proof it will be clear that
the general case follows by the same argument, just using an extra index for the
levels of the initial multi-scale stability conditions.
\par
We define a (weak) full order~$\cgeq$ on $\Sim(\calA)$ by
\bes
S_1 \cgeq S_2 \quad \text{if} \quad \inf_{m \in \bN} |Z^{(m)}(S_1)| / |Z^{(m)}(S_2)| >0\,,
\ees
Equivalently, $S_1$ is strictly smaller than~$S_2$ if the ratio of central charges
tends to zero. Since there are finitely many simples, we may index the level sets
of this order by integers $0,1,\ldots,L$ and use these to generate Serre subcategories of $\calA$.
For consistence of indexing we assume that $\calA_L$ is generated by the set
of smallest simples (with respect to $\cgeq$), that $\calA_{L-1}$ is generated by
$\Sim(A_L)$ and the set of second smallest simples, etc., thus arriving at a nested
sequence
$$\calA_L \subset \calA_{L-1} \subset \cdots \subset \calA_1 \subset \calA_0 = \calA\,.$$
Let's order the simples of~$\calA$ so that $(S_1,\ldots,S_{r_0}) \not\in \Sim(\calA_1)$,
using implicitly the definition  $r_0 = \rk(K(\calA_0)) - \rk(K(\calA_1)) \geq 1$.
Since~$\bP^{n-1}$ is compact, we
may assume after passing to a sub sequence and choosing appropriated representatives
$(\calA, Z^{(m)})$ of the projectivized stability conditions that the sequence
$(Z^{(m)}(S_1),\ldots,  Z^{(m)}(S_n))$ converges, in fact to a point~$Z_0$ where
precisely the first~$r_0$ entries are different from zero by definition of~$\cgeq$.
Since $Z^{(m)}(S_i) \in \ol{\bH}$ we know that $Z_0(S_i) \in \ol{\bH}  \cup \bR_{>0} \cup
\{0\}$. Suppose that $Z_0(S_i) \in \bR_{>0}$ for none of the $S_i$. Then we iterate the
construction: We consider $(Z^{(m)}(S_{r_0+1}),\ldots,  Z^{(m)}(S_n)) \in \bP^{n-r_0-1}$
and use rescaling by real numbers and passage to a subsequence so that this
converges to a point $Z_1$. If again $Z_1(S_i) \in \bR_{>0}$ holds for none of the
$S_i \in \Sim(\calA_i)$ we continue to construct $Z_2, \ldots Z_L$, which we
consider as functions $Z_i: K(\calA_i) \to \bC$. It is now obvious that the tuple
$\sigma := (\calA_\bullet,Z_\bullet)$ is a multi-scale stability condition and
that $\sigma_n \to \sigma$ by definition of the plumbing procedure (in fact
plumbing with $\tau_i \in -i\bR$ purely imaginary suffices as the lower levels just
need to be rescaled appropriately).
\par
Finally we have to deal with the case that $Z_i(S) \in \bR_{>0}$ for some $S$ we excluded so
far. We'd like to rotate by some~$\lambda \in S^1$ and apply
Lemma~\ref{le:convafterturn}. We have to be careful since this rotation has
to be applied to $(\calA,Z^{(m)})$ (not just the central charge) and might
change the heart and alter our basic assumption. To circumvent this problem
we use that for $\calD = \calD^3_{A_n}$ the heart~$\calA$ has only finitely many stables.
Passing to a subsequence we may assume that there are only finitely many
problematic phases $P \in S^1$ that arise as limits of phases $\phi$ such
$\calP_m(\phi) \neq \emptyset$ for the slicing~$\calP_m$ associated with~$\sigma_m$ 
and that (possibly iteratively) tilting at the torsion pairs defined by those slices we 
stay in the same fundamental domain of $\PMStab(\calD^3_{A_n})$ 
with respect to the action of $G$.
Now we just apply the $\C$-action to the initial sequence for some~$\lambda$
such that $1 \not\in e^{i\lambda} P$ and run the argument of the previous paragraph.
\end{proof}
\par
Theorem~\ref{thm:Hausdorff}, Propositions~\ref{prop:complexcharts}
and~\ref{prop:finTws} together with Theorem~\ref{thm:compact}
complete the \emph{proof of Theorem~\ref{intro:main}}.

\section{The BCGGM-compactification and examples}
\label{sec:BCGGMforAn}

In this section we recall the main features and notions of the smooth
compactification (as orbifold or DM-stack) of the strata of abelian differentials
and quadratic differentials by multi-scale differentials, as constructed
in \cite{LMS} together with \cite{CoMoZa} (see also \cite{CGHMS} for the log
geometry viewpoint on this compactification). We focus on the
case of the stratum $Q_n$ corresponding to the $A_n$-quiver and
denote the multi-scale compactification by~$\ol{Q}_n$ and its projectivization
by $\bP\ol{Q}_n$. At the end of this section we will assert an isomorphism
\be \label{eq:isoofmulti}
\ol{K}_n : \C\backslash \ol{Q}_n/S_{n+1} \,\overset{\cong}{\longrightarrow} \, \bC \backslash \MStab^\circ(\cD^3_{A_n})/
\pzAut^\circ(\cD^3_{A_n})
\ee
of complex orbifolds  and give a sketch of proof.
\par
As complex variety the labeled version of the projectivized stratum~$\bP Q_n$
is isomorphic to the moduli space $\moduli[0,n+2]$ of pointed genus
zero surfaces and as such comes with its Deligne-Mumford compactification
$\barmoduli[0,n+2]$. This kind of compactification of quadratic differential
strata is available only for genus zero differentials. One of the goals
in this section is to explain why~$\bP\ol{Q}_n$ is \emph{not isomorphic}
to $\barmoduli[0,n+2]$. 
\par
\medskip
\paragraph{\textbf{Spaces of quadratic differentials}}
Let~$\w = (w_1,\ldots, w_r)$ be a tuple of integers $\geq -1$ and
let~$\w^- = (w_{r+1}, \ldots, w_{r+b}) $ be a tuple of
integers  $\leq -2$ in the quadratic case. Let $\qmoduli[g,r+b](\w,\w^-)$
be the moduli
space of quadratic differentials~$(X,\bfz,q)$ on a pointed curve $(X,\bfz)$
where $\bfz = (z_1,\ldots,z_{r+b})$ such that~$q$ has signature $(\w,\w^-)$.
In this space the critical points are labeled. The unlabeled version is denoted
by $\qmoduli[g](\w,\w^-)$, i.e., without the subscript. We abbreviate
\bes
Q_n = \qmoduli[0,n+2](1^{n+1},-n-3) \quad \text{and} \quad
Q_{[n]} = \qmoduli[0](1^{n+1},-n-3) = Q_n/S_{n+1}\,.
\ees
All these spaces come with their projectivized versions, the quotient by
the $\bC^*$-action, denoted by a letter~$\bP$ in front. Occasionally
we will compare with spaces of abelian differentials, denoted by
$\omoduli(\w,\w^-)$. For a fixed weighted DMS~$\sow$ as in Section~\ref{sec:BMQS}
we denote by $\FQuad(\sow)$ the \emph{moduli space of framed quadratic
differentials} $(X,\bfz,q,\psi)$ of signature $(\w,\w^-)$ with a (Teichm\"uller)
marking~$\psi$ of the real oriented blowup of~$X$ at the poles by the
surface~$\sow$. (We suppressed $\w^-$ in the notation.)
\par

\subsection{Enhanced level graphs} \label{sec:enhancedLG}

Recall (e.g.\ from \cite{acgh2}) that boundary strata of the Deligne-Mumford
compactification $\barmoduli$ are indexed by the \emph{dual graphs} of
the corresponding stable curves.
\par
\medskip
\paragraph{\textbf{Abelian case}} The first datum to characterize points in 
a boundary stratum of the multi-scale compactification of $\omoduli(\w,\w^-)$
is the following. An \emph{enhanced level graph $\wh\Gamma$
(for abelian differentials)} is the dual graph of a pointed stable curve
$(\wh{X},\wh\bfz)$, with a weak
total order on the vertices and a natural number $\kappa_e$, the
\emph{enhancement}, assigned with each edge. In particular $\wh\Gamma$ is
connected, unless specified
otherwise. The weak total order is usually given arranging the
vertices in levels, indexed by non-positive integers, the top level
being level zero. We call an edge \emph{horizontal} if it starts and
ends at the same level, and \emph{vertical} otherwise. We write
$E = E(\wh\Gamma) = E^h \cup E^v$ for this decomposition of the set of edges.
We require that $\wh\kappa_e = 0$ if and only if $e$ is horizontal. 
\par
The enhancement encodes the orders of zeros and poles of the collection
of differentials $\omega_v$ on the pointed stable curve $(\wh{X},\wh\bfz)$.
On the vertex $v \in \wh\Gamma$ the differential~$\omega_v$ is required to
have order~$w_i$, if the $i$-th marked point ($i=1,\ldots, r+b$) is
adjacent to~$v$. At (the node corresponding to) a horizontal edge~$e$ adjacent
to~$v$ the differential has a simple pole and the residues at the two ends
of~$e$ match, i.e., they add up to zero. At the upper end of a vertical edge~$e$
the differential has a zero of order $\wh\kappa_e-1$, at the lower end
a pole of order $-\wh\kappa_e-1$. In particular at each edge the orders
add up to~$-2$. Such a collection of differentials $\bfomega = (\omega_v)$
is called a \emph{twisted differential (of signature~$(\w,\w^-)$) compatible
with~$\wh\Gamma$} if moreover the global residue condition (GRC) from \cite{BCGGM1}
holds. An enhanced level graph comes with a \emph{vertex genus~$g_v$} for each
$v \in V$, defined by the requirement that $2g_v-2$ is the sum of the adjacent
zero and pole orders.
For each signature there is only a finite number of enhanced level graphs
(in particular a finite number of enhancements) for which the space of
twisted differentials on each vertex is non-empty. 
\par
\medskip
\paragraph{\textbf{Quadratic differentials}}
We can view the space of quadratic differentials inside the space
of abelian differentials (via the canonical cover construction) as a subspace
of surfaces with an involution, see e.g.\ \cite{CoMoZa, }.
Due to the involutions only some of the (abelian) enhanced level graphs
appear, encoded as follows. An \emph{enhanced level graph $\Gamma$
(for quadratic differentials)} is the dual graph of a stable curve
$(X,\bfz)$ with a level structure as above and enhancements $\kappa_e$ associated
with the edges $e \in E(\Gamma)$ with the only difference that we now
aim for \emph{twisted quadratic differentials} $\bfq = (q_v)$ compatible
with~$\Gamma$, which  comprises the vanishing according to the signature
$(\w,\w^-)$ at the points of~$\bfz$ and the following three conditions: at horizontal
edges $q_v$ should have a double pole with matching $2$-residues, at
vertical edges the order are $\kappa_e-2$ at the upper end and $-\kappa_e-2$
at the lower end, and the collection~$\bfq$ satisfies the global residue
condition. This condition (see \cite{kdiff} depends on a \emph{double cover
of enhanced level graphs $\wh{\pi}: \wh\Gamma \to \Gamma$}, which is a graph
morphism with the following conditions. Edges with even $\kappa_e$ have
two preimages with enhancement $\wh\kappa_e = \kappa_e/2$. Edges with odd
$\kappa_e$ have one preimage with enhancement $\wh\kappa_e = \kappa_e$.
The preimage of a vertex with an adjacent leg (marked point or edge)
that carries an odd label is a single vertex. The preimage of a vertex without
such an adjacent leg consists of two vertices, if the vertex genus is zero,
and one or two vertices otherwise. Here the vertex genus $g_v$ is defined
by the requirement that $2(2g-2)$ equals the sum of the adjacent zero and
pole orders. (All these conditions are necessary for $\wh\Gamma$ to be an enhanced
level graph compatible with a twisted differential $\bfomega = (\omega_v)$
on a cover, abusively also denoted by $\wh\pi: \wh{X} \to X$ which is on each
vertex~$v$ the canonical cover corresponding to $q_v$ and such that
$\wh\pi^* q_v = \omega_v^2$.) See \cite{kdiff} for an example where the
double cover is not uniquely determined by~$\Gamma$.
Again, for given~$\rho$ the number of enhanced
level graphs that allow a compatible~$\bfq$ is finite. Figure~\ref{cap:23together}
shows the double covers of enhanced level graphs for the boundary divisors where
two resp.\ three simple zeros have come together.
\begin{figure}[ht]
	\bes	
	\begin{tikzpicture}[
	baseline={([yshift=-.5ex]current bounding box.center)},
	scale=2,very thick,
	bend angle=30,
	every loop/.style={very thick}]
	\node[comp, fill] (T) [] {$g$}; 
	\node[comp, fill] (T-1) [below=of T] {};
	\path (T) edge [shorten >=4pt] (T-1.center);	
	\node[comp,fill] (B) [below=of T] {}
	edge 
	node [order top right] {$2$}
	node [order middle right] {$4$}
	node [order bottom right] {$-6$} (T);
	\node [text height=12pt,below right] (B-2) at (B.south east) {$1$};
	\path (B) edge [shorten >=4pt] (B-2.center);
	\node [text height=12pt,below left] (B-2) at (B.south west) {$1$};
	\path (B) edge [shorten >=4pt] (B-2.center);
	\end{tikzpicture}
	\quad\,\,\,\, {~\overset{\widehat{\pi}}{\longleftarrow}~}
        \quad	\,\,\,\,  
	\begin{tikzpicture}[	 
	baseline={([yshift=-.5ex]current bounding box.center)},
	scale=2,very thick,
	bend angle=30,
	every loop/.style={very thick}]
	\node[circle, draw, inner sep=2pt, minimum size=12pt,
	      label=0:{$\widehat{g}-1$}] (T) [] {}; 
	\node[comp, fill] (T-1) [below=of T] {};
	\path (T) edge [white,shorten >=4pt] (T-1.center);	
	\node[comp,fill] (B) [below=of T] {}
	edge [bend left] 
	node [order top left] {$1$} 
	node [order middle left] {$2$}
	node [order bottom left] {$-3$} (T) 
	edge [bend right]		
	node [order top right] {$1$} 
	node [order middle right] {$2$}
	node [order bottom right] {$-3$} 
	(T);
	\node [text height=12pt,below right] (B-2) at (B.south east) {$2$};
	\path (B) edge [shorten >=4pt] (B-2.center);
	\node [text height=12pt,below left] (B-2) at (B.south west) {$2$};
	\path (B) edge [shorten >=4pt] (B-2.center);
	\end{tikzpicture}
        \qquad \qquad
        	\begin{tikzpicture}[
	baseline={([yshift=-.5ex]current bounding box.center)},
	scale=2,very thick,
	bend angle=30,
	every loop/.style={very thick}]
	\node[comp, fill] (T) [] {$g$}; 
	\node[comp, fill] (T-1) [below=of T] {};
	\path (T) edge [shorten >=4pt] (T-1.center);	
	\node[comp,fill] (B) [below=of T] {}
	edge 
	node [order top right] {$3$}
	node [order middle right] {$5$}
	node [order bottom right] {$-7$} (T);
	\node [minimum width=.8cm,text height=10pt,below left] (B-2) at (B.south west) {$1$};
	\path (B) edge [shorten >=6pt] (B-2.center);
	\node [text height=14pt,below] (B-2) at (B.south) {$1$};
	\path (B) edge [shorten >=2pt] (B-2.center);
	\node [minimum width=.8cm,text height=10pt,below right] (B-3) at (B.south east) {$1$};
	\path (B) edge [shorten >=6pt] (B-3.center);	
	\end{tikzpicture}
	\quad {~\overset{\widehat{\pi}}{\longleftarrow}~} \quad	 
	\begin{tikzpicture}[	 
	baseline={([yshift=-.5ex]current bounding box.center)},
	scale=2,very thick,
	bend angle=30,
	every loop/.style={very thick}]
	\node[circle, draw, inner sep=2pt, minimum size=12pt,
          label=0:{$\widehat{g}-1$}] (T) [] {}; 
	\path (T) edge [shorten >=4pt] (T-1.center);	
	\node[circle, draw, inner sep=0pt, minimum size=12pt] (B) [below=of T] []{$1$}
	edge 		
	node [order top right] {$4$} 
	node [order middle right] {$5$}
	node [order bottom right] {$-6$} 
	(T);
	\node [minimum width=.8cm,text height=10pt,below left] (B-2) at (B.south west) {$2$};
	\path (B) edge [shorten >=6pt] (B-2.center);
	\node [text height=14pt,below] (B-2) at (B.south) {$2$};
	\path (B) edge [shorten >=2pt] (B-2.center);
	\node [minimum width=.8cm,text height=10pt,below right] (B-3) at (B.south east) {$2$};
	\path (B) edge [shorten >=6pt] (B-3.center);
	\end{tikzpicture}	
	\ees 
        \caption{Level graphs of two (left) resp.\ three (right) zeros coming
          together and their double covers.
          Simple zeros and the pole of order $-n-5$ on top level are
          omitted. The boxed numbers are the~$\kappa_e$.}
        \label{cap:23together}
\end{figure}
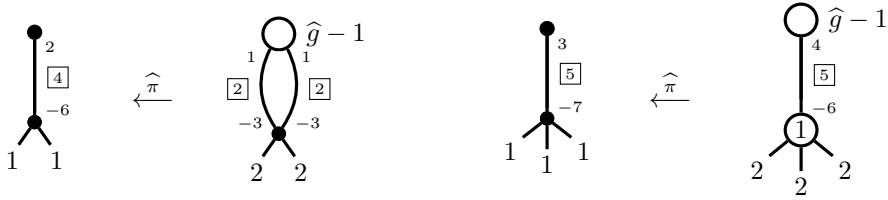
\par
\medskip
\paragraph{\textbf{Adjacency of boundary strata}} For an enhanced level graph
$\wh\Gamma$ we denote by $D_{\wh\Gamma}^\circ$ the open boundary stratum of multi-scale
differentials (defined below) compatible with~$\wh\Gamma$. The boundary strata
contained in the closure $D_{\wh\Gamma}$ of $D_{\wh\Gamma}^\circ$ can be described by the
process of \emph{degeneration}, or more easily starting with the converse process of
\emph{undegeneration}. Note that $D_{\wh\Gamma}$ is in general not irreducible:
the connected components of strata of meromorphic differentials that make up the
twisted differential, or more generally components of strata with residue conditions,
are one source that can create irreducible components 
\par
A \emph{horizontal undegeneration} selects a subset~$H$ of the horizontal edges
and contracts them. This results in a morphism~$\delta^H$ of enhanced level graphs.
To define the $i$-th \emph{vertical undegeneration} view the $i$-th level
passage as a line in the level graph just above level~$-i$ and contract all
the edges crossing that level passage. Again this results in a morphism
$\delta_i$ of enhanced level graphs. This can obviously be generalized for
any subset $I = \{i_1,\ldots,i_n\}$ of the set of levels to yield a
graph contraction map~$\delta_I$. These two notions of undegeneration commute
and a general undegeneration is
a composition of the two. A degeneration of level graphs is the inverse procedure.
\par
The complex codimension of a boundary stratum given by a level graph~$\wh\Gamma$
with $L$ levels below zero and $h$ horizontal edges is $h+L$.
\par
\medskip
\paragraph{\textbf{Boundary strata of~$Q_n$}} For these type
of strata the level graphs are strongly constrained.
\par
\begin{lemma} For the spaces $Q_n$ i.e.\ with type $(\w,\w^-) = (1^{n+1}, -n-3)$
the level graphs are trees without horizontal edges and with all vertex genera
$g_v = 0$. In particular for boundary strata of~$Q_n$ the graph~$\Gamma$
determines the double cover~$\wh\Gamma$.
\end{lemma}
\par
\begin{proof} For the statement about horizontal edges, undegenerate all but
one horizontal edge and all levels. Now note that each top level vertex must
have at least one pole of order~$\geq 2$ or positive genus. For the second
statement, the only ambiguity
for~$\wh\Gamma$ given~$\Gamma$ is the \lq\lq criss-cross\rq\rq\ (\cite[Example~4.3]{kdiff}),
which requires $\pi_1(\Gamma) \neq \{e\}$.
\end{proof}
\par

\subsection{Examples} \label{sec:examplesA2A3}

We list the boundary components of $\bP\ol{Q}_n$ and their adjacency in
the two examples of lowest complexity.
\par
\medskip
\paragraph{\textbf{The $A_2$-quiver}}
The projectivized space $\bP\ol{Q}_2$ is a smooth compactification
of $\moduli[0,4]$. Since the construction  introduces no orbifold structure
in codimension one (see~\cite[Section~6]{LMS}), it agrees with $\barmoduli[0,4]$.
The three boundary points correspond to the two-level graph with one edge, two
vertices, and the pole together with one of the three simple zeros on top level.
The action of~$S_3$ permutes the three boundary points.
\par
\medskip
\paragraph{\textbf{The $A_3$-quiver}}
In this case the projectivized labeled
space $\bP\ol{Q}_3$  is a surface, the largest dimension that can be visualized
on a piece of paper. There are three types of boundary divisors, namely
\bes
D_{1}=\left[
\begin{tikzpicture}[
baseline={([yshift=-.5ex]current bounding box.center)},
scale=2, 
very thick,
bend angle=30]
\node[comp,fill] (T) [] {};
\node [order node dis,above left] (T-1) at (T.north east) {$-8$};
\path[] (T) edge [shorten >=5pt] (T-1.center);
\node [minimum width=18pt,below right] (T-2) at (T.south east) {$1$};
\path (T) edge [shorten >=5pt] (T-2.center);
\node[comp,fill] (B) [below=of T] {}
edge 
node [order bottom left] {$-7$} 
node [order top left] {$3$} (T);
\node [minimum width=18pt,below right] (B-2) at (B.south east) {$1$};
\path (B) edge [shorten >=5pt] (B-2.center);
\node [text height= .45cm,below] (B-2) at (B.south) {$1$};
\path (B) edge [shorten >=3.5pt] (B-2.center);
\node [minimum width=18pt,below left] (B-2) at (B.south west) {$1$};
\path (B) edge [shorten >=5pt] (B-2.center);
\end{tikzpicture}
\right],
\quad 
D_2=\left[
\begin{tikzpicture}[
baseline={([yshift=-.5ex]current bounding box.center)},
scale=2,very thick,
bend angle=30,
every loop/.style={very thick},
comp/.style={circle,black,draw,thin,inner sep=0pt,minimum size=5pt,font=\tiny},
order bottom left/.style={pos=.05,left,font=\tiny},
order top left/.style={pos=.9,left,font=\tiny},
order bottom right/.style={pos=.05,right,font=\tiny},
order top right/.style={pos=.9,right,font=\tiny},
order node dis/.style={text width=.75cm}]
\node[comp,fill] (T) [] {};
\path (T) edge [shorten >=5pt] (T-1.center);
\node [order node dis,above left] (T-1) at (T.north east) {$-8$};
\node [minimum width=18pt,below right] (T-2) at (T.south east) {$1$};
\path (T) edge [shorten >=5pt] (T-2.center);
\node [minimum width=18pt,above right] (T-2) at (T.north east) {$1$};
\path (T) edge [shorten >=5pt] (T-2.center);
\node[comp,fill] (B) [below=of T] {}
edge 
node [order bottom left] {$-6$} 
node [order top left] {$2$} (T);
\node [minimum width=18pt,below right] (B-2) at (B.south east) {$1$};
\path (B) edge [shorten >=5pt] (B-2.center);
\node [minimum width=18pt,below left] (B-2) at (B.south west) {$1$};
\path (B) edge [shorten >=5pt] (B-2.center);
\end{tikzpicture}\right],
\quad
D_3=\left[
\begin{tikzpicture}[
baseline={([yshift=-.5ex]current bounding box.center)},
scale=2,very thick,
bend angle=30,
every loop/.style={very thick},
comp/.style={circle,black,draw,thin,inner sep=0pt,minimum size=5pt,font=\tiny},
order bottom left/.style={pos=.05,left,font=\tiny},
order top left/.style={pos=.9,left,font=\tiny},
order bottom right/.style={pos=.05,right,font=\tiny},
order top right/.style={pos=.9,right,font=\tiny},
order node dis/.style={text width=.75cm}]
\node[comp,fill] (T) [] {};
\path[draw] (T) edge [shorten >=5pt] (T-1.center);
\node [order node dis,above left] (T-1) at (T.north east) {$-8$};
\node[comp,fill] (B) [below left= 1cm and 0.4cm of T] {}
edge 
node [order bottom left] {$-6$}
node [order top left] {$2$} (T); 
\node[comp,fill] (C) [below right= 1cm and 0.4cm of T] {}
edge 
node [order bottom right] {$-6$}
node [order top right] {$2$} (T);
\node [minimum width=18pt,below right] (B-2) at (B.south east) {$1$};
\path (B) edge [shorten >=5pt] (B-2.center);
\node [minimum width=18pt,below left] (B-2) at (B.south west) {$1$};
\path (B) edge [shorten >=5pt] (B-2.center);
\node [minimum width=18pt,below right] (C-2) at (C.south east) {$1$};
\path (C) edge [shorten >=5pt] (C-2.center);
\node [minimum width=18pt,below left] (C-2) at (C.south west) {$1$};
\path (C) edge [shorten >=5pt] (C-2.center);
\end{tikzpicture}\right]\,.
\ees

The dual graphs are associated with labeled stable curves and this requires
labeling the simple zeros. We distinguish the enhanced graph further by
remembering the one (case $D_1$) or two (case $D_2$) simple zeros on top level,
or the grouping in pairs at the end of each cherry (case $D_3)$, see also
Figure~\ref{fig:A3conf} where each of these boundary divisors occurs.
\par
\begin{figure}
	\includegraphics[scale=.8]{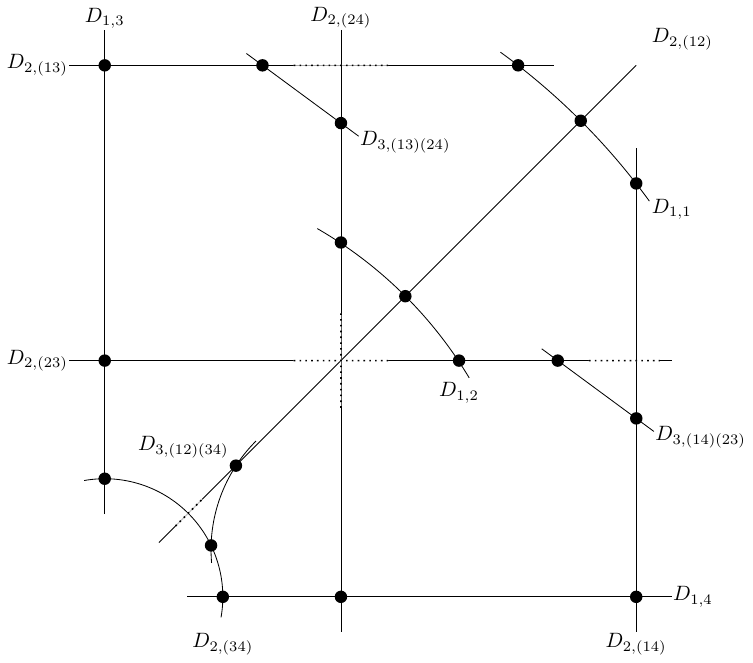}
	\caption{The boundary of the stratum $\bP\ol{Q}_3 =
          \bP \qmoduli[0,5](1^4,-8)$
        } \label{fig:A3conf}
\end{figure}
The codimension two strata are thus given by \lq\lq slanted cherries\rq\rq\ (one
of the lower ends of the level graph~$D_3$ pushed down further to level~$-2$),
which are the intersection point of a $D_3$-divisor and a $D_2$-divisor, and
a chain over three levels, with the pole and one zero on top, one on middle and two
on bottom level, giving an intersection point of $D_1$ and $D_2$. It is
easy to check that the boundary strata are all irreducible here, as depicted
in Figure~\ref{fig:A3conf}.
\par
The action of~$S_4$ is by permutation of the marked zeros and thus on boundary
strata by the natural permutation action on the additional indices of
each boundary divisor type~$D_i$.
\par
\medskip
\paragraph{\textbf{The Deligne-Mumford compactification.}}
For comparison recall that boundary divisors of $\barmoduli[0,5]$ are in
bijection with $2$-element subsets of $\{1,\ldots,5\}$. This shows that
for the $A_3$-quiver the boundary divisors of $\barmoduli[0,n+2]$ are in
bijection with those of type~$D_1$ and~$D_2$. There is a natural forgetful
map $\bP\ol{Q}_3 \to \barmoduli[0,5]$ that contracts the divisors of type~$D_3$.
The existence of \lq\lq cherry shaped\rq\rq\ divisors like~$D_3$ shows that 
$\bP\ol{Q}_n$ is not isomorphic to $\barmoduli[0,n+2]$ for any~$n \geq 3$.
See \cite[Section~7]{CGHMS}  for more on this birational map.
\par

\subsection{Multi-scale differentials} \label{sec:MultiScale}

We now give the key definition and explain the remaining terminology subsequently.
A \emph{quadratic multi-scale differential of type} $(\w,\w^-)$ on a stable
pointed curve $(X,\bfz)$ consists of
\begin{itemize}
\item[(i)] an enhanced level structure on the dual graph~$\Gamma$ of~$(X,\bfz)$,
\item[(ii)] a twisted quadratic differential~$\bfq = (q_v)_{v \in V(\Gamma)}$
of type~$(\w,\w^-)$  compatible with the enhanced level structure,
\item[(iii)]  and a prong-matching~$\wp$ for each node of~$X$ joining components
  of non-equal level.
\end{itemize}
Two quadratic multi-scale differential are considered equivalent if
they differ by the action of the level rotation torus.
\par
We make the same definition for \emph{abelian multi-scale differentials},
skipping the word \lq\lq quadratic\rq\rq\ everywhere, replacing $\bfq$ by $\bfomega
= (\omega_v)$ and applying the abelian conventions
for enhanced level graphs. To motivate the notion of \lq\lq prong-matching\rq\rq\
and \lq\lq level rotation torus\rq\rq\ we start with the 
\par
\begin{proof}[Proof of the isomorphism~\eqref{eq:isoofmulti}, Part I]
Our main goal is to define $\ol{K}_n^{-1}$ as a map of sets, starting with
a multi-scale stability
condition. If $\sigma = [\calA_\bullet, Z_\bullet]$ is an honest stability
condition, we associate with it a quadratic differential using the
Bridgeland-Smith isomorphism recalled in Theorem~\ref{thm:BS15_iso}.
\par
Suppose from now on that $\sigma$ is a strict multi-scale stability
condition and suppose that the number of levels below zero is $L=1$,
leaving the bookkeeping for larger~$L$ to the reader. By
Lemma~\ref{lem:possibleV} we may associate with~$\calV = \calV_1^Z$ a type
$\rho = (n_1,\ldots,n_{|J|})$ where~$J$ is an index set for the components
of~$\calV$. We associate with~$\sigma$ the level graph~$\Gamma$ consisting of
a tree with one vertex on top level (carrying the unique pole)
and~$|J|$ vertices on bottom level, each of them carrying $n_j+1$ markings
for simple zeros and enhancement $\kappa_j = n_j+3$.  As part of the
bijectivity claim for the map $\sigma \mapsto (X,\bfz,\Gamma, \bfq, \wp)$
we are about to construct, we observe that all possible level graphs
with~$L=1$ for quadratic differentials of type~$A_n$ arise in this way.
We now apply the isomorphism from Theorem~\ref{thm:KrhoBihol}
(for the quotient category~$\calD/\calV$) to the stability condition
$(\ol{\calA}_0,\ol{Z_0})$ on top level. We get the complex
structure of the
irreducible component of~$(X,\bfz)$ corresponding to the top level vertex~$v_0$
together with the quadratic differential~$q_{v_0}$ on this components.
Similarly we apply this isomorphism (for each $\calV^j)$ to each
stability condition $( \calA_1 \cap \calV^{j}, Z_1|_{\calV^j})$
on lower level to get the complex structure and the quadratic differential
corresponding to the vertices $v_j$ of~$\Gamma$ on lower level. (For $L=1$ there
is no further quotient, so we can as well apply Theorem~\ref{thm:BS15_iso}
on lower level.) The enhancements of~$\Gamma$ were chosen so that the
collection $\bfq = (q_v)_{v \in V(\Gamma)}$ is indeed a twisted differential
compatible with the enhanced level structure. We let~$\bfz$ be the unordered
tuple of zeros and poles (different from the nodes of~$X$) of the various
differentials~$q_v$. (We have no canonical way to label the points~$\bfz$,
and this fits with our target being the $S_{n+1}$-quotient of $\ol{Q}_n$.)
\par
We need to be more precise about automorphisms in the application of
Theorem~\ref{thm:KrhoBihol} (or Bridgeland-Smith) at each level. In fact, above
we were using that this isomorphism is equivariant with respect to the action
of the mapping class group $\MCG(\surf_{\Delta})$ on the domain and of the
group $\mathpzc{Aut}^\circ(\cD)$ on the range (see \cite[Theorem~7.2]{BMQS}).
So far we have given a well-defined map $\sigma \mapsto (X,\bfz,\Gamma,
\bfq)$ with~$\bfz$ considered up to the $S_{n+1}$-action.
\end{proof}
\par
The missing notions will be motivated by making this map bijective thanks
to the prong~$\wp$ and well-defined on equivalence classes. First observe that
the above assignment depended on the stability conditions up to
$G := \mathpzc{Aut}_\lift^\circ(\cD/\cV) \times \prod_{j \in J}
\mathpzc{Aut}(\cV^{j})$. However the group fixing the boundary stratum
of~$\sigma$ is $A:= \mathpzc{Aut}(\cD, \cV)$ and its natural map~$\varphi: A \to G$ is not surjective. In fact each $\mathpzc{Aut}(\cV^{j})$
has an exact sequence~\eqref{eq:MCG1structure}, and the braid groups
$B_{n_j+1}$ for each~$j$ as well as $\mathpzc{Aut}_\lift^\circ(\cD/\cV)$ are
in the image of~$\varphi$, but the product of cokernels (each isomorphic
to $\bZ/(n_j+3)\bZ$) is not hit surjectively.
(Apply~\eqref{eq:MCG1structure} to $\mathpzc{Aut}(\cD, \cV)$ and use that
the cokernel is generated by the shift to prove this.)
As a conclusion the equivalence relation generated by autoequivalences
on multi-scale stability conditions is coarser than what is expected
for Theorem~\ref{thm:KrhoBihol} to be a bijection. We thus need an additional
datum. To motivate the following definition, recall that the shift acts
(via the correspondence to framed quadratic differentials) by cyclically
shifting the marked points at each pole. 
\par
\medskip
\paragraph{\textbf{Prong matchings in the abelian case}}
A \emph{prong} at a zero of order $m$ of an abelian differential is a tangent
vector that coincides with one of its $\kappa = m+1$ outgoing horizontal
directions. A \emph{prong} at a pole of order $|m|$ is a tangent vector
that coincides with one of its $\kappa = |m|-1$ incoming horizontal direction.
(For poles this is the same as choosing one of the marked points in~$M_i$
as defined in Section~\ref{sec:wDMS}.) The prongs are labelled cyclically
(by embedding in the plane) in clockwise order in the case of zeros (resp.\
counterclockwise order in the case of poles). Given an enhanced
level graph~$\wh\Gamma$ a \emph{prong-matching} $\wh\wp = (\wh\wp_e)_{e \in E^v}$
is a bijection of the prongs at the upper and lower even of each edge that
reverses the cyclic order. Consequently, there are $\wh{K}_{\wh\Gamma}
= \prod_{e \in E^v} \wh\kappa_e$ different prong matchings for~$\wh\Gamma$.
\par
\medskip
\paragraph{\textbf{The level rotation torus}} The lower level differential
should be projectivised, since only in this way limits are well-defined
(compare with the proof of Theorem~\ref{thm:compact}) and since only
in this way the $\sigma \mapsto (X,\bfz,\Gamma, \bfq)$ will pass to the
equivalence class of~$\sigma$. However just rescaling the lower level by~$\bC^*$ is
no longer well-defined, as this comprises rotation and changes the horizontal
direction that the notion of prong relies on. There is a finite unramified
cover of~$\bC^*$ (of course: still abstractly isomorphic to~$\bC^*$) that
naturally acts by rotation on the differentials~$\bfomega$ and on~$\wp$
simultaneously so that the preimages of~$1 \in \bC^*$ fix the differential
and permute cyclically each~$\wp_e$. 
This algebraic torus (for general~$\sigma$ isomorphic to~$(\bC^*)^L$)
is called \emph{level rotation torus}, see \cite[Section~6]{LMS} for the
full definition. 
\par
\medskip
\paragraph{\textbf{Prong matchings and level rotation torus
in the quadratic case}} Prongs and their
matchings are defined as in the abelian case, noting that a zero order~$m$
of a quadratic differential has $\kappa_e = m+2$ outgoing horizontal directions
(to be counted on a local square root!) and a pole of order~$|m|$ has
$\kappa_e = |m|-2$ incoming horizontal directions. There are 
${K}_{\Gamma} = \prod_{e \in E^v} \kappa_e$ prong matchings.
\par
To understand the action of the level rotation torus, the easiest
way is to pass to the canonical cover and use that a prong-matching
of the quadratic differential induces a prong-matching of an abelian differential.
Now the equivalence relation given by the level rotation torus is
just defined as in the abelian case, restricted to those abelian
multi-scale differentials that actually arise as double covers.
See \cite[Section~7]{CMS} for full details.
\par
\begin{proof}[Proof of the isomorphism~\eqref{eq:isoofmulti}, Part II]
Finally we show how to associate with~$\sigma$ a prong-matching~$\wp$. We
continue with the setting above, in particular $L=1$. Consider
the ray obtained by plumbing  $(it) * \sigma \in \Stab^\circ(\cD_{A_n})$ with
a purely imaginary parameter, i.e., without rotation. The limit
$t \to \infty$ of the Bridgeland-Smith preimages $K^{-1}(it * \sigma)$ is a
multi-scale differential with underlying $(X,\bfz,\Gamma,\bfq)$ as
above, by definition of plumbing and of the topology on $\ol{Q}_n$.
It thus comes with a prong-matching~$\wp$  and we now set
$\ol{K}_n^{-1}(\sigma) = (X,\bfz,\Gamma, \bfq, \wp)$. (We remark that this~$\wp$
is the only choice if we want $\ol{K}_n^{-1}$ to be continuous. Formally,
in the language of \cite[Section~7]{LMS} this is the only prong-matching
so that the comparison diffeomorphisms between the welding of the limiting
stable curve and the nearby plumbed curves is almost turning-number preserving.
Informally, when $|J| \geq 2$ the right choice of~$\wp$ differs from a wrong
choice of~$\wp'$ by rotating say one prong for the subsurface $j=1$ on lower level.
With the wrong~$\wp'$ the turning numbers near the subsurface
$j=1$ do not work out. By rotating the whole lower level using the
$\bC$-action, turning numbers can be fixed for $j=1$, but then since~$|J| \geq 2$
the turning numbers will not work out at some other subsurface of lower level.)
\par
To see that~$\ol{K}_n^{-1}$ is bijective note that both the initial failure
(due to the cokernel of~$\varphi:A \to G$) and the additional datum~$\wp$ capture
the possibility of rotating a lower level component independently of 
other components. We leave the details to the reader.
\par
To show continuity (and well-definedness mod $\pzAut^\circ(\cD^3_{A_n})$)
it is best to first lift the map~$\ol{K}_n^{-1}$ to a map from $\MStab^\circ
(\calD^3_{A_n})$
to the Teichmüller-framed version of~$\ol{Q}_n$, the \emph{augmented Teichmüller
space} in the sense of \cite[Section~7]{LMS}. This is a bordification of
$\FQuad(\sow)$ on which the mapping class group acts. One now needs to check
that this lifted~$\ol{K}_n^{-1}$ is a homeomorphism using the respective
definition of topologies and the equivariance with respect to the
mapping class group and $\pzAut^\circ(\cD^3_{A_n})$-action.
\par
To show compatibility with the complex structure one
needs to recall that the complex structure on~$\ol{Q}_n$ is defined
using plumbing (in the sense of complex geometry). This gives a collection
of periods that defines local coordinates (the perturbed period coordinates
in \cite[Section~9]{LMS}, in fact no modification of the differential is needed
for $A_n$-type since all the residues are zero) and one only needs to check that
they correspond to the coordinates defined in Section~\ref{subsec:complexstr}.
We leave again the details to the reader.
\end{proof}
\par

\subsection{Why taking $\pzAut^\circ(\cD^3_{A_n})$-quotients?
The $A_2$-quiver revisited} \label{sec:A2revisited}
In this subsection we revisit
$\PMStab(\DQQ[3][A_2])$ to show that prior to taking the
$\pzAut^\circ(\cD^3_{A_n})$-quotient is neither a compact space (contrary
to the Thurston-type compactifications in \cite{BDL_Thurston}) nor carries a
complex structure.
\par
In fact $\PMStab(\DQQ[3][A_2])$ coincides with the upper half plane
with cusps $\wt{\bH} = \bH \cup \bP^1_\bQ$ provided with the horoball
topology where a neighborhood basis of $\infty$ consists of the
sets $U_C = \{\tau: \Im(\tau) > C\}$ and a neighborhood basis of $z \in \bQ$
are the images of~$U_C$ under a Möbius transformation mapping~$\infty$ to~$z$. Here $\bP^1_\bQ=\bQ\cup\{\infty\}$ and it is known that $\bP\Stab(\calD^3_{A_2})\simeq \bH$, see e.g., \cite{sutherland}.
\par
To prove this we start with a classification of the boundary strata. In
this case necessarily $L=1$. Since the Grothendieck group of a vanishing subcategory~$\calV$ of $\calD:=\calD^3_{A_2}$ has rank~$1$, all
stability conditions on it are projectively equivalent. 
\par
Next we list the possible $\calV$. Recall that a heart 
suporting a stability condition on the space 
$\stab(\calD)/\sph(\calD)$ can be identified 
with one of the following: the standard heart $\cl H_0=\bra S_1,S_2\ket$ or its shift $\cl H_0[1]$, or $\bra S_1[1], E\ket$, $\bra S_1,S_2[1]\ket$, $\bra S_2, E[1]\ket$, where $S_2\to E\to S_1$ is a short exact sequence. In fact, the vanishing category $\calA_1$ arising from one of the hearts above is generated by one of the indecomposables $S_1,S_2,E$ of $\cl H_0$, those appearing in Figure~\ref{fig:A2rotate}. Therefore in $\stab(\calD)$, 
we associate with any such $\calV$ the
image of a generating simple in $\bP^1(K(\calD)) \cong \bP^1_\bQ$ and call
this map~$c$. We let $\cl H_0$ be the standard
heart of the $A_2$-quiver and let 
$\calV_1 = \bra S_2 \ket$, corresponding to
$c(\calV_1) = \bigl(\begin{smallmatrix} 0 \\ 1 \end{smallmatrix} \bigr)
\in K(\calD)$. This subcategory obviously corresponds to any central charge
with $Z_0(S_1) \in \pm\bH$ and $Z_0(S_2) = 0$. 
\par
We first consider the action of the Seidel-Thomas group $\sph(\calD) \cong
B_3$ on $\cl H_0$ and~$\calV_1$. The element $\tau_2$ stabilizes~$\calV_1$.
The generator of center~$\theta_2$ of~$B_3$ acts by the shift by $[\pm 5]$ and
thus trivially on $\calV_1$. Given that $B_3/\bra\theta_2\ket\simeq \PSL_2(\bZ)$ \cite{farbmarg} and that $\tau_2$ acts as $\bigl(\begin{smallmatrix} 1 & 1 \\ 0& 1
  \end{smallmatrix} \bigr)$ on $K(\calD^3_{A_2})$, the orbits of the action
$\sph(\calD)$ on $\calV_1$ are in bijection (of cosets) with
\be
B_3 / \bra \theta_2, \tau_2 \ket \cong \PSL_2(\bZ) / \bra
\bigl(\begin{smallmatrix} 1 & 1 \\ 0 & 1 \end{smallmatrix} \bigr)\ket 
\cong \bP^1_\bQ 
\ee
\par
The quotient $\Aut(\calD)/\sph(\calD)$ is generated by~$[-1]$, which also
acts trivially on any~$\calV$. To summarize, the orbit $\Aut(\calD) \cdot
\calV_1$ is  in natural bijection with $\bP^1_\bQ$ via the map~$c$. The resulting space $\PMStab(\calD)/\Aut(\calD)$ is the compact orbifold $\widetilde{\H}/\PSL_2(\bZ)$.
\par

\printbibliography

\end{document}